\documentclass[12pt,reqno]{amsart}
\usepackage{mathrsfs}
\usepackage{url}
\usepackage{mathtools}
\usepackage{latexsym,epsfig,amssymb,amsmath,amsthm,color,url}
\usepackage[inline,shortlabels]{enumitem}
\usepackage{hyperref}
\usepackage[foot]{amsaddr}

\allowdisplaybreaks 
 \numberwithin{equation}{section}
\newtheorem{theorem}{Theorem}[section]
\newtheorem{remark}[theorem]{Remark}
\newtheorem{lemma}[theorem]{Lemma}
\newtheorem{corollary}[theorem]{Corollary}

\renewcommand{\theenumi}{\roman{enumi}}

\theoremstyle{definition}
\newtheorem{defn}[theorem]{Definition}

\theoremstyle{remark}
\newtheorem{rem}[theorem]{Remark}

\DeclareMathOperator{\Prob}{\mathbf{P}}
\DeclareMathOperator{\EHR}{EHR}
\DeclareMathOperator{\E}{\mathbf{E}}
\DeclareMathOperator{\maxden}{\rho_{\max}}
\DeclareMathOperator{\poi}{Poisson}

\title[Existential zero-one laws]{Zero-one laws for existential first order sentences of bounded quantifier depth}
\date{}
\author{Moumanti Podder}
\address{Moumanti Podder, \ Department of Mathematics, \ University of Washington, \ C-524 Padelford Hall, West Stevens Way Northeast, Seattle, WA 98105, United States.}
\email{mpodder3@uw.edu}
\author{Maksim Zhukovskii}
\address{Maksim Zhukovskii, Moscow Institute of Physics and Technology, laboratory of advanced combinatorics and network applications, 9 Institutskiy per., Dolgoprodny, Moscow Region, 141701, Russian Federation; 
Adyghe State University, Caucasus mathematical center, ul. Pervomayskaya, 208, Maykop, Republic of Adygea, 385000, Russian Federation; 
The Russian Presidential Academy of National Economy and Public Administration, Prospect Vernadskogo, 84, bldg 2, Moscow, 119571, Russian Federation.}
\email{zhukmax@gmail.com}

\begin{document}
\bibliographystyle{plainnat}

\begin{abstract}
For any fixed positive integer $k$, let $\alpha_{k}$ denote the smallest $\alpha \in (0,1)$ such that the random graph sequence $\left\{G\left(n, n^{-\alpha}\right)\right\}_{n}$ does not satisfy the zero-one law for the set $\mathcal{E}_{k}$ of all existential first order sentences that are of quantifier depth at most $k$. This paper finds upper and lower bounds on $\alpha_{k}$, showing that as $k \rightarrow \infty$, we have $\alpha_{k} = \left(k - 2 - t(k)\right)^{-1}$ for some function $t(k) = \Theta(k^{-2})$. We also establish the precise value of $\alpha_{k}$ when $k = 4$. 
\end{abstract}



\thanks{The first author was partially supported by the grant NSF DMS-1444084. The second author was supported by Grant 16-11-10014 of the Russian Science Foundation.}

\maketitle

\section{Introduction}

\sloppy Given a graph $G$, we denote by $V(G)$ the set of its vertices and by $E(G)$ the set of its edges; when $G$ is clear from the context, we simply write $V$ and $E$ respectively. We set $e(G) = |E(G)|$ and $v(G) = |V(G)|$. For two vertices $x$ and $y$ in $V$, we say that $\{x, y\} \in E$ when $x$ and $y$ is adjacent. When the graph and its edge set are evident, we also alternatively use the notation $x \sim y$. We denote by $\mathbb{N}$ the set of all positive integers and by $\mathbb{N}_{0}$ the set of all non-negative integers.\\

A \emph{sentence} is a formula in mathematical logic that does not have free variables (see [\cite{libkin}, Subsection~2.1] for basic definitions in mathematical logic). The \emph{first order} logic on graphs comprises finite sentences involving the following components:
\begin{itemize}
\item the vertices as propositional variables, denoted in general by lower case letters such as $x, y, z \ldots$;
\item the relation of vertex equality (denoted $x = y$) and the relation of vertex adjacency (denoted $x \sim y$);
\item Boolean connectives $\wedge, \vee, \neg, \implies, \Leftrightarrow$;
\item existential (denoted $\exists$) and universal (denoted $\forall$) quantification, allowed only over vertices.
\end{itemize}
We refer the reader to \cite{libkin, immerman, marker, strange} for further reading on first order logic. Examples of first order sentences include 
$$\exists x [\forall y [\neg x \sim y]],$$
which expresses the property of a graph that it contains an isolated vertex, and 
$$\exists x [\exists y [[x \sim y] \wedge \forall z [x \sim z \implies z = y]]],$$
which expresses the property of a graph that there exists a vertex with degree precisely $1$, i.e.\ precisely $1$ neighbour. The \emph{quantifier depth}, also referred to as the \emph{quantifier rank}, of a first order sentence is defined as the maximum number of nested quantifiers in the sentence; we refer the reader to [\cite{libkin}, Definition 3.8] for the formal definition. We call a first order sentence \emph{existential} if 
\renewcommand{\labelenumi}{\theenumi}
\renewcommand{\theenumi}{\roman{enumi})}
\begin{enumerate*}
\item all its quantifiers are existential,
\item negations are allowed only in front of atomic formulas.
\end{enumerate*}
Neither of the examples above is an existential first order sentence; an example of such a sentence would be 
$$\exists x [\exists y \exists z [[x \sim y] \wedge [x \sim z] \wedge [\neg y = z]]],$$
which expresses the property that there exists a vertex with at least $2$ children. Given a first order sentence $\gamma$ and a graph $G$, the notation $G \models \gamma$ indicates that $\gamma$ is true on $G$. \\

We recall the Erd\H{o}s-R\'{e}nyi random graph model $G(n, p(n))$ -- starting with $n$ vertices that are denoted $1, \ldots, n$, the edge between vertices $i$ and $j$ is added with probability $p(n)$, mutually independently over all pairs $\{i, j\}$. We say that a graph property $A$ holds \emph{asymptotically almost surely} (abbreviated henceforth as a.a.s.)\ on $\{G(n, p(n))\}$ if the probability that $A$ holds on $G(n, p(n))$ approaches $1$ as $n \rightarrow \infty$; similarly, we say that a sentence $\gamma$ is a.a.s.\ true on $\{G(n, p(n))\}$ if $\lim_{n \rightarrow \infty} \Prob\left[G(n, p(n)) \models \gamma\right] = 1$. 

\begin{defn}
Given a set $\mathcal{L}$ of first order sentences, we say that the random graph sequence $\left\{G(n, p(n))\right\}_{n}$, for a given sequence $\{p(n)\}$ of edge probability functions, satisfies the zero-one law for $\mathcal{L}$ if for every sentence $\gamma$ in $\mathcal{L}$, the limit $\lim_{n \rightarrow \infty} \Prob\left[G(n, p(n)) \models \gamma\right]$ exists and equals either $0$ or $1$.
\end{defn}
In \cite{04}, the following well-known theorem was established.
\begin{theorem}\label{irrational}
For $\alpha \in (0,1)$, the random graph sequence $\left\{G(n, n^{-\alpha})\right\}_{n}$ satisfies the zero-one law for first order logic if and only if $\alpha$ is irrational.
\end{theorem}
In \cite{01}, it was shown that when first order sentences of quantifier depth at most $k$, for a given positive integer $k$, are considered, $\left\{G(n, n^{-\alpha})\right\}_{n}$ satisfies the zero-one law for all $\alpha < (k-2)^{-1}$, and it fails to satisfy the zero-one law when $\alpha = (k-2)^{-1}$. This is stated in the following theorem:
\begin{theorem}\label{bounded_qd}
If $0 < \alpha < \frac{1}{k-2}$, then the random graph sequence $\left\{G(n, n^{-\alpha})\right\}_{n}$ satisfies the zero-one law for the set $\mathcal{F}_{k}$ of all first order sentences of quantifier depth of at most $k$; for $\alpha = \frac{1}{k-2}$, the random graph sequence $\left\{G(n, n^{-\alpha})\right\}_{n}$ does not satisfy the zero-one law.
\end{theorem}
In this paper, we consider the next most natural question: what range of $\alpha$ in $(0,1)$ will allow $\left\{G(n, n^{-\alpha})\right\}_{n}$ to satisfy the zero-one law when we consider \emph{existential} first order sentences of bounded quantifier depth? Theorem~\ref{MAIN} gives the statement of the main result of this paper.

\begin{theorem}\label{MAIN}
For any positive integer $k$, let $\mathcal{E}_{k}$ denote the set of all existential first order sentences of quantifier depth at most $k$. Let $\alpha_{k}$ be the minimum $\alpha$ in $(0,1)$ such that $\left\{G\left(n, n^{-\alpha_{k}}\right)\right\}_{n}$ does not satisfy the zero-one law for $\mathcal{E}_{k}$. Then
\begin{equation}
\alpha_{k} = \left\{k - 2 - \Theta\left(\frac{1}{k^{2}}\right)\right\}^{-1} \text{ as } k \rightarrow \infty.
\end{equation}
For $k = 4$, we have $\alpha_{k} = \frac{7}{13}$.
\end{theorem}
Since the set $\mathcal{E}_{k}$ is a fragment of $\mathcal{F}_{k}$, for $\alpha \in (0,1)$, if $\left\{G(n, n^{-\alpha})\right\}_{n}$ satisfies the zero-one law for $\mathcal{F}_{k}$, then it satisfies the zero-one law for $\mathcal{E}_{k}$ as well. Hence from Theorem~\ref{bounded_qd}, it follows that $\alpha_{k} \geq \frac{1}{k-2}$. One may even expect that $\alpha_{k}$ would be significantly larger than $\frac{1}{k-2}$. But Theorem~\ref{MAIN} shows that the order of magnitude, as $k \rightarrow \infty$, of the difference $\alpha_{k}-\frac{1}{k-2}$ is only $O(k^{-4})$, which is surprisingly small. In this sense, $\mathcal{F}_{k}$ is more expressive than $\mathcal{E}_{k}$, but not by much.\\

We give here a brief discussion of existing literature that deals with questions of a similar flavour. In \cite{kaufmann}, it was shown that there exists an existential monadic second order (EMSO) sentence $\psi$ over a vocabulary of $4$ binary relations and $9$ first-order variables such that, if, for each $n \in \mathbb{N}$, we denote by $\mu_{n}(\psi)$ the fraction of finite models with domain $\{0, 1, \ldots, n-1\}$ for which $\psi$ holds, then the asymptotic probability $\mu(\psi) = \displaystyle \lim_{n \rightarrow \infty} \mu_{n}(\psi)$ does not exist, establishing the failure of zero-one laws for EMSO logic. In \cite{lebars_0-1_fails_EMSO}, this was improved by establishing the existence of an EMSO sentence $\psi'$ on undirected graphs, over a vocabulary comprising only $1$ binary relation, which has no asymptotic probability. In \cite{monica}, denoting by $L_{\infty, \omega}^{k}$ the logic comprising arbitrarily many conjunctions and disjunctions but at most $k$ distinct variables, it was shown that the random graph sequence $\left\{G(n, n^{-\alpha})\right\}_{n}$ satisfies the zero-one law for $L_{\infty, \omega}^{k}$ when $k = \lceil 1/\alpha \rceil$, whereas $\left\{G(n, n^{-\alpha})\right\}_{n}$ fails to satisfy convergence laws for $L_{\infty, \omega}^{k+1}$ if $\alpha = 1/m$ for $m > 2$. Finally, in \cite{raza_maksim}, it was shown that for every $k \in \mathbb{N}$ and $\epsilon > 0$, there exists $\alpha \in \left((k-1)^{-1}, (k-1)^{-1}+\epsilon\right)$ such that $\left\{G(n, n^{-\alpha})\right\}_{n}$ does not satisfy the zero-one law for first order sentences with $k$ variables, whereas it does satisfy the zero-one law for this logic for $\alpha \in (0, (k-1)^{-1}]$.\\

In Section~\ref{tools}, we describe several definitions, tools and results from the literature that are to be used in the proof of Theorem~\ref{MAIN}, in Sections~\ref{upper} and \ref{lower} respectively we discuss our derivation of the upper and lower bounds on $\alpha_{k}$ as $k \rightarrow \infty$, and in Section~\ref{k=4}, we discuss the derivation of the exact value of $\alpha_{4}$ (part of the proof that $\alpha_{4} = \frac{7}{13}$ is in Appendix, Section~\ref{Appendix}).  

\section{Tools and results used in our paper}\label{tools}
We start by describing a suitable version of the well-known \emph{Ehrenfeucht-Fra\"{i}sse} games used to analyse existential first order sentences on graphs. For $n \in \mathbb{N}$, let $[n] = \left\{1, \ldots, n\right\}$.
\begin{defn}\label{EHR_existential}
Given two graphs $G$ and $H$ and a positive integer $k$, the \emph{existential Ehrenfeucht game} of $k$ rounds, denoted $\EHR[G, H, k]$, is played by Spoiler and Duplicator as follows. At the very beginning of the game, Spoiler chooses one of the graphs $G$ and $H$. Without loss of generality, assume that he chooses $G$. Each of the $k$ rounds consists the choice of a vertex from $G$ by Spoiler, followed by the choice of a vertex from $H$ by Duplicator. Suppose $x_{i}$ is the vertex selected from $G$ and $y_{i}$ that from $H$ in round $i$, for $i \in [k]$. Duplicator wins if \emph{all} of the following conditions hold: for all $i, j \in [k]$,
\begin{enumerate}
\item $x_{i} = x_{j} \Leftrightarrow y_{i} = y_{j}$;
\item $x_{i} \sim x_{j} \Leftrightarrow y_{i} \sim y_{j}$.
\end{enumerate}
\end{defn}
We define the relation $\sim_{k}$ as follows: given two graphs $G$ and $H$, we say $G \sim_{k} H$ if Duplicator wins $\EHR[G, H, k]$. This is an equivalence relation which partitions the space of all graphs into \emph{finitely} many equivalence classes (see [Section 2.2, \cite{strange}] and [\cite{libkin}, Lemma~3.13]). The following well-known theorem states the connection between existential Ehrenfeucht games and existential first order sentences of bounded quantifier depth (see \cite{eh_1960}, [\cite{libkin}, Theorem~3.9], \cite{immerman} and \cite{marker}).
\begin{theorem}\label{EHR_importance}
Given any two graphs $G$ and $H$ and any $k \in \mathbb{N}$, $G \sim_{k} H$ if and only if, for any existential first order sentence $\gamma$ of quantifier depth at most $k$, either both $G \models \gamma$ and $H \models \gamma$, or both $G \models \neg \gamma$ and $H \models \neg \gamma$, i.e.\ $G$ and $H$ have the same truth value for \emph{all} existential first order sentences of quantifier depth at most $k$.
\end{theorem}

\begin{corollary}
Duplicator wins $\EHR\left[G_{1}, G_{2}, k\right]$ a.a.s.\ for $G_{1} \sim G\left(m, p(m)\right)$ and $G_{2} \sim G\left(n, p(n)\right)$ that are independent of each other, as $m, n \rightarrow \infty$, if and only if $\left\{G\left(n, p(n)\right)\right\}_{n}$ satisfies the zero-one law for existential first order logic of quantifier depth $k$.
\end{corollary}
Let us now switch to some very helpful results describing distributions of small subgraphs inside the random graph.
\begin{defn}
Let $H$ and $G$ be graphs on vertex sets $\left\{x_{1}, \ldots, x_{k}\right\}$ and $\left\{x_{1}, \ldots, x_{\ell}\right\}$ respectively with $H \subset G$. A graph $\widetilde{G}$ on $\left\{\widetilde{x}_{1}, \ldots, \widetilde{x}_{\ell}\right\}$ is a \emph{strict $(G, H)$-extension} of a graph $\widetilde{H}$ on $\left\{\widetilde{x}_{1}, \ldots, \widetilde{x}_{k}\right\}$ if $\widetilde{H} \subset \widetilde{G}$ and $\{x_{i}, x_{j}\} \in E(G) \setminus E(H)$ iff $\left\{\widetilde{x}_{i}, \widetilde{x}_{j}\right\} \in E\left(\widetilde{G}\right) \setminus E\left(\widetilde{H}\right)$ for all $i, j \in [\ell]$.

\end{defn}
Setting $v(G, H) = |V(G) \setminus V(H)|$ and $e(G, H) = |E(G) \setminus E(H)|$, for $\alpha > 0$, we define
\begin{equation}
f_{\alpha}(G, H) = v(G, H) - \alpha e(G, H).
\end{equation}
We define the pair $(G, H)$ to be \emph{$\alpha$-safe} if for every $H \subset S \subseteq G$, we have $f_{\alpha}(S, H) > 0$. 

Let $\widetilde{H} = \left\{\widetilde{x}_{1}, \ldots, \widetilde{x}_{k}\right\}$ be a subset of the vertex set $[n]$ of $G(n, p(n))$. For each subset $W$ of $[n] \setminus \widetilde{H}$ of cardinality $\ell-k$, if we can enumerate the vertices of $W$ as $\widetilde{x}_{k+1}, \ldots, \widetilde{x}_{\ell}$ such that the induced subgraph on $W \cup \widetilde{H}$ is a strict $(G, H)$-extension of the induced subgraph on $\widetilde{H}$, we set $\mathbf{1}_{W} = 1$; otherwise $\mathbf{1}_{W} = 0$. We define the random variable
\begin{equation}
N_{(G, H)}\left(\widetilde{x}_{1}, \ldots, \widetilde{x}_{k}\right) = \sum_{W \subset \mathscr{V}_{n} \setminus \widetilde{H}, |W| = \ell-k} \mathbf{1}_{W}.
\end{equation}
\begin{theorem}\label{safe_extension_copies}
[See \cite{02} and \cite{03}] Suppose we are given graphs $G$ and $H$ and $\alpha > 0$ such that the pair $(G, H)$ is $\alpha$-safe. Then for any $\epsilon > 0$, 
\begin{multline}
\lim_{n \rightarrow \infty} \Prob\Big[\left|N_{(G, H)}\left(\widetilde{x}_{1}, \ldots, \widetilde{x}_{k}\right) - \E\left[N_{(G, H)}\left(\widetilde{x}_{1}, \ldots, \widetilde{x}_{k}\right)\right]\right| \leq \epsilon \E\left[N_{(G, H)}\left(\widetilde{x}_{1}, \ldots, \widetilde{x}_{k}\right)\right] \\ \forall \left\{\widetilde{x}_{1}, \ldots, \widetilde{x}_{k}\right\} \subset \mathscr{V}_{n} \text{ in } G(n, n^{-\alpha})\Big] = 1.
\end{multline}
Moreover, we have, for all $\widetilde{x}_{1}, \ldots, \widetilde{x}_{k} \in \mathscr{V}_{n}$, 
\begin{equation}
\E\left[N_{(G, H)}\left(\widetilde{x}_{1}, \ldots, \widetilde{x}_{k}\right)\right] \sim \Theta\left(n^{f_{\alpha}(G, H)}\right).
\end{equation}

\end{theorem}

\begin{defn}[Alice's restaurant property]\label{r_ext}
For $r \in \mathbb{N}$, graph $G$, $a, b \in \mathbb{N}_{0}$ with $a+b \leq r$, distinct vertices $x_{1}, \ldots, x_{a}, y_{1}, \ldots, y_{b}$, we call a vertex $v$, distinct from all $x_{i}$'s and $y_{j}$'s, an \emph{$(a,b)$-witness with respect to $\left(\left\{x_{1}, \ldots, x_{a}\right\}, \left\{y_{1}, \ldots, y_{b}\right\}\right)$} if $v \sim x_{i}$ for all $i \in [a]$ and $v \nsim y_{j}$ for all $j \in [b]$. A graph $G$ satisfies the \emph{full level-$r$ extension property}, sometimes referred to as \emph{Alice's restaurant property}, if for \emph{every} $a, b \in \mathbb{N}_{0}$ with $a+b \leq r$ and \emph{every} distinct $x_{1}, \ldots, x_{a}, y_{1}, \ldots, y_{b}$, there exists an $(a,b)$-witness with respect to $\left(\left\{x_{1}, \ldots, x_{a}\right\}, \left\{y_{1}, \ldots, y_{b}\right\}\right)$ in $G$.
\end{defn}

We state here a useful lemma that shows that given $r \in \mathbb{N}$, the random graph sequence $\left\{G(n, n^{-\alpha})\right\}_{n}$ a.a.s.\ has the full level-$r$ extension property for all sufficiently small $\alpha$. See [\cite{spencer_spectra}, Theorem 1.7] for a proof of this fact. 

\begin{lemma}\label{r_ext_suff_cond}
For any positive integer $r$, the random graph sequence $\left\{G(n, n^{-\alpha})\right\}_{n}$ a.a.s.\ has the full level-$r$ extension property whenever $\alpha < \frac{1}{r}$.
\end{lemma}
Thus for $\alpha < \frac{1}{r}$, Duplicator a.a.s.\ wins $\EHR\left[G_{1}, G_{2}, r+1\right]$ for $G_{1} \sim G\left(m, p(m)\right)$ and $G_{2} \sim G\left(n, p(n)\right)$, independent of each other, as $m, n \rightarrow \infty$.\\

To prove that, for given $k \in \mathbb{N}$ and suitable $\alpha_{k}$, $\left\{G\left(n, n^{-\alpha_{k}}\right)\right\}_{n}$ does not satisfy the zero-one law for existential first order sentences of quantifier depth $k$, we come up with a sentence of quantifier depth $k$ that implies the existence of a small, suitable (in some sense) fixed graph as an induced subgraph in $G\left(n, n^{-\alpha_{k}}\right)$. For this, we require some tools from random graph theory describing the asymptotic behaviour of the number of copies of a given finite graph as an induced subgraph in $G\left(n, n^{-\alpha_{k}}\right)$ as $n \rightarrow \infty$. This motivates the quantities defined below.

For a finite graph $G$ with $|V(G)| = v$ and $|E(G)| = e$, we call $\rho(G) = \frac{e}{v}$ the \emph{density} of the graph. We define the \emph{maximal density} of a finite $G$, denoted $\maxden(G)$, as the maximum of $\rho(H)$ over all subgraphs $H$ of $G$. A finite $G$ is called \emph{balanced} if $\rho(H) \leq \rho(G)$ for every subgraph $H$ of $G$; it is called \emph{strictly balanced} if $\rho(H) < \rho(G)$ for every proper subgraph $H$ of $G$; it is called \emph{unbalanced} otherwise. Clearly $\maxden(G) = \rho(G)$ for strictly balanced $G$.  

The next two theorems discuss the asymptotic probability of $G(n, p(n))$ containing a given, finite graph as an induced subgraph -- whereas Theorem~\ref{strictly_balanced_no._copies_thm} discusses the case where the given graph is strictly balanced, Theorem~\ref{unbalanced_no._copies_thm} discusses the more general scenario where it may be balanced or unbalanced.
\begin{theorem}\label{strictly_balanced_no._copies_thm}[See [\cite{b.bas_weirman}, Lemma 1], [\cite{b.bas_threshold}, Theorem 4], and \cite{b.bas_book}]
Let $H_{1}, \ldots, H_{m}$ be given strictly balanced, non-isomorphic finite graphs such that $\rho(H_{1}) = \cdots = \rho(H_{m}) = \rho_{0}$. Let $a_{i}$ be the number of automorphisms of $H_{i}$ for $i \in [m]$. Set $\alpha = \frac{1}{\rho_{0}}$, and for some $c > 0$, consider $\left\{G(n, c n^{-\alpha})\right\}_{n}$. If $N_{H_{i}}^{n}$ denotes the number of induced copies of $H_{i}$ in $G(n, c n^{-\alpha})$ for $n \in \mathbb{N}$ and $i \in [m]$, then the random vector $\left(N^{n}_{H_{1}}, \ldots, N^{n}_{H_{m}}\right)$ converges in distribution to the random vector $\left(\xi_{1}, \ldots, \xi_{m}\right)$ where $\xi_{1}, \ldots, \xi_{m}$ are independent with $\xi_{i} \sim \poi\left(a_{i}^{-1} c^{e_{i}}\right)$, where $e_{i}$ is the number of edges in $H_{i}$.

\end{theorem}

For a given finite graph $G$, we define $\Phi_{G}(n, p(n)) = \min\left\{\E\left[N^{n}_{H}\right]: H \subseteq G, e(H) > 0\right\}$. 
\begin{theorem}\label{unbalanced_no._copies_thm}[See [\cite{janson_luczak_rucinski}, Theorem 3.9]]
For a given finite graph $G$ with at least one edge, for every edge probability sequence $\{p(n)\}_{n \in \mathbb{N}}$, we have
\begin{equation}
\exp\left\{-\frac{1}{1-p(n)} \Phi_{G}(n, p(n))\right\} \leq \Prob\left[G(n, p(n)) \text{ contains induced } G\right] \leq \exp\left\{-\Theta\left(\Phi_{G}(n, p(n))\right)\right\}.
\end{equation}
\end{theorem}
From Theorem~\ref{unbalanced_no._copies_thm} and [\cite{janson_luczak_rucinski}, Section 3.3], we conclude that 
\begin{multline}\label{general_graph_not_0-1}
0 < \liminf_{n \rightarrow \infty} \Prob\left[G\left(n, n^{-1/\maxden(G)}\right) \text{ contains induced } G\right] \\ \leq \limsup_{n \rightarrow \infty} \Prob\left[G\left(n, n^{-1/\maxden(G)}\right) \text{ contains induced } G\right] < 1.
\end{multline}






A function $r(n)$ is called a \emph{threshold function} for some graph property $P$ if a.a.s.\ $\{G(n, p(n))\}$ does not satisfy $P$ whenever $p(n) \ll r(n)$ and a.a.s.\ $\{G(n, p(n))\}$ satisfies $P$ whenever $p(n) \gg r(n)$, or vice versa. 
\begin{theorem}[See \cite{b.bas_book}, \cite{janson_luczak_rucinski} and \cite{02}, Theorem 4.4.2] \label{threshold}
If $G$ is strictly balanced and as before, $N^{n}_{G}$ denotes the number of induced copies of $G$ in $G\left(n, p(n)\right)$, then $n^{-1/\rho(G)}$ is the threshold function for the property $\left\{N^{n}_{G} \geq 1\right\}$.

\end{theorem}

\section{Upper bound}\label{upper}
This section is devoted to finding an upper bound on $\alpha_{k}$ as defined in Theorem~\ref{MAIN}, which involves coming up with:
\renewcommand{\theenumi}{(\roman{enumi})}
\begin{enumerate}
\item a sentence $\varphi_{k} \in \mathcal{E}_{k}$,
\item a finite set $\Sigma_{k}$ of finite graphs such that $\varphi_{k}$ is true on a graph $G$ if and only if $G$ contains $\Gamma$ as an induced subgraph for some $\Gamma \in \Sigma_{k}$.
\end{enumerate}
The sentence $\varphi_{k}$ states that there exists a clique of size $k-3$, comprising vertices $a_{1}, \ldots, a_{k-3}$, henceforth called the \emph{roots}, such that:
\renewcommand{\theenumi}{(\roman{enumi})}
\begin{enumerate}
\item for distinct $i, j \in [k-3]$, there exists a vertex $v_{i,j}$ that is adjacent to $a_{\ell}$ for all $\ell \in [k-3] \setminus \{i, j\}$ but not to $a_{i}$ and $a_{j}$, and we call these $v_{i,j}$ \emph{ground vertices};
\item for each $v_{i,j}$ and $\ell \in [k-3]$, there exists a vertex $v_{i, j, \ell}$ which is adjacent to $v_{i,j}$ and $a_{t}$ for all $t \in [k-3] \setminus \{\ell\}$, but not to $a_{\ell}$, and we call these $v_{i,j,\ell}$ \emph{first-level to $v_{i,j}$};
\item for each $v_{i,j}$, each $v_{i,j,\ell}$, and $m \in [k-3]$, there exists a vertex $v_{i,j,\ell,m}$ which is adjacent to $v_{i,j}$, $v_{i,j,\ell}$ and $a_{t}$ for all $t \in [k-3] \setminus \{m\}$, but not to $a_{m}$, and we call these $v_{i,j,\ell,m}$ \emph{second-level to $v_{i,j}$ and $v_{i,j,\ell}$};
\item for each $v_{i,j}$ and each $v_{i,j,\ell}$, there exists a vertex $w_{i,j,\ell}$ which is adjacent to $v_{i,j}$, $v_{i,j,\ell}$ and $a_{t}$ for all $t \in [k-3]$, and we call these $w_{i,j,\ell}$ \emph{universal to $v_{i,j}$ and $v_{i,j,\ell}$}.

\end{enumerate}
The definition makes it immediate that $\varphi_{k}$ is existential first order with quantifier depth precisely $k$.

Henceforth, a vertex $v_{i,j,\ell}$ that is first-level to ground vertex $v_{i,j}$ is simply referred to as a first-level vertex, a vertex $v_{i,j,\ell,m}$ that is second-level to ground vertex $v_{i,j}$ and first-level vertex $v_{i,j,\ell}$ is referred to as a second-level vertex, and a vertex $w_{i,j,\ell}$ that is universal to $v_{i,j}$ and $v_{i,j,\ell}$ is referred to as a universal vertex (the notations reveal which ground and / or first-level vertices they correspond to).

By definition of $\varphi_{k}$, there exists a \emph{finite} family $\Sigma_{k}$ of finite graphs $\Gamma$ comprising \emph{only} the following:
\renewcommand{\labelenumi}{\theenumi}
\renewcommand{\theenumi}{(\roman{enumi})}
\begin{enumerate*}
\item the $k-3$ roots,
\item the ${k-3 \choose 2}$ ground vertices,
\item for each ground vertex, the vertices first-level to it,
\item for each ground vertex and each corresponding first-level vertex, the vertices second-level to them,
\item for each ground vertex and each corresponding first-level vertex, the vertex universal to them,
\end{enumerate*} 
such that $G \models \varphi_{k}$ if and only if $G$ contains an induced subgraph isomorphic to some $\Gamma$ in $\Sigma_{k}$. 

\begin{theorem}\label{main_upper_bound}
If $\alpha_{0}^{-1} := \min\left\{\maxden(\Gamma): \Gamma \in \Sigma_{k}\right\}$, then $\alpha_{0}^{-1}$ is at least $k - 2 - O\left(\frac{1}{k^{2}}\right)$.
\end{theorem}

The proof comprises a few lemmas. For $\Gamma \in \Sigma_{k}$, we call a first-level vertex $v_{i,j,\ell}$ \emph{unique} if it coincides with no other first-level vertex, and a second-level vertex $v_{i,j,\ell,m}$ \emph{unique} if it coincides with no other vertex in $\Gamma$ (i.e.\ with neither any first-level vertex nor any other second-level vertex). Let $\Gamma$ contain $a$ unique first-level vertices and $b$ unique second-level vertices $v_{i,j,\ell,m}$ such that the corresponding first-level vertex $v_{i,j,\ell}$ is also unique. We consider the subgraph $H$ of $\Gamma$ induced on:
\renewcommand{\labelenumi}{\theenumi}
\renewcommand{\theenumi}{(\roman{enumi})}
\begin{enumerate*}
\item all roots,
\item all ground vertices,
\item all unique first-level vertices,
\item all unique second-level vertices whose corresponding first-level vertices are also unique,
\item all universal vertices $w_{i,j,\ell}$ whose corresponding first-level vertices $v_{i,j,\ell}$ are unique, and let the number of distinct such vertices be $\mu$. 
\end{enumerate*}

In $H$, there are 
\renewcommand{\labelenumi}{\theenumi}
\renewcommand{\theenumi}{(\roman{enumi})}
\begin{enumerate*}
\item ${k-3 \choose 2}$ edges with both end-points roots, 
\item $(k-5){k-3 \choose 2}$ edges with one end-point a root and another a ground vertex, 
\item $a(k-3)$ edges with one end-point a unique first-level vertex and the other a root or a ground vertex, 
\item $b(k-2)$ edges with one end-point a unique second-level vertex and the other a root or a ground vertex or a unique first-level vertex. 
\end{enumerate*}
For each universal $w_{i,j,\ell}$ with $v_{i,j,\ell}$ unique, $w_{i,j,\ell}$ is adjacent to all roots and the corresponding ground vertex $v_{i,j}$, thus contributing $\mu(k-2)$ many edges, whereas each unique first-level vertex will have an edge with one of these universal vertices, contributing additional $a$ edges. Thus
\begin{align}\label{density_H_lower_bound_1}
\rho(H) &\geq \frac{{k-3 \choose 2} + (k-5){k-3 \choose 2} + a(k-3) + b(k-2) + \mu(k-2) + a}{k-3 + {k-3 \choose 2} + a + b + \mu} \nonumber\\
&= \frac{\frac{(k-3)(k-4)^{2}}{2} + (a+b+\mu)(k-2)}{\frac{(k-3)(k-2)}{2} + a+b+\mu}.
\end{align}
If $b \geq \frac{k^{4}}{4}$, then $\rho(H)$ is at least $k-2-O(k^{-2})$. 


\begin{lemma}\label{upper_bound_lem_1}
For $\epsilon \in \left(0, \frac{1}{2}\right)$, if $a < \left(\frac{1}{2} - \epsilon\right)k^{3}$, then $\maxden(\Gamma) > k-2$ for $k$ sufficiently large.
\end{lemma}

\begin{proof}
Let $\lambda$ be the total number of distinct first-level vertices in $\Gamma$. Recall that $\Gamma$ contains $a$ unique first-level vertices. If a first-level vertex is not unique, it coincides with at least one other first-level vertex. Thus the remaining $(k-3) {k-3 \choose 2} - a$ non-unique first-level vertices contribute at most $\frac{1}{2}\left\{(k-3){k-3 \choose 2} - a\right\}$ distinct vertices. Therefore
\begin{equation}
\lambda \leq a + \frac{1}{2}\left\{(k-3) {k-3 \choose 2} - a\right\} < \left(\frac{1}{2} - \frac{\epsilon}{2}\right) k^{3} (1+o(1)).
\end{equation}
Consider the subgraph $H$ of $\Gamma$ comprising:
\renewcommand{\labelenumi}{\theenumi}
\renewcommand{\theenumi}{(\roman{enumi})}
\begin{enumerate*}
\item all roots,
\item all ground vertices, 
\item all first-level vertices, 
\item all universal vertices (let there be $\mu$ distinct such vertices).
\end{enumerate*}

There are ${k-3 \choose 2}$ edges in $H$ with both end-points roots and $(k-5){k-3 \choose 2}$ edges with one end-point a root and the other a ground vertex. Each of the distinct first-level vertices is adjacent to $k-4$ roots, hence there are $\lambda(k-4)$ edges in $H$ with one end-point a root and another a first-level vertex. When a first-level vertex $v$ is not unique, there exists some $\ell$ and some subset $S$ of $\left\{\{i, j\}: i \neq j, i, j \in [k-3]\right\}$ with $|S| \geq 2$ such that $v = v_{i,j,\ell}$ for all $\{i, j\} \in S$. The edges in $H$ whose one end-point is $v$ and the other a ground vertex, are given by $\left\{v_{i,j}, v\right\} = \left\{v_{i,j}, v_{i,j,\ell}\right\}$ for all $\{i, j\} \in S$, and they are distinct. Thus the number of edges in $H$ whose one end-point is a first-level vertex and the other a ground vertex is at least $(k-3) {k-3 \choose 2}$. Each distinct universal vertex is adjacent to $k-3$ roots, and to at least one ground vertex (for example, if universal vertices $w_{i,j,\ell}$ and $w_{i,j,\ell'}$, for distinct $\ell$ and $\ell'$, coincide, then edges $\left\{w_{i,j,\ell}, v_{i,j}\right\}$ and $\left\{w_{i,j,\ell'}, v_{i,j}\right\}$ will coincide too, and will be counted once). Thus the number of edges in $H$ whose one end-point is a universal vertex and another a root or a ground vertex is at least $\mu(k-2)$. Finally, each first-level vertex is adjacent to at least one universal vertex, even when the first-level vertex is non-unique and the universal vertex coincides with another universal vertex (for example, the first-level vertices $v_{i_{1},j_{1},\ell}$ and $v_{i_{2},j_{2},\ell}$ may coincide, and the universal vertices $w_{i_{1},j_{1},\ell}$ and $w_{i_{2},j_{2},\ell}$ may coincide, so that edges $\left\{v_{i_{1},j_{1},\ell}, w_{i_{1},j_{1},\ell}\right\}$ and $\left\{v_{i_{2},j_{2},\ell}, w_{i_{2},j_{2},\ell}\right\}$ coincide and are counted once). Thus the number of edges in $H$ whose one end-point is a universal vertex and the other a first-level vertex is at least $\lambda$. Thus
\begin{align}
\rho(H) &\geq \frac{{k-3 \choose 2} + (k-5){k-3 \choose 2} + \lambda(k-4) + {k-3 \choose 2}(k-3) + \mu(k-2) + \lambda}{(k-3) + {k-3 \choose 2} + \lambda + \mu} \nonumber\\
&= k-2 + \frac{\frac{1}{2}(k-3)^{2}(k-8) - \lambda}{\frac{1}{2}(k-3)(k-2) + \lambda + \mu} \nonumber\\
& \geq k-2 + \frac{\frac{1}{2}(k-3)^{2}(k-8) - \left(\frac{1}{2} - \frac{\epsilon}{2}\right)k^{3}}{\frac{1}{2}(k-3)(k-2) + \lambda + \mu} > k-2
\end{align}
for all $k$ large enough. Consequently, $\maxden(\Gamma) \geq \rho(H) > k-2$. 

\end{proof}

By our definition, a unique first-level vertex is forbidden from coinciding with another first-level vertex, but may coincide with a second-level vertex. If $v_{i,j,\ell}$ is a unique first-level vertex that coincides with the second-level vertex $v_{i_{1},j_{1},\ell_{1},\ell}$ for some $\{i_{1},j_{1}\} \neq \{i,j\}$, we call the edge $\left\{v_{i,j,\ell}, v_{i_{1},j_{1}}\right\}$ a \emph{skewed} edge. A first-level vertex can coincide with either another first-level vertex or a second-level vertex. Suppose two first-level vertices $v_{i_{1},j_{1},\ell_{1}}$ and $v_{i_{2},j_{2},\ell_{2}}$ are adjacent. If there exists no $m \in [k-3]$ such that $v_{i_{2},j_{2},\ell_{2}} = v_{i_{1},j_{1},\ell_{1},m}$ nor $m' \in [k-3]$ such that $v_{i_{1},j_{1},\ell_{1}} = v_{i_{2},j_{2},\ell_{2},m'}$, then the edge $\left\{v_{i_{1},j_{1},\ell_{1}}, v_{i_{2},j_{2},\ell_{2}}\right\}$ only serves to unnecessarily increase the density. Thus we may assume that two first-level vertices are adjacent only if at least one of them coincides with a second-level vertex. We call an edge with both end-points first-level vertices a \emph{first-level} edge.


\begin{lemma}\label{upper_bound_lem_2}
Suppose the number of skewed edges in $\Gamma$ is $g$ and that of first-level edges is $h$. Suppose $g + h \geq 2k^{2}$, then $\maxden(\Gamma) > k-2$ for all $k$ sufficiently large.
\end{lemma}

\begin{proof}
Consider the subgraph $H$ of $\Gamma$ comprising:
\renewcommand{\labelenumi}{\theenumi}
\renewcommand{\theenumi}{(\roman{enumi})}
\begin{enumerate*}
\item all roots,
\item all ground vertices, 
\item all first-level vertices (let the number of distinct first-level vertices be $\lambda$, as in Lemma~\ref{upper_bound_lem_1}),
\item all universal vertices (let the number of distinct universal vertices be $\mu$).
\end{enumerate*}
As in Lemma~\ref{upper_bound_lem_1}, we argue that
\begin{align}
\rho(H) & \geq \frac{{k-3 \choose 2} + (k-5){k-3 \choose 2} + \lambda(k-4) + (k-3){k-3 \choose 2} + g + \ell + \mu(k-2) + \lambda}{(k-3) + {k-3 \choose 2} + \lambda + \mu} \nonumber\\
&= k-2 + \frac{\frac{1}{2}(k-3)^{2}(k-8) + (g+\ell) - \lambda}{\frac{1}{2}(k-3)(k-2) + \lambda + \mu} \nonumber\\
&\geq k-2 + \frac{\frac{1}{2}(k-3)^{2}(k-8) + 2k^{2} - (k-3){k-3 \choose 2}}{\frac{1}{2}(k-3)(k-2) + \lambda + \mu} \nonumber\\
&= k-2 + \frac{2k^{2} - 2(k-3)^{2}}{\frac{1}{2}(k-3)(k-2) + \lambda + \mu} > k-2. 
\end{align}


\end{proof}

Lemmas~\ref{upper_bound_lem_1} and \ref{upper_bound_lem_2} show that we only need to consider the scenario where $a > \left(\frac{1}{2} - \epsilon\right)k^{3}$ for every $\epsilon \in (0, 1/2)$ and $\ell + g < 2k^{2}$. As $k \rightarrow \infty$, this implies that there are at least $\left(\frac{1}{2} - \epsilon\right)k^{3}$ unique first-level vertices which are non-adjacent to all other first-level vertices and all ground vertices other than the corresponding ground vertex, i.e.\ such a unique first-level vertex $v_{i,j,\ell}$ is adjacent to no ground vertex $v_{i',j'}$ with $\{i', j'\} \neq \{i, j\}$. We call such a unique first-level vertex \emph{pure}. No second-level vertex can coincide with a pure first-level vertex.

\begin{lemma}\label{upper_bound_lem_3}
For any $\epsilon \in (0, 1/8)$, if $\Gamma$ contains $\widetilde{a} \geq \left(\frac{1}{2} - \epsilon\right)k^{3}$ pure first-level vertices and $b < \frac{k^{4}}{4}$, then $\maxden(\Gamma) > k-2$ for all $k$ sufficiently large.
\end{lemma}

\begin{proof}
Let $\widetilde{\lambda}$ be the number of second-level vertices in $\Gamma$ with the corresponding first-level vertices pure. Recall that the number of unique second-level vertices in $\Gamma$ with the corresponding first-level vertices also unique is $b$, and the pure first-level vertices form a subset of the unique first-level vertices. If $v_{i,j,\ell,m}$ is a non-unique second-level vertex with the corresponding first-level vertex $v_{i,j,\ell}$ pure, then $v_{i,j,\ell,m}$ coincides with another second-level vertex.



There are at most $b$ unique second-level vertices $v_{i,j,\ell,m}$ with $v_{i,j,\ell}$ pure, and $(k-3)^{2} {k-3 \choose 2} - b$, not necessarily distinct, second-level vertices that are either not unique, or unique but the corresponding first-level vertices are not unique. Thus there are at most $(k-3)^{2} {k-3 \choose 2} - b$ second-level vertices that coincide with some other second-level vertices, and hence at most $\frac{1}{2}\left\{(k-3)^{2} {k-3 \choose 2} - b\right\}$ distinct second-level vertices each of which is non-unique and the corresponding first-level vertex is pure. This gives us 
\begin{align}
\widetilde{\lambda} \leq b + \frac{1}{2}\left\{(k-3)^{2} {k-3 \choose 2} - b\right\} < \frac{3k^{4}}{8},
\end{align}
since $b < \frac{k^{4}}{4}$. Suppose $H$ is the subgraph of $\Gamma$ comprising
\renewcommand{\labelenumi}{\theenumi}
\renewcommand{\theenumi}{(\roman{enumi})}
\begin{enumerate*}
\item all roots,
\item all ground vertices, 
\item all pure first-level vertices,
\item all second-level vertices whose corresponding first-level vertices are pure,
\item all universal vertices whose corresponding first-level vertices are pure, and let the number of distinct such vertices be $\mu$.
\end{enumerate*}

There are ${k-3 \choose 2}$ edges in $H$ with both end-points roots and $(k-5){k-3 \choose 2}$ edges with one end-point a ground vertex and another a root. Each pure first-level vertex is adjacent to $k-4$ roots and $1$ ground vertex -- thus there are $\widetilde{a}(k-3)$ edges in $H$ with one end-point a pure first-level vertex and another a root or a ground vertex. The $k-3$ second-level vertices corresponding to a single pure first-level vertex are distinct from each other. Hence there are $\widetilde{a}(k-3)$ edges with one end-point a second-level vertex and the other a pure first-level vertex. Each second-level vertex in $H$ is adjacent to $k-4$ roots and at least one ground vertex. Thus there are at least $\widetilde{\lambda}(k-3)$ edges in $H$ with one end-point a second-level vertex and the other a root or a ground vertex. Each universal vertex is adjacent to $k-3$ roots and at least one ground vertex, hence there are at least $\mu(k-2)$ edges in $H$ with one end-point a universal vertex and the other a root or a ground vertex. Each pure first-level vertex is adjacent to a single universal vertex, hence there are $\widetilde{a}$ edges in $H$ with one end-point a universal vertex and the other a pure first-level vertex. Thus
\begin{align}
\rho(H) & \geq \frac{{k-3 \choose 2} + (k-5){k-3 \choose 2} + \widetilde{a}(k-3) + \widetilde{a}(k-3) + \widetilde{\lambda}(k-3) + \mu(k-2) + \widetilde{a}}{(k-3) + {k-3 \choose 2} + \widetilde{a} + \widetilde{\lambda} + \mu} \nonumber\\
&= k-2 + \frac{\widetilde{a}(k-3) - \widetilde{\lambda} - 2(k-3)^{2}}{(k-3) + {k-3 \choose 2} + \widetilde{a} + \widetilde{\lambda} + \mu} \nonumber\\
&> k-2 + \frac{\left(\frac{1}{2} - \epsilon\right)k^{3}(k-3) - \frac{3k^{4}}{8} - 2(k-3)^{2}}{(k-3) + {k-3 \choose 2} + \widetilde{a} + \widetilde{\lambda} + \mu},
\end{align}
hence as long as $\epsilon \in (0, 1/8)$, we have $\maxden(\Gamma) \geq \rho(H) > k-2$ for all $k$ large enough. 


\end{proof}

Combining everything, we have the proof of Theorem~\ref{main_upper_bound}. Let $\Sigma_{k}^{0} = \left\{S_{1}, \ldots, S_{m}\right\}$, for some $m \in \mathbb{N}$, denote the subset of $\Sigma_{k}$ comprising only those graphs whose maximal density equals $\alpha_{0}^{-1}$ as defined in Theorem~\ref{main_upper_bound}. Let $H_{i}$ be a strictly balanced subgraph of $S_{i}$ with $\rho(H_{i}) = \maxden(S_{i}) = \alpha_{0}^{-1}$ for all $i \in [m]$. As $H_{1}, \ldots, H_{m}$ are non-isomorphic and strictly balanced, from Theorem~\ref{strictly_balanced_no._copies_thm} we have 
\begin{align}\label{not_0-1_eq_1}
& \Prob\left[\bigcup_{i=1}^{m} \left\{G\left(n, n^{-\alpha_{0}}\right) \text{ contains induced } H_{i}\right\}\right] = 1 - \Prob\left[N_{H_{i}}^{n} = 0 \forall i \in [m]\right] \rightarrow 1 - \prod_{i=1}^{m} e^{-\lambda_{i}},
\end{align}
where $\lambda_{i} = a_{i}^{-1}$, with $a_{i}$ the number of automorphisms of $H_{i}$, for $i \in [m]$. For any $S \in \Sigma_{k} \setminus \Sigma_{k}^{0}$, if $H$ is a strictly balanced subgraph of $S$ with $\rho(H) = \maxden(S) > \alpha_{0}^{-1}$, then by Theorem~\ref{threshold}, a.a.s.\ $G\left(n, n^{-\alpha_{0}}\right)$ contains no induced copy of $H$ and hence no induced copy of $S$. Thus
\begin{align}\label{not_0-1_eq_2}
& \Prob\left[G\left(n, n^{-\alpha_{0}}\right) \models \varphi_{k}\right] = \Prob\left[\bigcup_{\Gamma \in \Sigma_{k}} \left\{G\left(n, n^{-\alpha_{0}}\right) \text{ contains induced } \Gamma\right\}\right] \nonumber\\
\leq & \Prob\left[\bigcup_{\Gamma \in \Sigma_{k}^{0}} \left\{G\left(n, n^{-\alpha_{0}}\right) \text{ contains induced } \Gamma\right\}\right] + \sum_{\Gamma \in \Sigma_{k} \setminus \Sigma_{k}^{0}} \Prob\left[G\left(n, n^{-\alpha_{0}}\right) \text{ contains induced } \Gamma\right] \nonumber\\
=& \Prob\left[\bigcup_{i=1}^{m} \left\{G\left(n, n^{-\alpha_{0}}\right) \text{ contains induced } S_{i}\right\}\right] + o(1).
\end{align}
Since $H_{i}$ is a subgraph of $S_{i}$, hence the event $\left\{G\left(n, n^{-\alpha_{0}}\right) \text{ contains induced } S_{i}\right\}$ implies that the event $\left\{G\left(n, n^{-\alpha_{0}}\right) \text{ contains induced } H_{i}\right\}$ holds. Hence we can write 
\begin{align}\label{not_0-1_eq_3}
\Prob\left[\bigcup_{i=1}^{m} \left\{G\left(n, n^{-\alpha_{0}}\right) \text{ contains induced } S_{i}\right\}\right] \leq \Prob\left[\bigcup_{i=1}^{m} \left\{G\left(n, n^{-\alpha_{0}}\right) \text{ contains induced } H_{i}\right\}\right],
\end{align}
so that, combining \eqref{not_0-1_eq_1} and \eqref{not_0-1_eq_3}, we conclude that
\begin{align}
\limsup_{n \rightarrow \infty} \Prob\left[\bigcup_{i=1}^{m} \left\{G\left(n, n^{-\alpha_{0}}\right) \text{ contains induced } S_{i}\right\}\right] \leq 1 - \prod_{i=1}^{m} e^{-\lambda_{i}}.
\end{align}
From Theorem~\ref{unbalanced_no._copies_thm}, noting that $\maxden(S_{1}) = \alpha_{0}^{-1}$, we get
\begin{multline}
\liminf_{n \rightarrow \infty} \Prob\left[\bigcup_{i=1}^{m} \left\{G\left(n, n^{-\alpha_{0}}\right) \text{ contains induced } S_{i}\right\}\right] \\ \geq \liminf_{n \rightarrow \infty} \Prob\left[G\left(n, n^{-\alpha_{0}}\right) \text{ contains induced } S_{1}\right] > 0,
\end{multline}
as desired.

\section{Lower bound}\label{lower}
To get a lower bound on $\alpha_{k}$ as defined in Theorem~\ref{MAIN}, we consider $\EHR\left[G_{1}, G_{2}, k\right]$ where $G_{1} \sim G(m, m^{-\alpha})$ and $G_{2} \sim G(n, n^{-\alpha})$ are generated independently, for $\alpha = \left(k - 2 - t(k)\right)^{-1}$ with a suitable $t(k) = \Theta(k^{-2})$. Assume without loss of generality that Spoiler plays on $G_{1}$, the vertices chosen from $G_{1}$ are denoted $x_{i}$ and those from $G_{2}$ denoted $y_{i}$, $i \in [k]$. From Lemma~\ref{r_ext_suff_cond}, when $\alpha < (k-3)^{-1}$, the full level-$(k-3)$ extension property holds a.a.s.\ in $G(n, n^{-\alpha})$, which Duplicator uses for the first $k-4$ rounds to respond to Spoiler (she can use this property up to round $k-2$, but doing so may result in her losing). For $m \in \mathbb{N}$, we define 
\begin{equation}\label{adj_notation}
\mathfrak{E}_{m} = \left\{\vec{e} = (e_{1}, \ldots, e_{m}): e_{i} \in \{0,1\}\right\},
\end{equation}
and denote $|\vec{e}| = \sum_{i=1}^{m} e_{i}$ for any $\vec{e} \in \mathfrak{E}_{m}$. In round $k-3$, Spoiler chooses $x_{k-3}$ such that $x_{k-3} \sim x_{i}$ iff $e_{i} = 1$ for $i \in [k-4]$, for some $\vec{e} \in \mathfrak{E}_{k-4}$. 

Suppose $H$ is a graph on vertices $a_{1}, \ldots, a_{k-4}$, and let a vertex $a_{k-3}$ be adjacent to $a_{i}$ iff $e_{i} = 1$ for all $i \in [k-4]$, for the $\vec{e}$ mentioned above. Consider the graph $G$, with $H \subset G$, where $G \setminus H$ comprises the following vertices, edges and non-edges:
\renewcommand{\theenumi}{(\roman{enumi})}
\begin{enumerate}
\item $a_{k-2}$, adjacent to all of $a_{1}, \ldots, a_{k-3}$;
\item $a_{k-1}$, adjacent to all of $a_{1}, \ldots, a_{k-2}$;
\item $a_{k-1,j}$, adjacent to $a_{\ell}$ for all $\ell \in [k-2] \setminus \{j\}$ and non-adjacent to $a_{j}$ for all $j \in [k-2]$;
\item $a_{k}$, adjacent to all of $a_{1}, \ldots, a_{k-2}, a_{k-1}$;
\item $a_{k}^{i}$, adjacent to $a_{\ell}$ for all $\ell \in [k-1] \setminus \{i\}$ and non-adjacent to $a_{i}$, for all $i \in [k-2]$;
\item $a_{k}^{k-1,j}$, adjacent to $a_{1}, \ldots, a_{k-2}, a_{k-1,j}$ for all $j \in [k-2]$;
\item $a_{k,i}^{k-1,j}$, adjacent to $a_{\ell}$ for $\ell \in [k-2] \setminus \{i\}$ and to $a_{k-1,j}$, and non-adjacent to $a_{i}$ for $i, j \in [k-2]$.
\end{enumerate}

\begin{lemma}\label{lower_bound_lem_1}
The pair $(G, H)$ is $\alpha$-safe for some $\alpha = \left(k-2-t(k)\right)^{-1}$ with $t(k) = \Theta(k^{-2})$.
\end{lemma}
Note that, in the enumeration of the vertices above, there are $4$ clear layers: the $(k-3)$-rd layer comprises $a_{k-3}$, the $(k-2)$-nd layer comprises $a_{k-2}$, the $(k-1)$-st layer comprises $a_{k-1}$ and $a_{k-1,j}$, $j \in [k-2]$, and the $k$-th layer the rest. In the proof, we assume that the vertices of any graph $S$ with $H \subset S \subseteq G$ are ordered such that those in a layer of smaller index appear earlier. For any vertex $u \in S$, we say that $u$ \emph{brings $e$ edges} if there are exactly $e$ edges between $u$ and the vertices in $S$ that appear \emph{before} $u$ in the above order.

\begin{proof}
For $S$ with $H \subset S \subseteq G$, suppose 
\begin{itemize}
\item $\mathcal{A}_{0} = S \cap \left\{a_{k}, a_{k}^{k-1,j}, j \in [k-2]\right\}$ and $\gamma_{0} = \left|\mathcal{A}_{0}\right|$, and each vertex in this set brings at most $k-1$ edges;
\item $\mathcal{A} = \left\{a_{k-1}, a_{k-1,j}: j \in [k-2]\right\} \cap S$ and $\gamma = |\mathcal{A}|$, at most one vertex from this set brings $k-2$ edges and every other vertex brings at most $k-3$ edges;
\item $\beta = \left|S \cap \left\{a_{k}^{i}, a_{k,i}^{k-1,j}: i, j \in [k-2]\right\}\right|$, and each vertex in this set brings at most $k-2$ edges;
\item $c = \left|S \cap \left\{a_{k-3}, a_{k-2}\right\}\right|$, and each vertex brings at most $k-3$ edges.
\end{itemize}
At most $\gamma$ vertices out of $\mathcal{A}_{0}$ bring $k-1$ edges each, and each of the rest at most $k-2$ edges. Let $\delta = \gamma_{0} - \gamma$ if $\gamma_{0} \geq \gamma$, and $\delta = 0$ otherwise. When $c = 0$, each vertex in $S$ brings at most $k-3$ edges, hence 
\begin{equation}\label{each_degree_leq_k-3}
\frac{e(S, H)}{v(S, H)} \leq k-3.
\end{equation}
When $c = 1$, the single vertex in the set $S \cap \left\{a_{k-3}, a_{k-2}\right\}$ brings $k-4$ edges. Thus
\begin{align}\label{density_c=1_lem_1}
\frac{e(S, H)}{v(S, H)} & \leq \frac{(\gamma_{0} - \delta)(k-1) + \delta(k-2) + (\gamma-1)(k-3) + (k-2) + \beta(k-2) + (k-4)}{\gamma_{0} + \gamma + \beta + 1} \nonumber\\
&= \frac{(\gamma_{0} - \delta)(k-1) + \delta(k-2) + \gamma(k-3) + 1 + \beta(k-2) + (k-4)}{(\gamma_{0} - \delta) + \gamma + \delta + \beta + 1},
\end{align}
and as $\delta(k-2) + \gamma(k-3) + 1 + \beta(k-2) + (k-4) < (k-1)\left\{\gamma + \delta + \beta + 1\right\}$, the fraction in \eqref{density_c=1_lem_1} is strictly increasing in $\gamma_{0} - \delta$. As $\gamma_{0} - \delta \leq \gamma$, hence  
\begin{align}
\frac{e(S, H)}{v(S, H)} & \leq \frac{\gamma(k-1) + \delta(k-2) + \gamma(k-3) + 1 + \beta(k-2) + (k-4)}{\gamma + \gamma + \delta + \beta + 1} \nonumber\\
&= \frac{\gamma(2k-4) + \delta(k-2) + \beta(k-2) + (k-3)}{2 \gamma + \delta + \beta + 1} \nonumber\\
&= k - 2 - \frac{1}{2 \gamma + \delta + \beta + 1} \leq k - 2 - \Theta\left(\frac{1}{k^{2}}\right).
\end{align}
When $c = 2$, we have
\begin{align}\label{density_c=2_lem_1}
\frac{e(S, H)}{v(S, H)} & \leq \frac{(\gamma_{0} - \delta)(k-1) + \delta(k-2) + (\gamma-1)(k-3) + (k-2) + \beta(k-2) + 2(k-3)}{\gamma_{0} + \gamma + \beta + 2} \nonumber\\
&= \frac{(\gamma_{0} - \delta)(k-1) + \delta(k-2) + \gamma(k-3) + 1 + \beta(k-2) + 2(k-3)}{(\gamma_{0} - \delta) + \gamma + \delta + \beta + 2}
\end{align} 
so that, since $\delta(k-2) + \gamma(k-3) + 1 + \beta(k-2) + 2(k-3) < (k-1)\left\{\gamma + \delta + \beta + 2\right\}$, the fraction in \eqref{density_c=2_lem_1} is strictly increasing in $\gamma_{0} - \delta$. As $\gamma_{0} - \delta \leq \gamma$, hence
\begin{align}
\frac{e(S, H)}{v(S, H)} & \leq \frac{\gamma(k-1) + \delta(k-2) + \gamma(k-3) + 1 + \beta(k-2) + 2(k-3)}{\gamma + \gamma + \delta + \beta + 2} \nonumber\\
&= \frac{\gamma(2k-4) + \delta(k-2) + \beta(k-2) + 1 + 2(k-3)}{2\gamma + \delta + \beta + 2} \nonumber\\
&= k - 2 - \frac{1}{2\gamma + \delta + \beta + 2} \leq k - 2 - \Theta\left(\frac{1}{k^{2}}\right).
\end{align} 

\end{proof} 

Theorem~\ref{safe_extension_copies} and Lemma~\ref{lower_bound_lem_1} together guarantee that, given $y_{1}, \ldots, y_{k-4}$ and $\vec{e}$ such that Spoiler's chosen $x_{k-3}$ is adjacent to $x_{i}$ iff $e_{i} = 1$ for all $i \in [k-4]$, there exist in $G_{2}$ vertices $a_{k-3}$, $a_{k-2}$, $a_{k-1}$, $a_{k}$, $a_{k-1,j}$ for $j \in [k-2]$, $a_{k}^{i}$ for $i \in [k-2]$, $a_{k}^{k-1,j}$ for $j \in [k-2]$, $a_{k,i}^{k-1,j}$ for $i, j \in [k-2]$, as described before Lemma~\ref{lower_bound_lem_1}, with $y_{1}, \ldots, y_{k-4}$ playing the roles of $a_{1}, \ldots, a_{k-4}$. Duplicator sets $y_{k-3} = a_{k-3}$. We split Spoiler's move in round $k-2$ into two possibilities, discussed in Subsections~\ref{k-2_case_1} and Subsection~\ref{k-2_case_2}.

\subsection{First possibility for round $k-2$:}\label{k-2_case_1} Suppose Spoiler chooses $x_{k-2}$ that is adjacent to all of $x_{1}, \ldots, x_{k-3}$. Then Duplicator selects $y_{k-2} = a_{k-2}$. If Spoiler selects $x_{k-1}$ adjacent to all of $x_{1}, \ldots, x_{k-2}$, Duplicator selects $y_{k-1} = a_{k-1}$. If Spoiler selects $x_{k}$ adjacent to $x_{1}, \ldots, x_{k-1}$, Duplicator sets $y_{k} = a_{k}$. If Spoiler selects $x_{k}$ that is adjacent to $x_{\ell}$ for all $\ell \in [k-1] \setminus \{i\}$ for some $i \in [k-2]$, Duplicator sets $y_{k} = a_{k}^{i}$. If Spoiler selects $x_{k}$ adjacent to $x_{1}, \ldots, x_{k-2}$ and non-adjacent to $x_{k-1}$, Duplicator sets $y_{k} = a_{k}^{k-1,1}$. By definition, $a_{k}^{k-1,1}$ is adjacent to $y_{1}, \ldots, y_{k-4}, y_{k-3}=a_{k-3}, y_{k-2}=a_{k-2}$ and $a_{k-1,1}$, but non-adjacent to $y_{k-1}=a_{k-1}$, and hence serves as a winning response for Duplicator.

If Spoiler selects $x_{k-1}$ adjacent to $x_{\ell}$ for all $\ell \in [k-2] \setminus \{j\}$ and non-adjacent to $x_{j}$, for some $j \in [k-2]$, Duplicator sets $y_{k-1} = a_{k-1,j}$. If Spoiler selects $x_{k}$ adjacent to $x_{1}, \ldots, x_{k-1}$, Duplicator sets $y_{k} = a_{k}^{k-1,j}$. If Spoiler selects $x_{k}$ adjacent to $x_{\ell}$ for all $\ell \in [k-1] \setminus \{i\}$ and non-adjacent to $x_{i}$, for some $i \in [k-2]$, Duplicator sets $y_{k} = a_{k,i}^{k-1,j}$. If Spoiler selects $x_{k}$ adjacent to $x_{1}, \ldots, x_{k-2}$ and non-adjacent to $x_{k-1}$, Duplicator sets $y_{k} = a_{k}$. By definition, $a_{k}$ is adjacent to $y_{1}, \ldots, y_{k-4}, y_{k-3} = a_{k-3}, y_{k-2} = a_{k-2}$ and $a_{k-1}$ but not to $y_{k-1} = a_{k-1,j}$, and hence serves as a winning response.

If in either of the two cases above, Spoiler selects $x_{k}$ that is adjacent to at most $k-3$ out of $x_{1}, \ldots, x_{k-1}$, then, denoting by $G$ and $H$ respectively the induced subgraphs of $G_{1}$ on $\{x_{1}, \ldots, x_{k}\}$ and $\{x_{1}, \ldots, x_{k-1}\}$, it is immediate that $(G, H)$ is $\alpha$-safe for $\alpha = (k-2-O(k^{-2}))^{-1}$. By Theorem~\ref{safe_extension_copies}, a.a.s.\ there exists $b$ in $G_{2}$ such that $b \sim y_{i}$ iff $x_{k} \sim x_{i}$ for all $i \in [k-1]$. Duplicator sets $y_{k} = b$. 

Suppose Spoiler selects $x_{k-1}$ that is adjacent to at most $k-4$ out of $x_{1}, \ldots, x_{k-2}$. Let $H$ be the induced subgraph of $G_{2}$ on $y_{1}, \ldots, y_{k-4}$, $y_{k-3} = a_{k-3}$, $y_{k-2} = a_{k-2}$, $a_{k-1}$ and $a_{k}$ (where $a_{i}$ for $k-3 \leq i \leq k$ are as in Lemma~\ref{lower_bound_lem_1}). Consider  
\renewcommand{\theenumi}{(\roman{enumi})}
\begin{enumerate}
\item $b_{k-1}$ with $b_{k-1} \sim y_{i}$ iff $x_{k-1} \sim x_{i}$ for all $i \in [k-2]$ and $b_{k-1} \sim a_{k-1}$, but $b_{k-1} \nsim a_{k}$;
\item $b_{k}^{i}$ with $b_{k}^{i} \sim b_{k-1}$ and $b_{k}^{i} \sim y_{\ell}$ for all $\ell \in [k-2] \setminus \{i\}$, but $b_{k}^{i} \nsim y_{i}$, for all $i \in [k-2]$;
\end{enumerate}
and let $G$ be the graph on $y_{1}, \ldots, y_{k-2}, a_{k-1}, a_{k}, b_{k-1}, b_{k}^{i}, i \in [k-2]$ with $H \subset G$ and $E(G) \setminus E(H)$ comprising precisely the edges and non-edges described above.
 
\begin{lemma}\label{lower_bound_lem_2}
The pair $(G, H)$ is $\alpha$-safe for $\alpha = (k-2-t(k))^{-1}$ for a suitable $t(k) = \Theta(k^{-2})$.
\end{lemma}

\begin{proof}
Consider $S$ with $H \subset S \subseteq G$. As in Lemma~\ref{lower_bound_lem_1}, there are two layers, the $(k-1)$-st comprising $b_{k-1}$ and the $k$-th one the rest. If $b_{k-1} \in S$, it brings at most $k-3$ edges. Let $\gamma = \left|S \cap \left\{b_{k}^{i}: i \in [k-2]\right\}\right|$: each vertex in this set brings $k-2$ edges. We then have 
\begin{align} 
\frac{e(S, H)}{v(S, H)} \leq \frac{(k-3) + (k-2)\gamma}{1 + \gamma} = k-2 - \frac{1}{\gamma+1} \leq k - 2 - \Theta\left(\frac{1}{k^{2}}\right). 
\end{align}
If $b_{k-1} \notin S$, then each vertex in $S$ brings at most $k-3$ edges, hence \eqref{each_degree_leq_k-3} holds.

\end{proof}

By Lemma~\ref{lower_bound_lem_2} and Theorem~\ref{safe_extension_copies}, the vertices $b_{k-1}$, $b_{k}^{i}$, $i \in [k-2]$, exist a.a.s.\ in $G_{2}$. Duplicator sets $y_{k-1} = b_{k-1}$. If Spoiler selects $x_{k}$ adjacent to all of $x_{1}, \ldots, x_{k-1}$, Duplicator sets $y_{k} = a_{k-1}$. If Spoiler selects $x_{k}$ that is adjacent to $x_{\ell}$ for all $\ell \in [k-1] \setminus \{i\}$ for some $i \in [k-2]$, and non-adjacent to $x_{i}$, Duplicator sets $y_{k} = b_{k}^{i}$. If Spoiler selects $x_{k}$ adjacent to $x_{1}, \ldots, x_{k-2}$ and non-adjacent to $x_{k-1}$, Duplicator sets $y_{k} = a_{k}$. By definition, $a_{k}$ is adjacent to all of $y_{1}, \ldots, y_{k-4}, y_{k-3} = a_{k-3}, y_{k-2} = a_{k-2}$ and $a_{k-1}$, but not to $b_{k-1}$, and hence serves as a winning response. If Spoiler selects $x_{k}$ that is adjacent to at most $k-3$ out of $x_{1}, \ldots, x_{k-1}$, we argue as above that Duplicator finds a winning response via Theorem~\ref{safe_extension_copies}.

\subsection{Second possibility for round $k-2$:}\label{k-2_case_2} Spoiler chooses $x_{k-2}$ that is adjacent to at most $k-4$ out of $x_{1}, \ldots, x_{k-3}$. Let $H$ be the induced subgraph of $G_{2}$ on $y_{1}, \ldots, y_{k-3}$. Consider \renewcommand{\theenumi}{(\roman{enumi})}
\begin{enumerate}
\item $c_{k-2}$ with $c_{k-2} \sim y_{i}$ iff $x_{k-2} \sim x_{i}$ for all $i \in [k-3]$;
\item $c_{k-1}$ which is adjacent to all of $y_{1}, \ldots, y_{k-3}, c_{k-2}$;
\item $c_{k-1,j}$ which is adjacent to $y_{\ell}$ for all $\ell \in [k-3] \setminus \{j\}$, adjacent to $c_{k-2}$, and non-adjacent to $y_{j}$, for all $j \in [k-3]$; $c_{k-1,k-2}$ which is adjacent to $y_{1}, \ldots, y_{k-3}$ and non-adjacent to $c_{k-2}$;
\item $c_{k}$ which is adjacent to all of $y_{1}, \ldots, y_{k-3}, c_{k-2}, c_{k-1}$;
\item $c_{k}^{i}$ which is adjacent to $y_{\ell}$ for all $\ell \in [k-3] \setminus \{i\}$, adjacent to $c_{k-2}$ and $c_{k-1}$, and non-adjacent to $y_{i}$, for all $i \in [k-3]$; $c_{k}^{k-2}$ which is adjacent to $y_{1}, \ldots, y_{k-3}, c_{k-1}$ and non-adjacent to $c_{k-2}$;
\item $c_{k}^{k-1,j}$ which is adjacent to all of $y_{1}, \ldots, y_{k-3}, c_{k-2}, c_{k-1,j}$, for all $j \in [k-2]$;
\item  for each $j \in [k-2]$, $c_{k,i}^{k-1,j}$ which is adjacent to $y_{\ell}$ for all $\ell \in [k-3] \setminus \{i\}$, to $c_{k-2}$ and $c_{k-1,j}$, and non-adjacent to $y_{i}$, for each $i \in [k-3]$; $c_{k,k-2}^{k-1,j}$ which is adjacent to $y_{1}, \ldots, y_{k-3}, c_{k-1,j}$ and non-adjacent to $c_{k-2}$.
\end{enumerate}
and let $G$ be the graph on $y_{1}, \ldots, y_{k-3}, c_{k-2}, c_{k-1}, c_{k-1,j}, j \in [k-2], c_{k}, c_{k}^{i}, i \in [k-2], c_{k}^{k-1,j}, j \in [k-2], c_{k,i}^{k-1,j}, i, j \in [k-2]$, with $H \subset G$ and $E(G) \setminus E(H)$ comprising precisely the edges and non-edges described above.

\begin{lemma}\label{lower_bound_lem_3}
The pair $(G, H)$ is $\alpha$-safe for $\alpha = (k-2-t(k))^{-1}$ for a suitable $t(k) = \Theta(k^{-2})$.
\end{lemma}

\begin{proof}
As in Lemma~\ref{lower_bound_lem_1}, there are three layers in which the vertices in a graph $S$ with $H \subset S \subseteq G$ can be added: the $(k-2)$-nd layer comprising only $c_{k-2}$, the $(k-1)$-st comprising $c_{k-1}$, $c_{k-1,j}$ for $j \in [k-2]$, and the $k$-th layer the rest; we also have the same ordering of vertices in any $S$ with $H \subset S \subseteq G$ and the same meaning of ``bringing edges". Suppose 
\begin{itemize}
\item $\mathcal{A}_{0} = S \cap \left\{c_{k}, c_{k}^{k-1,j}: j \in [k-2]\right\}$ with $\gamma_{0} = \left|\mathcal{A}_{0}\right|$, and each vertex in this set brings at most $k-1$ edges;
\item $\mathcal{A} = S \cap \left\{c_{k-1}, c_{k-1,j}: j \in [k-2]\right\}$ with $\gamma = |\mathcal{A}|$, and at most one vertex in this set brings at most $k-2$ edges, and each of the others at most $k-3$ edges;
\item $\beta = \left|S \cap \left\{c_{k}^{i}, c_{k,i}^{k-1,j}: i, j \in [k-2]\right\}\right|$, and each vertex in this set brings at most $k-2$ edges;
\item $c = |S \cap \{c_{k-2}\}|$, and if this set is non-empty, it brings at most $k-4$ edges.
\end{itemize}
At most $\gamma$ vertices in $\mathcal{A}_{0}$ bring $k-1$ edges each, and the remaining bring at most $k-2$ edges each. Let $\delta = \gamma_{0} - \gamma$ if $\gamma_{0} \geq \gamma$, and $\delta = 0$ otherwise. If $c = 0$, each vertex in $S$ brings at most $k-3$ edges, hence \eqref{each_degree_leq_k-3} holds. If $c = 1$, the same analysis as the $c = 1$ case of Lemma~\ref{lower_bound_lem_1} applies.


\end{proof}

By Lemma~\ref{lower_bound_lem_3} and Theorem~\ref{safe_extension_copies}, the vertices $c_{k-2}, c_{k-1}, c_{k-1,j}, j \in [k-2], c_{k}, c_{k}^{i}, i \in [k-2], c_{k}^{k-1,j}, j \in [k-2], c_{k,i}^{k-1,j}, i, j \in [k-2]$ exist a.a.s.\ in $G_{2}$. Duplicator sets $y_{k-2} = c_{k-2}$. If Spoiler selects $x_{k-1}$ which is adjacent to $x_{1}, \ldots, x_{k-2}$, Duplicator sets $y_{k-1} = c_{k-1}$. If, after that, Spoiler selects $x_{k}$ that is adjacent to $x_{1}, \ldots, x_{k-1}$, Duplicator sets $y_{k} = c_{k}$. If Spoiler selects $x_{k}$ such that $x_{k}$ is adjacent to $x_{\ell}$ for all $\ell \in [k-1] \setminus \{i\}$ for some $i \in [k-2]$ and non-adjacent to $x_{i}$, Duplicator sets $y_{k} = c_{k}^{i}$. If Spoiler selects $x_{k}$ that is adjacent to $x_{1}, \ldots, x_{k-2}$ and non-adjacent to $x_{k-1}$, Duplicator sets $y_{k} = c_{k}^{k-1,1}$. By definition, $c_{k}^{k-1,1}$ is adjacent to $y_{1}, \ldots, y_{k-3}, y_{k-2} = c_{k-2}$ and $c_{k-1,1}$ and non-adjacent to $y_{k-1} = c_{k-1}$, hence serves as a winning response.

If Spoiler selects $x_{k-1}$ which is adjacent to $x_{\ell}$ for all $\ell \in [k-2] \setminus \{j\}$ and non-adjacent to $x_{j}$, for some $j \in [k-2]$, Duplicator sets $y_{k-1} = c_{k-1,j}$. If, after that, Spoiler selects $x_{k}$ which is adjacent to $x_{1}, \ldots, x_{k-1}$, Duplicator sets $y_{k} = c_{k}^{k-1,j}$. If Spoiler selects $x_{k}$ which is adjacent to $x_{\ell}$ for all $\ell \in [k-1] \setminus \{i\}$ and non-adjacent to $x_{i}$ for some $i \in [k-2]$, Duplicator sets $y_{k} = c_{k,i}^{k-1,j}$. If Spoiler selects $x_{k}$ which is adjacent to $x_{1}, \ldots, x_{k-2}$ and non-adjacent to $x_{k-1}$, Duplicator sets $y_{k} = c_{k}$. By definition, $c_{k}$ is adjacent to $y_{1}, \ldots, y_{k-3}, y_{k-2} = c_{k-2}$ and to $c_{k-1}$ but not to $c_{k-1,j}$, hence serves as a winning response. 

In both the cases above, if Spoiler chooses $x_{k}$ that is adjacent to at most $k-3$ out of $x_{1}, \ldots, x_{k-1}$, Duplicator wins by the same argument as outlined previously, using Theorem~\ref{safe_extension_copies}. Suppose Spoiler selects $x_{k-1}$ that is adjacent to at most $k-4$ out of $x_{1}, \ldots, x_{k-2}$. By Lemma~\ref{lower_bound_lem_2}, we conclude that there exist, a.a.s.\ in $G_{2}$, the following vertices, edges and non-edges:
\renewcommand{\theenumi}{(\roman{enumi})}
\begin{enumerate}
\item $d_{k-1}$ with $d_{k-1} \sim y_{i}$ iff $x_{k-1} \sim x_{i}$ for all $i \in [k-2]$ and $d_{k-1} \sim c_{k-1}$, but $d_{k-1} \nsim c_{k}$;
\item $d_{k}^{i}$ which is adjacent to $y_{\ell}$ for all $\ell \in [k-2] \setminus \{i\}$, adjacent to $d_{k-1}$, and non-adjacent to $y_{i}$, for each $i \in [k-2]$.
\end{enumerate}
Duplicator sets $y_{k-1} = d_{k-1}$. If Spoiler selects $x_{k}$ which is adjacent to all of $x_{1}, \ldots, x_{k-1}$, Duplicator sets $y_{k} = c_{k-1}$. If Spoiler selects $x_{k}$ which is adjacent to $x_{\ell}$ for all $\ell \in [k-1] \setminus \{i\}$ for some $i \in [k-2]$, and non-adjacent to $x_{i}$, Duplicator sets $y_{k} = d_{k}^{i}$. If Spoiler selects $x_{k}$ which is adjacent to $x_{1}, \ldots, x_{k-2}$ and not to $x_{k-1}$, Duplicator sets $y_{k} = c_{k}$. Finally, if Spoiler chooses $x_{k}$ that is adjacent to at most $k-3$ out of $x_{1}, \ldots, x_{k-1}$, Duplicator wins by the same argument as outlined previously, using Theorem~\ref{safe_extension_copies}.
 

\section{Existential first order sentences of quantifier depth $4$}\label{k=4}
In this section, we consider the class $\mathcal{E}_{4}$ of all existential first order sentences with quantifier depth at most $4$. From Theorem~\ref{bounded_qd}, we conclude that for all $0 < \alpha < \frac{1}{2}$, the random graph $G(n, n^{-\alpha})$ satisfies the zero-one law for $\mathcal{E}_{4}$. Our goal, therefore, is to find the minimum $\alpha_{4} \in \left[\frac{1}{2}, 1\right)$ such that the random graph $G(n, n^{-\alpha})$ fails to satisfy the zero-one law for $\mathcal{E}_{4}$. 

To this end, our purpose is to come up with an $\alpha_4 \in \left[\frac{1}{2},1\right)$, a finite graph $G_{0}$ and an $\mathcal{E}_4$-sentence $\varphi$ such that

\begin{itemize}
    \item for every $\alpha<\alpha_{4}$, $G(n,n^{-\alpha})$ obeys the zero-one law for $\mathcal{E}_4$;
    \item a.a.s., $\varphi$ is true on $G(n,n^{-\alpha_{4}})$ if and only if $G(n,n^{-\alpha_{4}})$ contains an induced subgraph isomorphic to $G_0$;
    \item $G_0$ is strictly balanced and its density equals $1/\alpha_4$.
\end{itemize}
Then, by Theorem~\ref{strictly_balanced_no._copies_thm}, we know that the limit of the probability that $G\left(n, n^{-\alpha_{4}}\right)$ contains a copy of $G_{0}$ is a positive real strictly less than $1$.  

We first state here the sentence $\varphi$. For any vertices $a_{1}, a_{2}, b_{1}, b_{2}, b_{3}$, define the sentence
\begin{multline}\label{atom_1_k=4}
\psi_{1}(a_{1}, a_{2}, b_{1}, b_{2}, b_{3}) = [b_{1} \sim a_{1}] \wedge [b_{1} \sim a_{2}] \wedge [b_{2} \sim a_{1}] \wedge [b_{2} \nsim a_{2}] \wedge [b_{3} \nsim a_{1}] \wedge [b_{3} \sim a_{2}],
\end{multline}
and for vertices $a_{1}, a_{2}, a_{3}, b_{1}, b_{2}, b_{3}, b_{4}$, define the sentence
\begin{multline}\label{atom_2_k=4}
\psi_{2}(a_{1}, a_{2}, a_{3}, b_{1}, b_{2}, b_{3}, b_{4}) = [b_{1} \sim a_{1}] \wedge [b_{1} \sim a_{2}] \wedge [b_{1} \sim a_{3}] \wedge [b_{2} \sim a_{1}] \wedge [b_{2} \sim a_{2}] \wedge [b_{2} \nsim a_{3}] \\ \wedge [b_{3} \sim a_{1}] \wedge [b_{3} \nsim a_{2}] \wedge [b_{3} \sim a_{3}] \wedge [b_{4} \nsim a_{1}] \wedge [b_{4} \sim a_{2}] \wedge [b_{4} \sim a_{3}].
\end{multline}
The sentence $\varphi$ is now stated as follows:

\begin{multline}\label{sentence_k_4}
\exists x \exists a \exists a_{1,1} \exists a_{1,0} \exists a_{0,1} \exists a_{1,1}^{1,1} \exists a'_{1,1} \exists a_{1,1}^{1,0} \exists a_{1,1}^{0,1} \exists a_{1,0}^{1,1} \exists a'_{1,0} \exists a_{1,0}^{1,0} \exists a_{1,0}^{0,1} \exists a_{0,1}^{1,1} \exists a'_{0,1} \exists a_{0,1}^{1,0} \exists a_{0,1}^{0,1} \\ \exists b \exists b_{1,1} \exists b_{1,0} \exists b_{0,1} \exists b_{1,1}^{1,1} \exists b_{1,1}^{1,0} \exists b'_{1,1} \exists b_{1,1}^{0,1} \exists b_{1,0}^{1,1} \exists b'_{1,0} \exists b_{1,0}^{1,0} \exists b_{1,0}^{0,1} \exists b_{0,1}^{1,1} \exists b'_{0,1} \exists b_{0,1}^{1,0} \exists b_{0,1}^{0,1} [x \sim a] \wedge \\ \left[\psi_{1}(x, a, a_{1,1}, a_{1,0}, a_{0,1})\right] \wedge \left[\psi_{2}\left(x, a, a_{1,1}, a_{1,1}^{1,1}, a'_{1,1}, a_{1,1}^{1,0}, a_{1,1}^{0,1}\right)\right] \wedge \left[\psi_{2}\left(x, a, a_{1,0}, a_{1,0}^{1,1}, a'_{1,0}, a_{1,0}^{1,0}, a_{1,0}^{0,1}\right)\right] \wedge \\ \left[\psi_{2}\left(x, a, a_{0,1}, a_{0,1}^{1,1}, a'_{0,1}, a_{0,1}^{1,0}, a_{0,1}^{0,1}\right)\right] \wedge [x \nsim b] \wedge \left[\psi_{1}(x, b, b_{1,1}, b_{1,0}, b_{0,1})\right] \wedge \\ \left[\psi_{2}\left(x, b, b_{1,1}, b_{1,1}^{1,1}, b'_{1,1}, b_{1,1}^{1,0}, b_{1,1}^{0,1}\right)\right] \wedge \left[\psi_{2}\left(x, b, b_{1,0}, b_{1,0}^{1,1}, b'_{1,0}, b_{1,0}^{1,0}, b_{1,0}^{0,1}\right)\right] \\ \wedge \left[\psi_{2}\left(x, b, b_{0,1}, b_{0,1}^{1,1}, b'_{0,1}, b_{0,1}^{1,0}, b_{0,1}^{0,1}\right)\right].
\end{multline}

We now show that the graph $G_{0}$, shown in Figure~\ref{fig_final_optimum}, is such that:
\renewcommand{\theenumi}{(\roman{enumi})}
\begin{enumerate}
\item \label{claim_1} the graph $G_{0}$ is strictly balanced;
\item \label{claim_2} if a graph $G$ contains an induced subgraph isomorphic to $G_{0}$, then $G \models \varphi$; if a graph $G \models \varphi$, then $G$ contains either an induced subgraph isomorphic to $G_{0}$ or it contains an induced subgraph with at most $43$ vertices whose maximal density is at least as large as $\maxden(G_{0}) = \rho(G_{0})$ (this identity follows from the observation that $G_{0}$ is strictly balanced).

\end{enumerate}
\begin{figure}[h!]
  \begin{center}
   \includegraphics[width=0.8\textwidth]{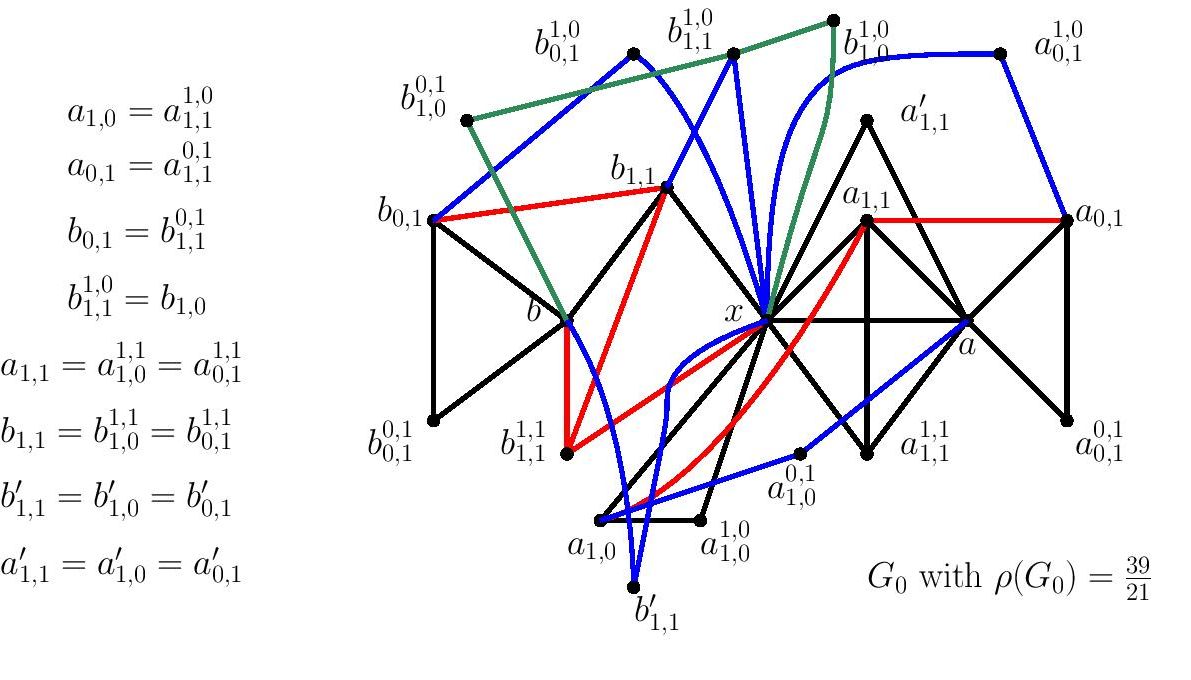}
\end{center}
  \caption{}
\label{fig_final_optimum}
\end{figure}

We first prove the claim in \ref{claim_1}, in the following lemma:
\begin{lemma}\label{strict_balanced_k=4}
The graph $G_{0}$ in Figure~\ref{fig_final_optimum} is strictly balanced.
\end{lemma}

\begin{proof}
Observe that $G_{0}$ can be split into two edge-disjoint parts having only $x$ in common, illustrated in Figure~\ref{parts}.  
\begin{figure}[h!]
  \begin{center}
   \includegraphics[width=0.8\textwidth]{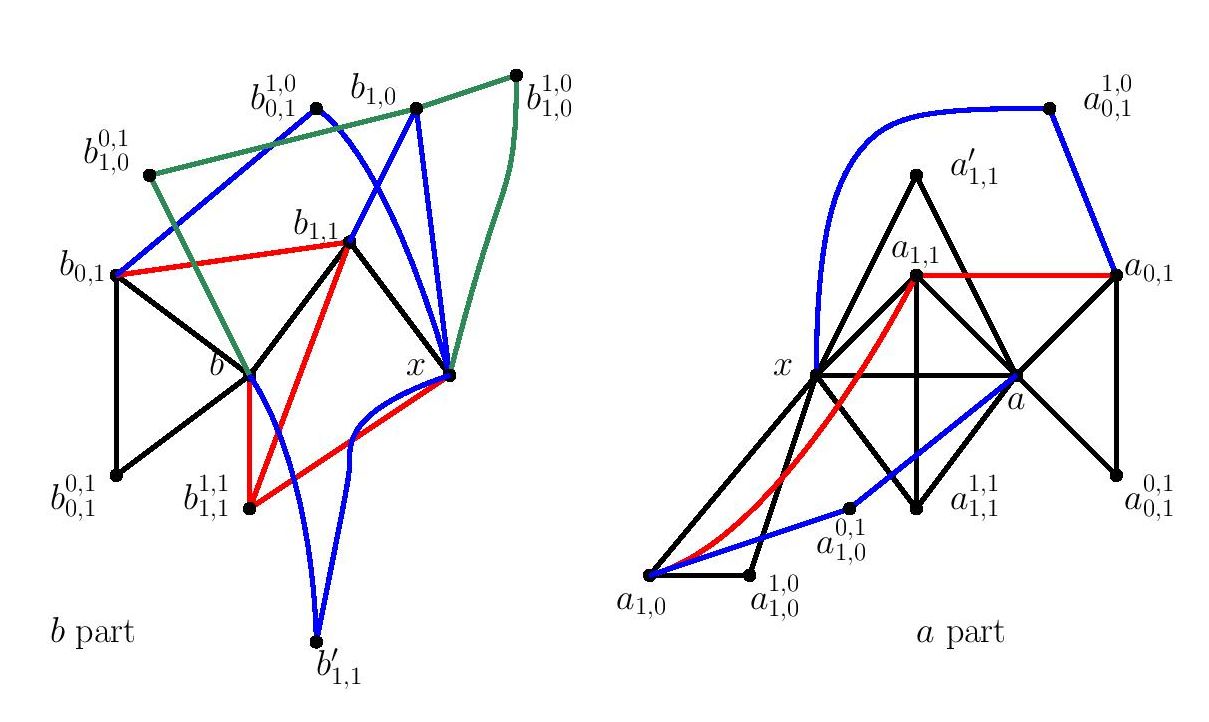}
\end{center}
  \caption{}
\label{parts}
\end{figure}
The vertices in the $a$ part can be added in the following sequence:
\renewcommand{\theenumi}{(\alph{enumi})}
\begin{enumerate}
\item in step $0$, the vertex $x$,
\item in step $1$, the vertex $a$, bringing $1$ edge,
\item in step $2$, the vertices $a'_{1,1}$ and $a_{1,1}$, each bringing precisely $2$ edges,
\item in step $3$, the vertices $a_{1,0}$ and $a_{0,1}$, each bringing precisely $2$ edges,
\item in step $4$, the vertices $a_{1,0}^{1,0}$, $a_{0,1}^{0,1}$, $a_{1,0}^{0,1}$ and $a_{0,1}^{1,0}$, each bringing precisely $2$ edges,
\item in step $5$, the vertex $a_{1,1}^{1,1}$, bringing $3$ edges.
\end{enumerate}
The vertices in the $b$ part can be added in the following sequence:
\renewcommand{\theenumi}{(\alph{enumi})}
\begin{enumerate}
\item in step $0$, the vertex $x$,
\item in step $1$, the vertex $b_{1,1}$, bringing $1$ edge,
\item in step $2$, the vertices $b_{1,1}^{1,1}$ and $b_{1,0}$, each bringing $2$ edges,
\item in step $3$, the vertices $b$ and $b_{1,0}^{1,0}$, each bringing $2$ edges,
\item in step $4$, the vertices $b_{0,1}$, $b'_{1,1}$, $b_{1,0}^{0,1}$, each bringing $2$ edges,
\item in step $5$, the vertices $b_{0,1}^{1,0}$ and $b_{0,1}^{0,1}$, each bringing $2$ edges.
\end{enumerate}

If we consider a subgraph $H$ of $G_{0}$ containing $x$, $a$, $b_{1,1}$, and $y_{i}$ vertices from step $i$ of $a$ part and $z_{i}$ vertices from step $i$ of $b$ part for $i \in \{2, 3, 4, 5\}$, with $y_{2}, y_{3}, z_{2}, z_{3}, z_{5} \in \{0, 1, 2\}$, $y_{4}, z_{4} \in \{0, 1, 2, 3, 4\}$ and $y_{5} \in \{0, 1\}$, the number of edges in $H$ is at most $1 + 2 \sum_{i=2}^{4} (y_{i}+z_{i}) + 3y_{5} + 2z_{5},$ whereas the number of vertices is $3 + \sum_{i=2}^{5}(y_{i}+z_{i})$. Letting $y = \sum_{i=2}^{4} (y_{i}+z_{i}) + z_{5}$ and $z = y_{5}$, $\rho(H)$ is at most 
\begin{equation}
\frac{1 + 2y + 3z}{3 + y + z}.
\end{equation}
We note here that
\begin{align}
\frac{3 + 2y + 3z}{4 + y + z} - \frac{1 + 2y + 3z}{3 + y + z} &= \frac{5-z}{(4+y+z)(3+y+z)}, \nonumber
\end{align}
which is positive since $z \leq 1$. Thus the density increases with $y$. Similarly, the density increases with $z$. Thus $\rho(H) < \rho(G_{0})$ whenever $H$ is a proper subgraph of $G_{0}$ containing $x$, $a$, $b_{1,1}$. When $H$ excludes at least one of $x$, $a$ and $b_{1,1}$, its density is even lower, since in both $a$ and $b$ parts, the vertex added in step $2$ brings $1$ edge, and subsequent vertices bring at most $2$ edges each, except possibly for $a_{1,1}^{1,1}$ if it is included in $H$.

\end{proof}

The first part of \ref{claim_2} is trivial. The appropriate values of $v(G_{0})$, $e(G_{0})$ and $\rho(G_{0})$ are given in Figure~\ref{fig_final_optimum}. We now prove the second part of \ref{claim_2}. For $\varphi$ to be true on a graph $G$, the bare minimum we need are the following vertices, edges and non-edges in $G$: 
\begin{enumerate}
\item $x$, $a$, $b$ where $x \sim a$ and $b \nsim x$;
\item $a_{1,1}$ with $a_{1,1} \sim x$, $a_{1,1} \sim a$;
\item $a_{1,0}$ with $a_{1,0} \sim x$ and $a_{1,0} \nsim a$;
\item $a_{1,1}^{1,1}$ with $a_{1,1}^{1,1} \sim x$, $a_{1,1}^{1,1} \sim a$ and $a_{1,1}^{1,1} \sim a_{1,1}$;
\item $a'_{1,1}$ with $a'_{1,1} \sim x$, $a'_{1,1} \sim a$, $a'_{1,1} \nsim a_{1,1}$;
\item $a_{1,0}^{1,0}$ with $a_{1,0}^{1,0} \sim x$, $a_{1,0}^{1,0} \sim a_{1,0}$ and $a_{1,0}^{1,0} \nsim a$;
\item $b_{0,1}$ with $b_{0,1} \sim b$, $b_{0,1} \nsim x$;
\item $b_{0,1}^{0,1}$ with $b_{0,1}^{0,1} \sim b$, $b_{0,1}^{0,1} \sim b_{0,1}$ and $b_{0,1}^{0,1} \nsim x$.
\end{enumerate}
Let us call this subgraph $H$ with $v(H) = 10$ and $e(H) = 14$ (see Figure~\ref{fig_1}). 
\begin{figure}[h!]
  \begin{center}
   \includegraphics[width=0.5\textwidth]{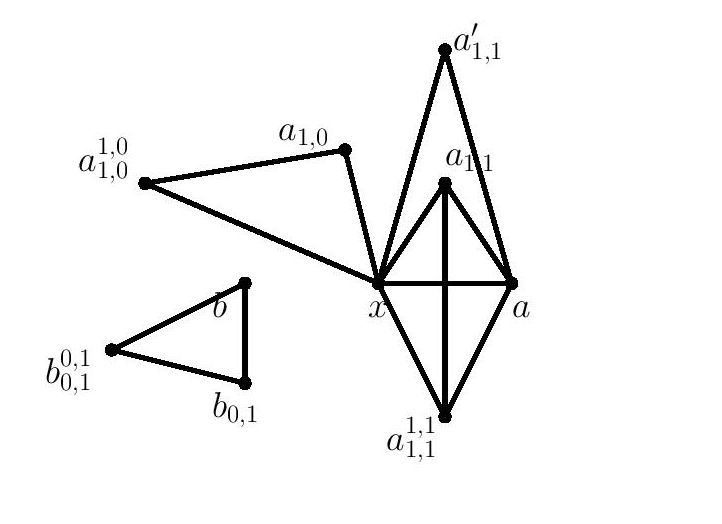}
\end{center}
  \caption{}
\label{fig_1}
\end{figure}
The set of variables left to evaluate is
\begin{multline}\label{total_remaining}
S_{0} = \Big\{b_{1,1}, a_{0,1}, a_{0,1}^{0,1}, a_{1,1}^{1,0}, a_{1,1}^{0,1}, a_{1,0}^{1,1}, a'_{1,0}, a_{1,0}^{0,1}, a_{0,1}^{1,1}, a'_{0,1}, a_{0,1}^{1,0}, b_{1,1}^{1,1}, b'_{1,1}, b_{1,1}^{1,0}, b_{1,1}^{0,1}, b_{1,0}, b_{1,0}^{1,1}, b'_{1,0}, \\ b_{1,0}^{1,0}, b_{1,0}^{0,1}, b_{0,1}^{1,1}, b'_{0,1}, b_{0,1}^{1,0}\Big\}.
\end{multline}
In what follows, we consider a family $\Sigma$ of finite graphs $\Gamma$ with at most $43$ vertices, adding vertices evaluating the elements of $S_{0}$ to $H$, such that if $G \models \varphi$, it contains an induced subgraph isomorphic to some $\Gamma \in \Sigma$. We show that $\maxden(\Gamma) > \rho(G_{0})$ for each $\Gamma \in \Sigma \setminus G_{0}$ by considering various scenarios, the most difficult of which is discussed in Subsection~\ref{subcase_1} below, and the rest in the Appendix (Section~\ref{Appendix}). The notation introduced below, unless otherwise stated, will be used in the Appendix.

\subsection{The most difficult scenario:}\label{subcase_1}
Here we assume that $a_{0,1}$ and $a_{0,1}^{0,1}$ are added as new vertices to $H$, thus yielding the graph $H_{0}$ as in Figure~\ref{H_{0}_subcase_1}, with $17$ edges and $12$ vertices.
\begin{figure}[h!]
  \begin{center}
   \includegraphics[width=0.5\textwidth]{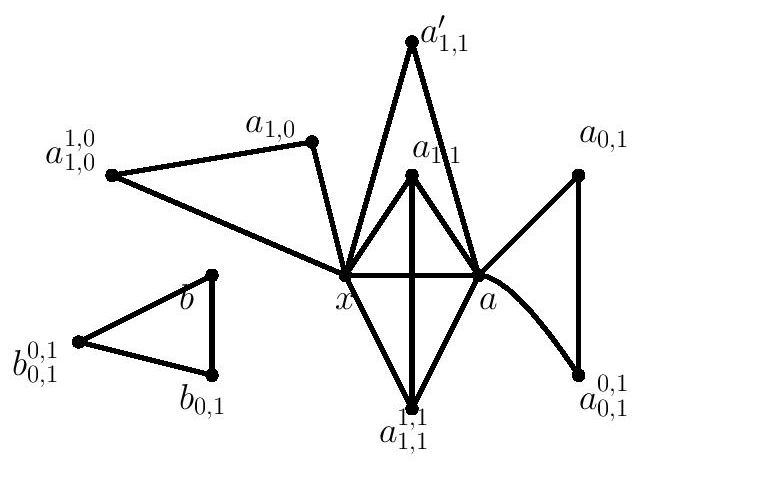}
\end{center}
  \caption{}
\label{H_{0}_subcase_1}
\end{figure}
Either $b_{1,1}$ is a new vertex added to $H_{0}$, in which case we add edges $\{b_{1,1}, x\}$ and $\{b_{1,1}, b\}$, or it belongs to
\begin{equation}\label{S_{1}}
S_{1} = \left\{a, a_{1,0}, a_{1,1}, a_{1,1}^{1,1}, a'_{1,1}, a_{1,0}^{1,0}\right\},
\end{equation}
in which case we add $\{b_{1,1}, b\}$. We denote by $H'_{0}$ the graph obtained after accounting for $b_{1,1}$. We define the vertices: 
\renewcommand{\theenumi}{(\roman{enumi})}
\begin{enumerate}
\item $v_{1}$ is adjacent to all of $x$, $b$ and $b_{0,1}$,
\item $v_{2}$ is adjacent to all of $x$, $b$ and $b_{1,1}$,
\item $v_{3}$ is adjacent to all of $x$, $a$ and $a_{1,0}$,
\item $v_{4}$ is adjacent to all of $x$, $a$ and $a_{0,1}$,
\end{enumerate}
and we call them the \emph{first-level vertices}. We call a first-level vertex \emph{old} if it coincides with a vertex in $H'_{0}$, we call it \emph{unique} if it is added as a new vertex to $H'_{0}$ and does not coincide with any other first-level vertex, and \emph{banal} otherwise. The graph obtained after we have accounted for all first-level vertices is denoted $H_{1}$. We define the vertices:
\renewcommand{\theenumi}{(\roman{enumi})}
\begin{enumerate}
\item $u_{1}$ is adjacent to $x$ and $b_{0,1}$, non-adjacent to $b$,
\item $u_{2}$ is adjacent to $x$ and $b$, non-adjacent to $b_{1,1}$,
\item $u_{3}$ is adjacent to $b$ and $b_{1,1}$, non-adjacent to $x$,
\item $u_{4}$ is adjacent to $x$ and $b_{1,1}$, non-adjacent to $b$,
\item $u_{5}$ is adjacent to $a$ and $a_{1,0}$, non-adjacent to $x$,
\item $u_{6}$ is adjacent to $x$ and $a_{1,1}$, non-adjacent to $a$,
\item $u_{7}$ is adjacent to $a$ and $a_{1,1}$, non-adjacent to $x$,
\item $u_{8}$ is adjacent to $x$ and $a_{0,1}$, non-adjacent to $a$.
\end{enumerate}
We call these \emph{second-level vertices}. We call a second-level vertex \emph{old} if it coincides with a vertex in $H'_{0}$, we call it \emph{banal} if it coincides with a first-level vertex that is not old (in other words, coincides with a vertex in $V(H_{1}) \setminus V\left(H'_{0}\right)$), and we call it \emph{unique} otherwise. Note that two or more unique second-level vertices are allowed to coincide with each other.

\begin{remark}
Note that we have not yet considered a vertex evaluating $b_{1,0}$, its ``extensions" which are vertices evaluating $b_{1,0}^{1,1}$, $b'_{1,0}$, $b_{1,0}^{1,0}$, $b_{1,0}^{0,1}$, as well as $b'_{0,1}$, $a'_{0,1}$ and $a'_{1,0}$. As it happens, some of these end up coinciding with some of the vertices in Figure~\ref{fig_final_optimum}.
\end{remark}


If $v_{3}$ is old, it either coincides with a vertex in 
\begin{equation}\label{S_{2}}
S_{2} = \left\{a_{1,1}, a_{1,1}^{1,1}, a'_{1,1}\right\},
\end{equation}
in which case we add $\{v_{3}, a_{1,0}\}$, or $v_{3} = b_{1,1}$ when $b_{1,1} \in V\left(H'_{0}\right) \setminus V(H_{0})$, in which case we add $\{b_{1,1}, a\}$ and $\{b_{1,1}, a_{1,0}\}$. When $v_{4}$ is old, either $v_{4} \in S_{2}$, in which case we add $\{v_{4}, a_{0,1}\}$, or with $b_{1,1}$ when $b_{1,1} \in V\left(H'_{0}\right) \setminus V(H_{0})$, in which case we add $\{b_{1,1}, a\}$ and $\{b_{1,1}, a_{0,1}\}$. When $v_{1}$ is old and coincides with $b_{1,1}$, we add $\{b_{1,1}, b_{0,1}\}$; if $v_{1} \in S_{1} \setminus \{b_{1,1}\}$, we add $\{v_{1}, b\}$ and $\{v_{1}, b_{0,1}\}$. When $v_{2}$ is old, the possibilities according to $b_{1,1}$ are:
\renewcommand{\theenumi}{(\alph{enumi})}
\begin{enumerate}
\item if $b_{1,1} \in V\left(H'_{0}\right) \setminus V(H_{0})$, and $v_{2} \in S_{1}$, we add $\{v_{2}, b\}$ and $\{v_{2}, b_{1,1}\}$;
\item if $b_{1,1} = a_{1,0}$, and $v_{2} = a_{1,0}^{1,0}$, we add $\{a_{1,0}^{1,0}, b\}$; if $v_{2} \neq a_{1,0}^{1,0}$, we add $\{v_{2}, a_{1,0}\}$ and $\{v_{2}, b\}$;
\item if $b_{1,1} = a_{1,0}^{1,0}$ and $v_{2} = a_{1,0}$, we add $\{a_{1,0}, b\}$; if $v_{2} \neq a_{1,0}$, we add $\{v_{2}, a_{1,0}^{1,0}\}$ and $\{v_{2}, b\}$;
\item if $b_{1,1} = a_{1,1}$ and $v_{2} \in \{a_{1,0}, a_{1,0}^{1,0}\}$, we add $\{v_{2}, b\}$ and $\{v_{2}, a_{1,1}\}$; if $v_{2} \in \{a, a_{1,1}^{1,1}\}$, we add $\{v_{2}, b\}$;
\item if $b_{1,1} = a_{1,1}^{1,1}$ and $v_{2} \in \{a_{1,0}, a'_{1,1}, a_{1,0}^{1,0}\}$, we add $\{v_{2}, b\}$ and $\{v_{2}, a_{1,1}^{1,1}\}$; if $v_{2} \in \{a, a_{1,1}\}$, we add $\{v_{2}, b\}$;
\item if $b_{1,1} = a'_{1,1}$ and $v_{2} \in \{a_{1,1}^{1,1}, a_{1,0}, a_{1,0}^{1,0}\}$, we add $\{v_{2}, b\}$ and $\{v_{2}, a'_{1,1}\}$; if $v_{2} = a$, we add $\{v_{2}, b\}$;
\item if $b_{1,1} = a$, then $v_{2} \in S_{2}$, and we add $\{v_{2}, b\}$.

\end{enumerate}


\begin{lemma}\label{old_first_level_subcase_1}
Suppose $\nu_{1}$ is the number of new, distinct first-level vertices (including both unique and banal ones) that are added to $H'_{0}$. Then the number of edges that we need to add on account of all first-level vertices is at least $2\nu_{1}+5$ when $b_{1,1} \in V\left(H'_{0}\right) \setminus V(H_{0})$ and $v_{2}$ is old; in all other cases, this number is at least $2\nu_{1}+4$.
\end{lemma}

\begin{proof}
When $b_{1,1} \in V\left(H'_{0}\right) \setminus V(H_{0})$ and both $v_{2}$ and $v_{3}$ are old, they bring edges $\{b_{1,1}, a\}$, $\{b_{1,1}, a_{1,0}\}$, $\{a_{1,0}, b\}$ when $v_{2} = a_{1,0}$ and $v_{3} = b_{1,1}$, or edges $\{v_{3}, a_{1,0}\}$, $\{v_{2}, b\}$, $\{v_{2}, b_{1,1}\}$ when $v_{3} \neq b_{1,1}$, and $4$ edges otherwise. If $v_{2}$ is old and $v_{3}$ is not, then $v_{2}$ brings edges $\{v_{2}, b\}$ and $\{v_{2}, b_{1,1}\}$. If $v_{1}$ is old, it brings $\{v_{1}, b_{0,1}\}$; if $v_{4}$ is old, it brings $\{v_{4}, a_{0,1}\}$. This shows that the number of edges added due to old first-level vertices is at least one more than the number of old first-level vertices when $v_{2}$ is old.

When $b_{1,1} \in V\left(H'_{0}\right) \setminus V(H_{0})$ but $v_{2}$ is not old, we get $\{v_{3}, a_{1,0}\}$ when $v_{3}$ is old, $\{v_{4}, a_{0,1}\}$ when $v_{4}$ is old, $\{v_{1}, b_{0,1}\}$ when $v_{1}$ is old; when $b_{1,1} \in V(H_{0})$, we additionally get $\{v_{2}, b\}$ when $v_{2}$ is old. Since $\{v_{1}, b_{0,1}\}$, $\{v_{2}, b\}$, $\{v_{3}, a_{1,0}\}$ and $\{v_{4}, a_{0,1}\}$ are all distinct, hence the number of edges added due to old first-level vertices bring is at least as large as the number of old first-level vertices here.  

Whether $b_{1,1} \in V\left(H'_{0}\right) \setminus V(H_{0})$ or not, any $2$ first-level vertices coinciding in a banal vertex bring at least $4$ edges, any $3$ coinciding in a banal vertex bring at least $5$ edges, and all $4$ coinciding in a banal vertex bring at least $6$ edges. This concludes the proof.

\end{proof}


\begin{lemma}\label{unique_banal_2nd_level}
If $s$ banal second-level vertices coincide with a first-level vertex $v_{i}$, except for when $u_{2}$ and $v_{1}$ coincide, at least $s$ additional edges are introduced due to these second-level vertices. If $s$ many unique second-level vertices coincide with each other, they together bring at least $s+1$ edges.
\end{lemma}
\begin{proof}
The only second-level vertices that can coincide with $v_{1}$ are $u_{2}$, $u_{6}$ and $u_{8}$: when $v_{1} = u_{2}$, we add no edge; when $v_{1} = u_{6}$, we add $\{v_{1}, a_{1,1}\}$; when $v_{1} = u_{8}$, we add $\{v_{1}, a_{0,1}\}$. The only second-level vertices that can coincide with $v_{2}$ are $u_{6}$ and $u_{8}$: when $v_{2} = u_{6}$, we add $\{v_{2}, a_{1,1}\}$; when $v_{2} = u_{8}$, we add $\{v_{2}, a_{0,1}\}$. The only second-level vertices that can coincide with $v_{i}$, for $i \in \{3, 4\}$, are $u_{1}$, $u_{2}$ and $u_{4}$: when $v_{i} = u_{1}$, we add $\{v_{i}, b_{0,1}\}$; when $v_{i} = u_{2}$, we add $\{v_{i}, b\}$; when $v_{i} = u_{4}$, we add $\{v_{i}, b_{1,1}\}$. This proves the first part of Lemma~\ref{unique_banal_2nd_level}.

Some or all of $u_{3}$, $u_{5}$ and $u_{7}$ may coincide (and none of them may coincide with any other second-level vertex) -- $u_{3}$ brings $\{u_{3}, b\}$ and $\{u_{3}, b_{1,1}\}$, $u_{5}$ brings $\{u_{5}, a\}$ and $\{u_{5}, a_{1,0}\}$ and $u_{7}$ brings $\{u_{7}, a\}$ and $\{u_{7}, a_{1,1}\}$. Of these, only $\{u_{5}, a\}$ and $\{u_{7}, a\}$ coincide if $u_{5} = u_{7}$. Among the remaining vertices $u_{1}$, $u_{2}$, $u_{4}$, $u_{6}$ and $u_{8}$, there are two inclusion-maximum sets $\left\{u_{1}, u_{4}, u_{6}, u_{8}\right\}$ and $\left\{u_{2}, u_{6}, u_{8}\right\}$ that may coincide. In the former case, $u_{6}$ brings $\{u_{6}, a_{1,1}\}$, $u_{8}$ brings $\{u_{8}, a_{0,1}\}$, $u_{1}$ brings $\{u_{1}, b_{0,1}\}$ and $u_{4}$ brings $\{u_{4}, b_{1,1}\}$, which are all distinct; additionally, the common vertex is adjacent to $x$. In the latter case, $u_{6}$ brings $\{u_{6}, a_{1,1}\}$, $u_{8}$ brings $\{u_{8}, a_{0,1}\}$ and $u_{2}$ brings $\{u_{2}, b\}$, which are all distinct; additionally, the common vertex is adjacent to $x$.
\end{proof}

Let $\nu_{2}$ denote the number of unique second-level vertices we add to the graph $H_{1}$ and $e$ the number of edges that are added to $H_{1}$ due to non-unique second-level vertices. Let the graph obtained after accounting for all second-level vertices be denoted $H_{2}$. 

\subsubsection{First subcase:} We consider $b_{1,1} \in V\left(H'_{0}\right) \setminus V(H_{0})$ and $v_{2}$ not old. If the first-level vertices bring precisely $2\nu_{1}+4$ edges together, and each distinct, unique second-level vertex brings precisely $2$ edges, then $\rho(H_{2})$ is at least
\begin{equation}\label{subcase_1_b_{1,1}_new}
\frac{17 + 2 + (2 \nu_{1}+4) + 2\nu_{2} + e}{12 + 1 + \nu_{1} + \nu_{2}} = \frac{23 + 2(\nu_{1} + \nu_{2}) + e}{13 + (\nu_{1} + \nu_{2})}.
\end{equation}
For this to be strictly less than $\frac{13}{7}$, we require $(\nu_{1} + \nu_{2}) + 7e < 8$. Thus, either $e = 1$ and $\nu_{1} + \nu_{2} = 0$, which is not possible since $\nu_{1} + \nu_{2} = 0$ implies $\nu_{1} = 0$, which implies that $v_{2}$ is old, contradicting our hypothesis, or else $e = 0$ and $\nu_{1} + \nu_{2} \leq 7$. In what follows, no two unique second-level vertices are allowed to coincide, nor any banal second-level vertex allowed to exist except when $v_{1}$ is not old and $v_{1} = u_{2}$, because of Lemma~\ref{unique_banal_2nd_level}. We now split the analysis into a few cases depending on the value of $\nu_{1}$, keeping in mind that $v_{2}$ is not old: 
\renewcommand{\theenumi}{(\roman{enumi})}
\begin{enumerate}
\item \label{nu_{1}=1} When $\nu_{1} = 1$, so that $\nu_{2} \leq 6$: First, suppose all of $v_{1}, v_{3}, v_{4}$ are old -- they together bring precisely $3$ edges only when $v_{3}, v_{4} \in S_{2}$ and $v_{1} = b_{1,1}$, and these edges are $\{v_{3}, a_{1,0}\}$, $\{v_{4}, a_{0,1}\}$, and $\{v_{1}, b_{0,1}\} = \{b_{1,1}, b_{0,1}\}$; the edge incident on $b_{1,1}$ and added due to $v_{2}$ is $\{v_{2}, b_{1,1}\}$. In $H_{1}$, the only neighbours of $a_{1,0}$ are $x$, $a_{1,0}^{1,0}$ and $v_{3}$, so that there exists no vertex that can play the role of $u_{5}$; the only neighbours of $b_{0,1}$ are $b$, $b_{0,1}^{0,1}$ and $v_{1}$, so that no vertex plays the role of $u_{1}$; the only neighbours of $a_{0,1}$ are $a$, $a_{0,1}^{0,1}$ and $v_{4}$, so that no vertex plays the role of $u_{8}$; the only neighbours of $b_{1,1}$ are $x$, $b$ and $v_{2}$, so that no vertex plays the role of $u_{4}$; the only common neighbours of $x$ and $b$ are $b_{1,1}$ and $v_{2}$, so that no vertex plays the role of $u_{2}$. Thus $u_{1}$, $u_{2}$, $u_{4}$, $u_{5}$, $u_{8}$ need to be added as distinct, unique second-level vertices. 

When not all of $v_{1}, v_{3}, v_{4}$ are old, in addition to $u_{1}$, $u_{2}$, $u_{4}$, $u_{5}$, $u_{8}$, we need to add at least one of $u_{3}$, $u_{6}$, $u_{7}$ as a unique, distinct second-level vertices (note that if $v_{1}$ is not old, it coincides with $v_{2}$ since $\nu_{1} = 1$, hence $u_{2}$ cannot coincide with $v_{1}$). 

When $v_{1}, v_{3}, v_{4}$ are old, we have room to add one new vertex and $2$ new edges to $H_{2}$; when $v_{1}, v_{3}, v_{4}$ are not all old, we can neither add a new vertex nor a new edge to $H_{2}$. We now consider all the possibilities for $b_{1,0}$:
\begin{itemize}
\item If $b_{1,0}$ coincides with $a_{1,0}$, since the only neighbours of $a_{1,0}$ are $x$, $a_{1,0}^{1,0}$ and $v_{3}$, hence no vertex in $H_{2}$ can play the role of $b_{1,0}^{1,1}$ or $b_{1,0}^{0,1}$ without adding further edges. We thus have to add both of these as new vertices.
\item If $b_{1,0}$ coincides with $u_{1}$, since the only neighbours of $u_{1}$ are $x$ and $b_{0,1}$, hence no vertex in $H_{2}$ can play the role of $b_{1,0}^{1,1}$ or $b_{1,0}^{1,0}$ without adding further edges. We thus have to add both of these as new vertices.
\item If $b_{1,0}$ coincides with $u_{4}$, since the only neighbours of $u_{4}$ are $b_{1,1}$ and $x$, hence no vertex in $H_{2}$ can play the role of $b_{1,0}^{1,0}$ or $b_{1,0}^{0,1}$ without adding further edges. We thus have to add both of these as new vertices.
\item If we add $b_{1,0}$ as a new vertex, it brings precisely $1$ edge, and we are allowed to add one more edge but no new vertices. This is not enough to account for all of $b_{1,0}^{1,1}$, $b_{1,0}^{1,0}$ and $b_{1,0}^{0,1}$. 
\end{itemize}
Thus it never suffices to add a single new vertex and $2$ new edges to $H_{2}$, thus violating the condition $\nu_{2} \leq 6$.

\item When $\nu_{1} = 2$: In this case, we need $\nu_{2} \leq 5$. If $v_{3}$ and $v_{4}$ are old, the only neighbours to $b_{1,1}$ in $H_{1}$ are $x$, $b$ and $v_{2}$, so that no vertex in $H_{1}$ plays the role of $u_{3}$ nor of $u_{4}$. If $v_{3}$ and $v_{1}$ are old, the only neighbours of $a_{1,1}$ are $x$, $a$, $a_{1,1}^{1,1}$ and possibly $a_{1,0}$ if $v_{3} = a_{1,1}$, so that no vertex in $H_{1}$ plays the role of $u_{7}$. If $v_{4}$ and $v_{1}$ are old, the only neighbours of $a_{1,1}$ are $x$, $a$, $a_{1,1}^{1,1}$ and possibly $a_{0,1}$ if $v_{4} = a_{1,1}$, so that no vertex in $H_{1}$ plays the role of $u_{6}$. Moreover, as in \ref{nu_{1}=1}, all of $u_{1}$, $u_{5}$, $u_{4}$, $u_{8}$ need to be added as unique, distinct second-level vertices in each of the above cases; when $v_{1}$ and $v_{3}$ are old or $v_{1}$ and $v_{4}$ are old, we also need to add $u_{2}$ as a new vertex. Clearly, this forces $\nu_{2} \geq 6$. If less than $2$ first-level vertices are old, even more second-level vertices need to be added as unique, distinct second-level vertices. 

\item When $\nu_{1} \geq 3$: In this case, we need $\nu_{2} \leq 4$. We argue as above that $u_{1}$, $u_{5}$, $u_{4}$, $u_{8}$ all need to be added as unique, distinct second-level vertices. If $v_{1}$ is old, then both $u_{6}$ and $u_{7}$ need to be added as distinct, unique second-level vertices; if $v_{3}$ or $v_{4}$ is old, we add both $u_{3}$ and $u_{4}$ as new second-level vertices. All these show that $u \leq 4$ is a condition that clearly cannot be met.

\end{enumerate}
\subsubsection{Second subcase:} We consider $b_{1,1} \in V\left(H'_{0}\right) \setminus V(H_{0})$ and $v_{2}$ old. By Lemma~\ref{old_first_level_subcase_1}, the number of edges added due to all first-level vertices is $2\nu_{1}+5$. Since each distinct unique second-level vertex brings at least $2$ edges, $\rho(H_{2})$ is at least
\begin{equation}
\frac{17 + 2 + (2 \nu_{1}+5) + 2\nu_{2} + e}{12 + 1 + \nu_{1} + \nu_{2}} = \frac{24 + 2(\nu_{1} + \nu_{2}) + e}{13 + (\nu_{1} + \nu_{2})},
\end{equation}
and for this to be strictly less than $\frac{13}{7}$, we require $\nu_{1}+\nu_{2} + 7e < 1$, implying $\nu_{1} = \nu_{2} = e = 0$. From the proof of Lemma~\ref{old_first_level_subcase_1}, we note that the number of edges contributed by all $4$ old first-level vertices is precisely $5$ only under one of the following situations.

If $v_{2} = a_{1,0}$ and $v_{3} = b_{1,1}$, the edges due to first-level vertices are $\{b_{1,1}, a\}$, $\{b_{1,1}, a_{1,0}\}$, $\{a_{1,0}, b\}$, $\{v_{1}, b_{0,1}\}$ and $\{v_{4}, a_{0,1}\}$ (if $v_{1} \neq b_{1,1}$, then we must have $v_{1} = a_{1,0}$). Note that the only neighbours of $b_{0,1}$ in $H_{1}$ are $b$, $b_{0,1}^{0,1}$ and $v_{1}$, all of which are adjacent to $b$, so that no vertex in $H_{1}$ plays the role of $u_{1}$. Thus we have to add $u_{1}$ as a unique second-level vertex, thus violating the required condition $\nu_{2} = 0$.

If $v_{3} \neq b_{1,1}$, the edges due to first-level vertices are $\{v_{3}, a_{1,0}\}$, $\{v_{2}, b\}$, $\{v_{2}, b_{1,1}\}$, $\{v_{1}, b_{0,1}\}$ and $\{v_{4}, a_{0,1}\}$ (if $v_{1} \neq b_{1,1}$, then we must have $v_{1} = v_{2}$). Once again, $u_{1}$ needs to be added as a unique second-level vertex, violating $\nu_{2} = 0$.

\subsubsection{Third subcase:} When $b_{1,1} \in V(H_{0})$, if each unique second-level vertex brings precisely $2$ edges, $\rho(H_{2})$ is at least
\begin{equation}
\frac{17 + 1 + (2 \nu_{1}+4) + 2\nu_{2} + e}{12 + \nu_{1} + \nu_{2}} = \frac{22 + 2(\nu_{1}+\nu_{2}) + e}{12 + (\nu_{1}+\nu_{2})}.
\end{equation}
For this to be strictly less than $\frac{13}{7}$, we require $(\nu_{1}+\nu_{2}) + 7e < 2$, forcing $e = 0$ and $\nu_{1}+\nu_{2} \leq 1$. The only neighbours of $a_{0,1}$ in $H_{1}$ are $a$, $a_{0,1}^{0,1}$ and $v_{4}$, and the only neighbours to $b_{0,1}$ are $b$, $b_{0,1}^{0,1}$ and $v_{1}$, so that there exist no vertices playing the roles of $u_{8}$ and $u_{1}$. Since both $u_{1}$ and $u_{8}$ need to be added as unique second-level vertices, they must coincide, but by Lemma~\ref{unique_banal_2nd_level}, the common vertex will bring $3$ edges instead of precisely $2$.


\subsection{Showing that for $\alpha < \frac{7}{13}$, zero-one law for $\mathcal{E}_{4}$ holds for $\left\{G(n, n^{-\alpha})\right\}_{n}$:} For $\alpha < \frac{7}{13}$, when $G_{1} \sim G(m, m^{-\alpha})$ and $G_{2} \sim G(n, n^{-\alpha})$ are independent, we show that a.a.s.\ Duplicator wins $\EHR[G_{1}, G_{2}, 4]$. Let Spoiler play on $G_{1}$ and Duplicator on $G_{2}$ without loss of generality, and the vertices picked in $G_{1}$ are denoted $x_{i}$ and in $G_{2}$ by $y_{i}$, for $i \in [4]$. A.a.s.\ as $m, n \rightarrow \infty$, both $G_{1}$ and $G_{2}$ contain induced subgraphs isomorphic to $G_{0}$, and Duplicator makes use of the copy of $G_{0}$ present inside $G_{2}$ in her winning strategy. Given our description of $G_{0}$, it is clear, for the most part (in particular, for the first $2$ rounds), how Duplicator responds to Spoiler while playing on the copy of $G_{0}$ inside $G_{2}$. We point out here some of the less obvious responses.  

\begin{lemma}\label{k=4_lem_1}
A.a.s.\ for every subgraph of $G_{2}$ that is isomorphic to $G_{0}$, and whose vertices are named the same as the corresponding vertices of $G_{0}$, there exist the following vertices, edges and non-edges in $G_{2}$:
\renewcommand{\theenumi}{(\roman{enumi})}
\begin{enumerate}
\item vertex $c$ that is non-adjacent to $x$, $a$ and $a'_{1,1}$ but adjacent to $a_{1,1}$;
\item vertex $c^{1,0}$ that is adjacent to $c$ and $x$ but not to $a$;
\item vertex $c^{0,1}$ that is adjacent to $c$ and $a$ but not to $x$.

\item vertex $d$ that is non-adjacent to $x$, $b$ and $b'_{1,1}$ but adjacent to $b_{1,1}$;
\item vertex $d^{1,0}$ that is adjacent to $d$ and $x$ but not to $b$;
\item vertex $d^{0,1}$ that is adjacent to $d$ and $b$ but not to $x$.
\end{enumerate}

\end{lemma}

\begin{proof}
Let $H$ be the induced subgraph on $x, a, a_{1,1}, a'_{1,1}$ and $G$ that on $x, a, a_{1,1}, a'_{1,1}, c, c^{1,0}, c^{0,1}$ in $G_{2}$ such that $E(G) \setminus E(H)$ comprises precisely the edges and non-edges described in the statement of Lemma~\ref{k=4_lem_1}. Consider $H \subset S \subseteq G$ with $c \in S$. It is straightforward to see that $e(S, H)$ is at most $\frac{5}{3}$ times $v(S, H)$, and hence $f_{\alpha}(S, H) > 0$ for $\alpha < \frac{7}{13}$. The argument is similar for the vertices $d$, $d^{1,0}$ and $d^{0,1}$. From Theorem~\ref{safe_extension_copies}, the rest follows.


\end{proof}

Suppose Spoiler picks $x_{3}$ that is non-adjacent to $x_{1}$ and $x_{2}$ -- then Duplicator sets $y_{3} = c$ if $y_{2} = a$, and $y_{3} = d$ otherwise (i.e.\ if $y_{2} = b$). If Spoiler selects $x_{4}$ that is adjacent to $x_{1}, x_{2}, x_{3}$, then Duplicator sets $y_{4} = a_{1,1}$ if $y_{2} = a$, and $y_{4} = b_{1,1}$ otherwise. If Spoiler selects $x_{4}$ that is adjacent to $x_{1}$, $x_{2}$ but not to $x_{3}$, Duplicator sets $y_{4} = a'_{1,1}$ when $y_{2} = a$ and $y_{4} = b'_{1,1}$ otherwise. If Spoiler selects $x_{4}$ that is adjacent to $x_{1}$, $x_{3}$ but not to $x_{2}$, then Duplicator sets $y_{4} = c^{1,0}$ when $y_{2} = a$ and $y_{4} = d^{1,0}$ otherwise. If Spoiler selects $x_{4}$ that is adjacent to $x_{2}$, $x_{3}$ but not to $x_{1}$, then Duplicator sets $y_{4} = c^{0,1}$ when $y_{2} = a$ and $y_{4} = d^{0,1}$ otherwise. Finally, no matter what $x_{1}, x_{2}, x_{3}$ may have been, suppose Spoiler selects $x_{4}$ that is adjacent to at most one of $x_{1}, x_{2}, x_{3}$. If $G$ is the subgraph of $G_{1}$ induced on $\{x_{1}, x_{2}, x_{3}, x_{4}\}$ and $H$ that on $\{x_{1}, x_{2}, x_{3}\}$, the pair $(G, H)$ is $\alpha$-safe for all $\alpha < 1$ -- hence, by Theorem~\ref{safe_extension_copies}, Duplicator a.a.s.\ finds $y_{4}$ in $G_{2}$ such that $y_{4} \sim y_{i}$ iff $x_{4} \sim x_{i}$ for $i \in [3]$.

\newpage
\section{Appendix}\label{Appendix}
Here we compile the remaining cases to complete the proof of \ref{claim_2} of Section~\ref{k=4}. At the very outset, we recall the sets $S_{1}$ and $S_{2}$ defined in \eqref{S_{1}} and \eqref{S_{2}} respectively. In Remarks~\ref{app_rem_1} through \ref{app_rem_5}, we collect all the possibilities for $b_{1,1}$ and the first-level vertices that we consider in \S~\ref{subcase_2} through \S~\ref{subcase_13}. 

\begin{rem}\label{app_rem_1}
If $b_{1,1} \in V\left(H'_{0}\right) \setminus V(H_{0})$, we add edges $\{b_{1,1}, x\}$ and $\{b_{1,1}, b\}$, and if $b_{1,1} \in S_{1} \setminus \{a\}$, we add $\{b_{1,1}, b\}$. In \S~\ref{subcase_2}, \S~\ref{subcase_3}, \S~\ref{subcase_4}, \S~\ref{subcase_6}, \S~\ref{subcase_8} and \S~\ref{subcase_12}, we do not add any edge if $b_{1,1} = a$, otherwise we add edge $\{a, b\}$.
\end{rem}

\begin{rem}\label{app_rem_2}
In \S~\ref{subcase_3} and \S~\ref{subcase_6}, we do not consider $v_{1}$ at all, whereas in \S~\ref{subcase_5}, \S~\ref{subcase_7}, \S~\ref{subcase_9} and \S~\ref{subcase_13}, we do not consider $v_{1}$ when $b_{1,1} = a$, since in all these scenarios, $a$ plays the role of $v_{1}$. What follows excludes these cases. If $v_{1} \in S_{1} \setminus \{a, b_{1,1}\}$, we add edges $\{v_{1}, b\}$ and $\{v_{1}, b_{0,1}\}$, and if $v_{1} = b_{1,1}$, we add $\{b_{1,1}, b_{0,1}\}$. In \S~\ref{subcase_2}, \S~\ref{subcase_4}, \S~\ref{subcase_8} and \S~\ref{subcase_12}, we add edge $\{a, b_{0,1}\}$ when $v_{1} = a$, in \S~\ref{subcase_5}, \S~\ref{subcase_7}, \S~\ref{subcase_9} and \S~\ref{subcase_13}, we add edge $\{a, b\}$ when $v_{1} = a$, and in \S~\ref{subcase_10} and \S~\ref{subcase_11}, we add both $\{a, b\}$ and $\{a, b_{0,1}\}$ when $v_{1} = a$.
\end{rem}

\begin{rem}\label{app_rem_3}
In \S~\ref{subcase_2}, \S~\ref{subcase_3}, \S~\ref{subcase_4}, \S~\ref{subcase_6}, \S~\ref{subcase_8} and \S~\ref{subcase_12}, we do not consider $v_{2}$ whenever $b_{1,1} \in S_{2}$, since $a$ plays the role of $v_{2}$; moreover, in these subsections, we add only edge $\{a, b_{1,1}\}$ whenever $b_{1,1} \in V\left(H'_{0}\right) \setminus V(H_{0})$ and $v_{2} = a$. For all other cases, we have:
\renewcommand{\theenumi}{(\alph{enumi})}
\begin{enumerate}
\item if $b_{1,1} \in V\left(H'_{0}\right) \setminus V(H_{0})$, $v_{2} \in S_{1}$ and we add $\{v_{2}, b\}$ and $\{v_{2}, b_{1,1}\}$;
\item if $b_{1,1} = a$, $v_{2} \in S_{2}$ and we add $\{v_{2}, b\}$;
\item if $b_{1,1} = a_{1,1}$, either $v_{2} \in \left\{a, a_{1,1}^{1,1}\right\}$ and add $\{v_{2}, b\}$, or $v_{2} \in \left\{a_{1,0}, a_{1,0}^{1,0}\right\}$ and we add $\{v_{2}, a_{1,1}\}$ and $\{v_{2}, b\}$;
\item if $b_{1,1} = a_{1,1}^{1,1}$, either $v_{2} \in \left\{a, a_{1,1}\right\}$ and add $\{v_{2}, b\}$, or $v_{2} \in \left\{a'_{1,1}, a_{1,0}, a_{1,0}^{1,0}\right\}$ and we add $\{v_{2}, b\}$ and $\left\{v_{2}, a_{1,1}^{1,1}\right\}$;
\item if $b_{1,1} = a'_{1,1}$, either $v_{2} = a$ and we add $\{a, b\}$, or $v_{2} \in \left\{a_{1,1}^{1,1}, a_{1,0}, a_{1,0}^{1,0}\right\}$ and add $\{v_{2}, b\}$ and $\{v_{2}, a'_{1,1}\}$;
\item if $b_{1,1} = a_{1,0}$, either $v_{2} = a_{1,0}^{1,0}$ and we add $\left\{a_{1,0}^{1,0}, b\right\}$, or $v_{2} \in S_{2}$ and we add $\{v_{2}, a_{1,0}\}$ and $\{v_{2}, b\}$;
\item if $b_{1,1} = a_{1,0}^{1,0}$, either $v_{2} = a_{1,0}$ and we add $\{a_{1,0}, b\}$, or $v_{2} \in S_{2}$ and we add $\left\{v_{2}, a_{1,0}^{1,0}\right\}$ and $\{v_{2}, b\}$.

\end{enumerate}

\end{rem}

\begin{rem}\label{app_rem_4}
If $v_{3}$ is old, either $v_{3} \in S_{2}$, in which case we add the edge $\{v_{3}, a_{1,0}\}$, or else $v_{3} = b_{1,1}$ when $b_{1,1} \in V\left(H'_{0}\right) \setminus V(H_{0})$, and we add the edges $\{b_{1,1}, a\}$ and $\{b_{1,1}, a_{1,0}\}$.
\end{rem}

\begin{rem}\label{app_rem_5}
In \S~\ref{subcase_2}, \S~\ref{subcase_3} and \S~\ref{subcase_4}, we do not consider $v_{4}$ whenever $b_{1,1} \in S_{2}$, since $b_{1,1}$ plays the role of $v_{4}$. In all other cases, if $v_{4} \in S_{2}$ and we add edge $\{v_{4}, b\}$ in \S~\ref{subcase_2} through \S~\ref{subcase_4}, $\{v_{4}, b_{0,1}\}$ in \S~\ref{subcase_5} through \S~\ref{subcase_7}, $\{v_{4}, a_{0,1}\}$ in \S~\ref{subcase_8} through \S~\ref{subcase_10}, and $\left\{v_{4}, b_{0,1}^{0,1}\right\}$ in \S~\ref{subcase_11} through \S~\ref{subcase_13}. If $b_{1,1} \in V\left(H'_{0}\right) \setminus V(H_{0})$ and $v_{4} = b_{1,1}$, we add only the edge $\{a, b_{1,1}\}$ in \S~\ref{subcase_2} through \S~\ref{subcase_4}, edges $\{a, b_{1,1}\}$ and $\{b_{1,1}, b_{0,1}\}$ in \S~\ref{subcase_5} through \S~\ref{subcase_7}, edges $\{a, b_{1,1}\}$ and $\{b_{1,1}, a_{0,1}\}$ in \S~\ref{subcase_8} through \S~\ref{subcase_10}, and edges $\{a, b_{1,1}\}$ and $\left\{b_{1,1}, b_{0,1}^{0,1}\right\}$ in \S~\ref{subcase_11} through \S~\ref{subcase_13}.

\end{rem}

\subsection{Second case:}\label{subcase_2}
Here we assume that $a_{0,1} = b$ and $a_{0,1}^{0,1}$ is added as a new vertex; we then add edges $\{a, b\}$, $\left\{b, a_{0,1}^{0,1}\right\}$, $\left\{a_{0,1}^{0,1}, a\right\}$ to get $H_{0}$ in Figure~\ref{H_{0}_subcase_2} with $17$ edges and $11$ vertices. We split the analysis according to $b_{1,1}$.
\begin{figure}[h!]
  \begin{center}
   \includegraphics[width=0.5\textwidth]{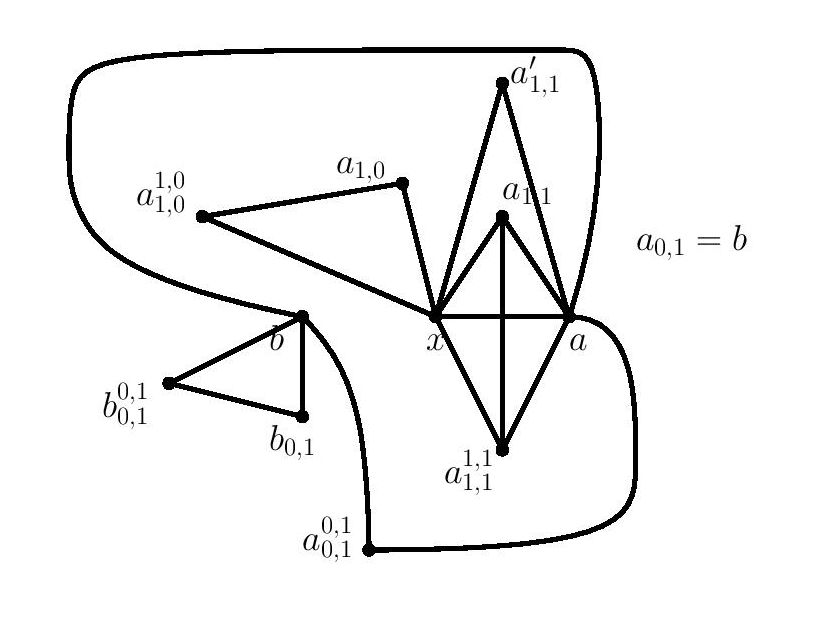}
\end{center}
  \caption{}
\label{H_{0}_subcase_2}
\end{figure}

\subsubsection{When $b_{1,1} \in V\left(H'_{0}\right) \setminus V(H_{0})$:}\label{1_subcase_2} Neither the edge $\left\{v_{3}, a_{1,0}\right\}$ that is added due to an old $v_{3}$ nor the edge $\{v_{1}, b_{0,1}\}$ added due to an old $v_{1}$ coincides with any other edge added due to old first-level vertices. Only when $v_{2}$ and $v_{4}$ are simultaneously old and $v_{2} = a$ and $v_{4} = b_{1,1}$, a single edge is added due to $2$ old first-level vertices. If $2$ first-level vertices are banal and coincide, the common vertex brings at least $4$ edges; if $3$ of them are banal and coincide, the common vertex brings at least $5$ edges; if all of them are banal and coincide, the common vertex brings at least $6$ edges. 

As before, $\nu_{1}$ denotes the number of distinct new first-level vertices. When $\nu_{1} \leq 2$, the number of edges added due to all first-level vertices is $2\nu_{1}+3$ and not higher only if both $v_{2}$ and $v_{4}$ are old with $v_{2} = a$ and $v_{4} = b_{1,1}$, and either $v_{1}$ is new or $v_{1} \in \{a, b_{1,1}\}$ (if not, we get two edges $\{v_{1}, b\}$ and $\{v_{1}, b_{0,1}\}$ from $v_{1}$, making the number of edges due to all first-level vertices at least $2\nu_{1}+4$). Therefore, in $H_{1}$, the graph obtained after accounting for all first-level vertices, $b_{0,1}$ has neighbours $b$, $b_{0,1}^{0,1}$ and $v_{1}$, so that there exists no vertex in $H_{1}$ that is adjacent to $x$ and $b_{0,1}$ but not to $b$. Similarly, $a_{1,0}$ has neighbours $x$, $a_{1,0}^{1,0}$ and $v_{3}$, so that there exists no vertex in $H_{1}$ that is adjacent to $a$ and $a_{1,0}$ but not to $x$. Hence the second-level vertices $u_{1}$ and $u_{5}$ cannot be found in $H_{1}$ without adding further edges. Now, the density of $H_{2}$, the graph obtained after accounting for all second-level vertices, is 
\begin{equation}
\rho(H_{2}) \geq \frac{17 + 2 + (2\nu_{1}+3) + 2\nu_{2} + e}{11 + 1 + \nu_{1} + \nu_{2}} = \frac{22 + 2(\nu_{1}+\nu_{2}) + e}{12 + (\nu_{1}+\nu_{2})},
\end{equation}
and for this to be less than $\frac{13}{7}$, we require $(\nu_{1}+\nu_{2}) + 7e < 2$, which implies $e = 0$ and $\nu_{1}+\nu_{2} \leq 1$. The condition $e = 0$ implies both $u_{1}$ and $u_{5}$ need to be added as unique second-level vertices, and the condition $\nu_{1}+\nu_{2} \leq 1$ implies that at most one of them can be added as a unique second-level vertex. Thus we arrive at a contradiction.

When $\nu_{1} \geq 3$, the number of edges added due to all first-level vertices is at least $2\nu_{1}+4$, making $\rho(H_{2})$ at least
\begin{equation} 
\frac{17 + 2 + 2\nu_{1}+4 + 2\nu_{2} + e}{11 + 1 + \nu_{1} + \nu_{2}} = \frac{23 + 2(\nu_{1}+\nu_{2}) + e}{12 + (\nu_{1}+\nu_{2})} \geq \frac{23}{12} > \frac{13}{7}.
\end{equation}


\subsubsection{When $b_{1,1} = a$:}\label{2_subcase_2} Once again, neither edge $\{v_{3}, a_{1,0}\}$ added if $v_{3}$ is old nor edge $\{v_{1}, b_{0,1}\}$ added if $v_{1}$ is old coincides with any other edge added due to old first-level vertices. When $v_{2}$ and $v_{4}$ are both old and $v_{2} = v_{4} = v$, a single edge $\{v, b\}$ is added due to them. 



If $v_{2}$ and $v_{4}$ are banal and coincide, the common vertex is adjacent to $x$, $b$, $a$; any other pair of coinciding banal first-level vertices bring at least $4$ edges. If $v_{1}, v_{2}, v_{4}$ are banal and coincide, the common vertex is adjacent to $x$, $b$, $b_{0,1}$, $a$; if $v_{2}, v_{3}, v_{4}$ are banal and coincide, the common vertex is adjacent to $x$, $b$, $a$, $a_{1,0}$; all other triplets of coinciding banal first-level vertices bring at least $5$ edges. If $v_{1}, v_{2}, v_{3}, v_{4}$ are banal and coincide, the common vertex is adjacent to $x$, $b$, $b_{0,1}$, $a$, $a_{1,0}$. 

These show that the number of edges added due to all first-level vertices is $2\nu_{1}+3$ and not higher when $\nu_{1} \leq 3$, $v_{2}$ and $v_{4}$ coincide, whether banal or old, and $v_{1}$ is either new or $v_{1} = a$ or $v_{1} = v_{2} = v_{4}$. Thus 
\begin{equation}
\rho(H_{2}) \geq \frac{17 + (2\nu_{1}+3) + 2\nu_{2} + e}{11 + \nu_{1} + \nu_{2}} = \frac{20 + 2(\nu_{1}+\nu_{2}) + e}{11 + (\nu_{1}+\nu_{2})},
\end{equation}
and for this to be less than $\frac{13}{7}$, we need $(\nu_{1}+\nu_{2}) + 7e < 3$, which implies $e = 0$ and $\nu_{1}+\nu_{2} \leq 2$. In $H_{1}$, the only neighbours of $b_{0,1}$ are $b$, $b_{0,1}^{0,1}$ and $v_{1}$, so that no vertex plays the role of $u_{1}$; the only neighbours of $a_{1,0}$ are $x$, $a_{1,0}^{1,0}$ and $v_{3}$, so that no vertex plays the role of $u_{5}$; the only neighbours of $b$ are $b_{0,1}$, $b_{0,1}^{0,1}$, $a_{0,1}^{0,1}$, $a$ and $v$, where $v_{2} = v_{4} = v$ either belongs to $S_{2}$ or to $V\left(H_{1}\right) \setminus V\left(H'_{0}\right)$, and hence there exists no vertex in $H_{1}$ that plays the role of $u_{8}$ which is adjacent to $x$ and $a_{0,1} = b$ but not to $a$. Thus none of $u_{1}$, $u_{5}$, $u_{8}$ can be found in $H_{1}$ without adding further edges. The condition $e = 0$ implies that all of them need to be added as unique second-level vertices, which violates the condition $\nu_{1} + \nu_{2} \leq 2$.

When $\nu_{1} = 4$, we get at least $12$ edges, making $\rho(H_{2})$ at least
\begin{equation}
\frac{17 + 12 + 2\nu_{2} + e}{11 + 4 + \nu_{2}} = \frac{29 + 2\nu_{2} + e}{15 + \nu_{2}} \geq \frac{29}{15} > \frac{13}{7}.
\end{equation}

\subsubsection{When $b_{1,1} = a_{1,0}$:}\label{3_subcase_2} Any two old first-level vertices bring at least $2$ edges. If $v_{2}, v_{3}, v_{4}$ are old, they bring $2$ edges and not more only if $v_{2} = v_{3} = v_{4} = v$ for some $v \in S_{2}$, and these edges are $\{v, b\}$ and $\{v, a_{1,0}\}$. Any other triplet of old first-level vertices bring at least $3$ edges. If all first-level vertices are old, they bring $3$ edges and not more only in the following two scenarios:
\renewcommand{\labelenumi}{\theenumi}
\renewcommand{\theenumi}{\roman{enumi})}
\begin{enumerate*}
\item either $v_{2} = v_{3} = v_{4} = v_{1} = v$ for some $v \in S_{2}$, and the edges added are $\{v, b\}$, $\{v, a_{1,0}\}$ and $\{v, b_{0,1}\}$;
\item or $v_{2} = v_{3} = v_{4} = v$ for some $v \in S_{2}$ and $v_{1} \in \{a, a_{1,0}\}$, and the edges added are $\{v, b\}$, $\{v, a_{1,0}\}$ and $\{v_{1}, b_{0,1}\}$.
\end{enumerate*}
Any $2$ coinciding banal first-level vertices bring at least $4$ edges; if $v_{2}$, $v_{3}$ and $v_{4}$ are banal and coincide, the common vertex is adjacent to $x$, $a$, $b$, $a_{1,0}$, whereas any other triplet of coinciding banal first-level vertices bring at least $5$ edges; if all are banal and coincide, the common vertex is adjacent to $x$, $a$, $b$, $a_{1,0}$ and $b_{0,1}$. Thus the first-level vertices contribute at least $2\nu_{1}+3$ edges when $\nu_{1} \leq 2$ and at least $2\nu_{1}+4$ when $\nu_{1} \geq 3$, making $\rho(H_{2})$ at least
\begin{equation}
\frac{17 + 1 + 2\nu_{1}+3 + 2\nu_{2} + e}{11 + \nu_{1} + \nu_{2}} = \frac{21 + 2(\nu_{1} + \nu_{2}) + e}{11 + (\nu_{1} + \nu_{2})} \geq \frac{21}{11} > \frac{13}{7}.
\end{equation}

\subsubsection{When $b_{1,1} = a_{1,0}^{1,0}$:}\label{4_subcase_2} The edge $\{v_{3}, a_{1,0}\}$ is added when $v_{3}$ is old; the edge $\{v_{1}, b_{0,1}\}$ is added when $v_{1}$ is old; the edge $\{v_{4}, b\}$ is added when $v_{4}$ is old. These edges are all distinct (even if $v_{1} = a_{1,0}$, the edges $\{v_{3}, a_{1,0}\}$ and $\{v_{1}, b_{0,1}\}$ do not coincide since $v_{3} \in S_{2}$ and thus $v_{3} \neq b_{0,1}$). Finally, if $v_{2} \in S_{2}$, it brings both the edges $\{v_{2}, b\}$ and $\left\{v_{2}, a_{1,0}^{1,0}\right\}$, so that even if $v_{2} = v_{4}$, the edge $\left\{v_{2}, a_{1,0}^{1,0}\right\}$ does not coincide with any of the previously mentioned edges; if $v_{2} = a_{1,0}$, it brings edge $\{a_{1,0}, b\}$, which does not coincide with $\{v_{4}, b\}$ since $v_{4} \in S_{2}$, does not coincide with $\{v_{3}, a_{1,0}\}$ since $v_{3} \in S_{2}$, and clearly $\{v_{1}, b_{0,1}\} \neq \{a_{1,0}, b\}$. 

The number of edges added due to all first-level vertices is at least $2\nu_{1}+4$, so that
\begin{equation}
\rho(H_{2}) \geq \frac{17 + 1 + (2\nu_{1}+4) + 2\nu_{2} + e}{11 + \nu_{1} + \nu_{2}} = \frac{22 + 2(\nu_{1}+\nu_{2}) + e}{11 + (\nu_{1}+\nu_{2})} \geq 2 > \frac{13}{7}.
\end{equation}

\subsubsection{When $b_{1,1} \in S_{2}$:}\label{5_subcase_2} Here we only consider $v_{1}$ and $v_{3}$. The number of edges due to them is $2\nu_{1}+2$ and not higher only if $v_{1}$ is new of $v_{1} \in \left\{a, b_{1,1}\right\}$, giving us the edges $\{v_{1}, b_{0,1}\}$ and $\{v_{3}, a_{1,0}\}$. Then
\begin{equation}
\rho(H_{2}) \geq \frac{17 + 1 + (2\nu_{1}+2) + 2\nu_{2} + e}{11 + \nu_{1} + \nu_{2}} = \frac{20 + 2(\nu_{1} + \nu_{2}) + e}{11+(\nu_{1}+\nu_{2})}.
\end{equation}
For this to be less than $\frac{13}{7}$, we need $\nu_{1} + \nu_{2} + 7e < 3$, implying that $\nu_{1}+\nu_{2} \leq 2$ and $e = 0$. The only neighbours of $b_{0,1}$ in $H_{1}$ are $b$, $b_{0,1}^{0,1}$ and $v_{1}$; the only neighbours of $a_{1,0}$ in $H_{1}$ are $x$, $a_{1,0}^{1,0}$ and $v_{3}$; the only common neighbours of $x$ and $b$ in $H_{1}$ are $b_{1,1}$ and $a$. Thus there exist no vertices in $H_{1}$ that play the roles of $u_{1}$, $u_{5}$ and $u_{2}$ without adding further edges, forcing us to add them as distinct unique second-level vertices and violating the condition $\nu_{1} + \nu_{2} \leq 2$. 

\subsection{Third case:}\label{subcase_3}
We assume that $a_{0,1} = b$ and $a_{0,1}^{0,1} = b_{0,1}$, and add edges $\{a, b\}$ and $\left\{a, b_{0,1}\right\}$ to get $H_{0}$ in Figure~\ref{H_{0}_subcase_3} with $16$ edges and $10$ vertices. 
\begin{figure}[h!]
  \begin{center}
   \includegraphics[width=0.5\textwidth]{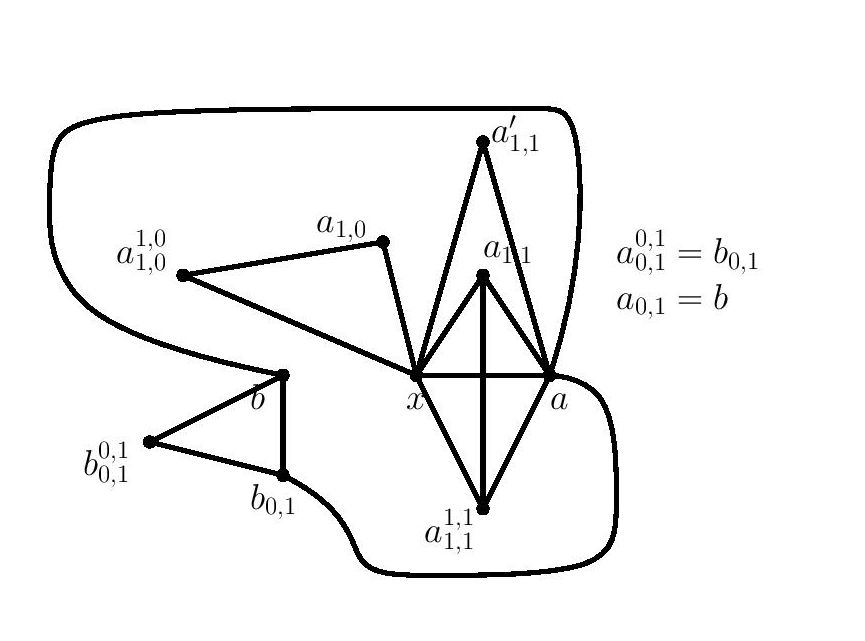}
\end{center}
  \caption{}
\label{H_{0}_subcase_3}
\end{figure}

\subsubsection{When $b_{1,1} \in V\left(H'_{0}\right) \setminus V(H_{0})$:} The only scenario where $2$ old first-level vertices contribute a single edge is when $v_{4} = b_{1,1}$ and $v_{2} = a$ and this edge is $\{a, b_{1,1}\}$; if old, $v_{3}$ brings at least the edge $\{v_{3}, a_{1,0}\}$. Any $2$ coinciding banal first-level vertices bring at least $4$ edges; if all $3$ are banal and coincide, they bring $6$ edges. Thus the first-level vertices contribute $2\nu_{1}+2$ edges and no more only when $\nu_{1} \leq 1$ and $v_{2} = a$, $v_{4} = b_{1,1}$, making $\rho(H_{2})$ at least
\begin{equation}
\frac{16 + 2 + 2\nu_{1}+2 + 2\nu_{2} + e}{10 + 1 + \nu_{1} + \nu_{2}} = \frac{20 + 2(\nu_{1} + \nu_{2}) + e}{11 + (\nu_{1} + \nu_{2})},
\end{equation}
which is less than $\frac{13}{7}$ only if $(\nu_{1}+\nu_{2}) + 7e < 3$, implying $e = 0$ and $\nu_{1}+\nu_{2} \leq 2$. The only neighbours of $a_{1,0}$ in $H_{1}$ are $x$, $a_{1,0}^{1,0}$ and $v_{3}$; the only neighbours of $b_{0,1}$ in $H_{1}$ are $a$, $b$ and $b_{0,1}^{0,1}$; the only common neighbours of $x$ and $b$ are $b_{1,1}$ and $a$ with $a \sim b_{1,1}$ in $H_{1}$. Thus there exist no vertices in $H_{1}$ playing the roles of $u_{1}$, $u_{5}$ and $u_{2}$, requiring them to be added as distinct unique second-level vertices, violating $\nu_{1}+\nu_{2} \leq 2$. When $\nu_{1} \geq 2$, we get at least $2\nu_{1}+3$ edges from first-level vertices, making $\rho(H_{2}) \geq \frac{21}{11} > \frac{13}{7}$.

\subsubsection{When $b_{1,1} = a$:} The edge added due to $v_{2}$ being old is $\{v_{2}, b\}$, that due to $v_{3}$ being old is $\{v_{3}, a_{1,0}\}$, and that due to $v_{4}$ being old is $\{v_{4}, b\}$. If $v_{2}$ and $v_{4}$ are banal and coincide, the common vertex is adjacent to $x$, $b$, $a$; if $v_{2}, v_{3}, v_{4}$ are banal and coincide, the common vertex is adjacent to $x$, $b$, $a$, $a_{1,0}$. The first-level vertices thus contribute $2\nu_{1}+2$ edges and no more only when $v_{2} = v_{4}$, whether old or banal. Thus
\begin{equation}
\rho(H_{2}) \geq \frac{16 + 2\nu_{1}+2 + 2\nu_{2} + e}{10 + \nu_{1} + \nu_{2}} = \frac{18 + 2(\nu_{1} + \nu_{2}) + e}{10 + (\nu_{1} + \nu_{2})},
\end{equation}
which is less than $\frac{13}{7}$ only if $(\nu_{1}+\nu_{2}) + 7e < 4$, implying $e = 0$ and $\nu_{1}+\nu_{2} \leq 3$. The only neighbours of $b_{0,1}$ in $H_{1}$ are $b$, $b_{0,1}^{0,1}$ and $a$, where both $b_{0,1}^{0,1}$ and $a$ are adjacent to $b$, hence there exists no vertex in $H_{1}$ playing the role of $u_{1}$. The only neighbours of $a_{1,0}$ are $x$, $a_{1,0}^{1,0}$ and $v_{3}$, where both $a_{1,0}^{1,0}$ and $v_{3}$ are adjacent to $x$, hence there exists no vertex in $H_{1}$ playing the role of $u_{5}$. The only common neighbours of $x$ and $b$ in $H_{1}$ are $a$ and $v_{2} = v_{4}$ which is adjacent to $b_{1,1} = a$, and hence there exists no vertex in $H_{1}$ playing the role of $u_{2}$. Since $u_{1}$, $u_{2}$, $u_{5}$ need to be added as distinct unique second-level vertices, all first-level vertices and all other second-level vertices must be old.

The only way both $u_{6}$ and $u_{7}$ are old is if $v_{2} = v_{3} = v_{4} = a_{1,1}$, in which case $u_{6} = a_{1,0}$ and $u_{7} = b$; also, $u_{3} = b_{0,1}$ and $u_{4} \in \left\{a'_{1,1}, a_{1,1}^{1,1}\right\}$ will work, whereas $u_{8} = u_{2}$. We must have $b_{1,0} \in \left\{a'_{1,1}, a_{1,1}^{1,1}, a_{1,0}, a_{1,0}^{1,0}, u_{1}\right\}$. If $b_{1,0} = a'_{1,1}$, its neighbours are $x$ and $a$, and if $b_{1,0} = a_{1,1}^{1,1}$, its neighbours are $x$, $a$ and $a_{1,1}$, hence there exist no vertices in $H_{2}$ playing the roles of $b_{1,0}^{1,0}$ and $b_{1,0}^{0,1}$. If $b_{1,0} = a_{1,0}$, its neighbours are $u_{5}$, $a_{1,1}$, $x$ and $a_{1,0}^{1,0}$, there exists no vertex playing the role of $b_{1,0}^{0,1}$. If $b_{1,0} = a_{1,0}^{1,0}$, its neighbours are $x$ and $a_{1,0}$, hence there exist no vertices playing the roles of $b_{1,0}^{1,1}$ and $b_{1,0}^{0,1}$. If $b_{1,0} = u_{1}$, its neighbours are $x$ and $b_{0,1}$, hence there exist no vertices playing the roles of $b_{1,0}^{1,1}$ and $b_{1,0}^{1,0}$.

\subsubsection{When $b_{1,1} = a_{1,0}$:} The first-level vertices contribute $2\nu_{1}+2$ edges only if $v_{2} = v_{3} = v_{4}$, whether old or banal; else at least $2\nu_{1}+3$ edges. Thus
\begin{equation}
\rho(H_{2}) \geq \frac{16 + 1 + (2\nu_{1}+2) + 2\nu_{2} + e}{10 + \nu_{1} + \nu_{2}} = \frac{19 + 2(\nu_{1}+\nu_{2}) + e}{10 + \nu_{2}} \geq \frac{19}{10} > \frac{13}{7}.
\end{equation}


\subsubsection{When $b_{1,1} = a_{1,0}^{1,0}$:} The first-level vertices contribute at least $2\nu_{1}+3$ edges, so that
\begin{equation}
\rho(H_{2}) \geq \frac{16 + 1 + 2\nu_{1}+3 + 2\nu_{2} + e}{10 + \nu_{1} + \nu_{2}} = \frac{20 + 2(\nu_{1} + \nu_{2}) + e}{10 + (\nu_{1}+\nu_{2})} \geq 2 > \frac{13}{7}.
\end{equation}

\subsubsection{When $b_{1,1} \in S_{2}$:} We only consider $v_{3}$, which, if old, brings $1$ edge, and if new, brings $3$ edges. In the former scenario, if each unique second-level vertex brings precisely $2$ edges, we have
\begin{equation}
\rho(H_{2}) \geq \frac{16 + 1 + 1 + 2\nu_{2} + e}{10 + \nu_{2}} = \frac{18 + 2\nu_{2} + e}{10 + \nu_{2}}.
\end{equation}
For this to be less than $\frac{13}{7}$, we need $\nu_{2} + 7e < 4$, implying $e = 0$ and $\nu_{2} \leq 3$. As above, we argue that $u_{1}$, $u_{2}$, $u_{5}$ need to be added as distinct, unique second-level vertices, which leaves no room for any other new vertex nor new edge. If $u_{6}$ and $u_{7}$ are to exist in $H_{1}$, we need $b_{1,1} = v_{3} = a_{1,1}$, giving $u_{6} = a_{1,0}$ and $u_{7} = b$; we set $u_{8} = u_{2}$ and $u_{4} = a_{1,1}^{1,1}$. But the only common neighbour of $b$ and $b_{1,1} = a_{1,1}$ in $H_{1}$ is $a$, which is adjacent to $x$, hence $u_{3}$ needs to be added as a unique second-level vertex. If $u_{3}$ coincides with $u_{5}$, then the common unique second-level vertex brings $4$ edges, which violates our assumption above.

When $v_{3}$ is a unique first-level vertex, 
\begin{equation}
\rho(H_{2}) \geq \frac{16 + 1 + 3 + 2\nu_{2} + e}{10 + 1 + \nu_{2}} = \frac{20 + 2\nu_{2} + e}{11 + \nu_{2}}.
\end{equation}
For this to be less than $\frac{13}{7}$, we need $\nu_{2} + 7e < 3$, forcing $e = 0$ and $\nu_{2} \leq 2$. As above, we need to add $u_{1}$, $u_{2}$, $u_{5}$ as distinct, unique second-level vertices, violating the condition $\nu_{2} \leq 2$.

\subsection{Fourth case:}\label{subcase_4}
We assume that $a_{0,1} = b$ and $a_{0,1}^{0,1} = b_{0,1}^{0,1}$, adding edges $\left\{a, b\right\}$ and $\left\{a, b_{0,1}^{0,1}\right\}$ to get$H_{0}$ in Figure~\ref{H_{0}_subcase_4} with $16$ edges and $10$ vertices.
\begin{figure}[h!]
  \begin{center}
   \includegraphics[width=0.5\textwidth]{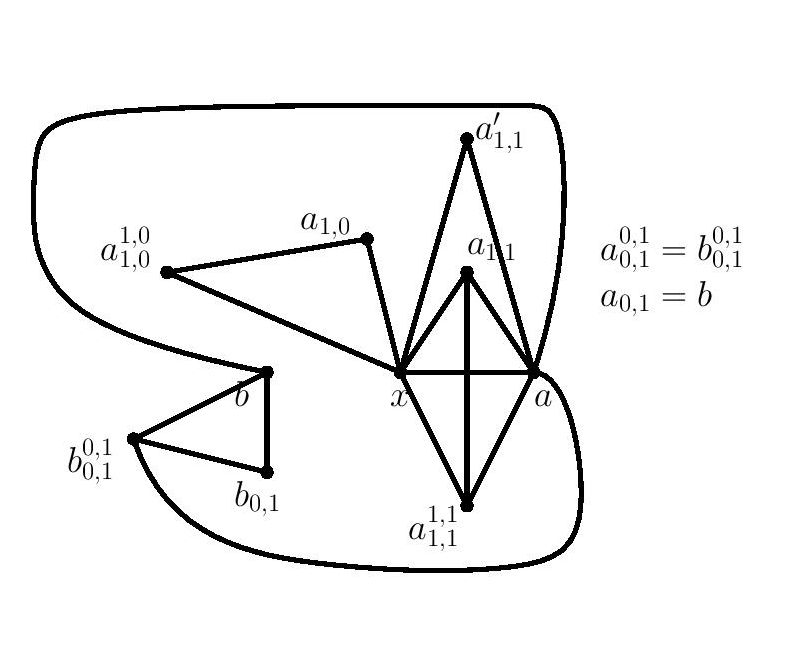}
\end{center}
  \caption{}
\label{H_{0}_subcase_4}
\end{figure} 


\subsubsection{When $b_{1,1} \in V\left(H'_{0}\right) \setminus V(H_{0})$:} We add at least the edge $\{v_{3}, a_{1,0}\}$ when $v_{3}$ is old, at least the edge $\{v_{1}, b_{0,1}\}$ when $v_{1}$ is old, and these are distinct from any other edge added due to old first-level vertices. When $v_{2}$ and $v_{4}$ are both old, they together bring a single edge and not more only when $v_{2} = a$ and $v_{4} = b_{1,1}$. Any $2$ coinciding banal first-level vertices bring at least $4$ edges, any $3$ at least $5$, and when all $4$ are banal and coincide, $6$ edges. Thus the first-level vertices contribute at least $2\nu_{1}+3$ edges (at least $2\nu_{1}+4$ when $\nu_{1} \geq 3$), making $\rho(H_{2})$ at least
\begin{equation}
\frac{16 + 2 + (2\nu_{1}+3) + 2\nu_{2} + e}{10 + 1 + \nu_{1} + \nu_{2}} = \frac{21 + 2(\nu_{1}+\nu_{2}) + e}{11 + (\nu_{1}+\nu_{2})} \geq \frac{21}{11} > \frac{13}{7}.
\end{equation}


\subsubsection{When $b_{1,1} = a$:} Once again, the edge $\{v_{3}, a_{1,0}\}$ added when $v_{3}$ is old and the edge $\{v_{1}, b_{0,1}\}$ added when $v_{1}$ is old are distinct from all other edges added due to old first-level vertices. When $v_{2}$ and $v_{4}$ are old, they together bring a single edge and not more only when $v_{2} = v_{4} = v$ for some $v \in S_{2}$. If $v_{2}$ and $v_{4}$ are banal and coincide, the common vertex is adjacent to $x$, $b$, $a$; any other pair of coinciding banal first-level vertices bring at least $4$ edges; any $3$ coinciding banal first-level vertices which include both $v_{2}$ and $v_{4}$ bring $4$ edges, else at least $5$ edges; if all $4$ are banal and coincide, they bring $5$ edges. The first-level vertices thus bring at least $2\nu_{1}+3$ edges when $\nu_{1} \leq 3$, and $12$ edges when $\nu_{1} = 4$, making

\begin{equation}
\rho(H_{2}) \geq \frac{16 + (2\nu_{1}+3) + 2\nu_{2} + e}{10 + \nu_{1} + \nu_{2}} = \frac{19 + 2(\nu_{1}+\nu_{2}) + e}{10 + (\nu_{1} + \nu_{2})} \geq \frac{19}{10} > \frac{13}{7}.
\end{equation}

\subsubsection{When $b_{1,1} = a_{1,0}$:} If $v_{2}$, $v_{3}$, $v_{4}$ are old, they bring $2$ edges and no more only if $v_{2} = v_{3} = v_{4} = v$ for some $v \in S_{2}$; we add at least the edge $\{v_{1}, b_{0,1}\}$ when $v_{1}$ is old. If $v_{2}$, $v_{3}$, $v_{4}$ are banal and coincide, the common vertex is adjacent to $x$, $b$, $a_{1,0}$, $a$. Any other triplet of coinciding banal first-level vertices bring at least $5$ edges; any pair of coinciding banal first-level vertices bring at least $4$ edges; when all $4$ are banal and coincide, we get $5$ edges. Thus the first-level vertices bring at least $2\nu_{1}+3$ edges when $\nu_{1} \leq 2$ and at least $2\nu_{1}+4$ edges when $\nu_{1} \geq 3$, making

\begin{equation}
\rho(H_{2}) \geq \frac{16 + 1 + (2\nu_{1}+3) + 2\nu_{2} + e}{10 + \nu_{1} + \nu_{2}} = \frac{20 + 2(\nu_{1}+\nu_{2}) + e}{10 + (\nu_{1}+\nu_{2})} \geq 2 > \frac{13}{7}.
\end{equation}

\subsubsection{When $b_{1,1} = a_{1,0}^{1,0}$:} If $v_{2}$ is old and equals $a_{1,0}$, and $v_{4}$ is old, note that the edges $\{a_{1,0}, b\}$ and $\{v_{4}, b\}$ cannot coincide as $v_{4} \in S_{2}$; similarly, if $v_{3}$ is old, $\{v_{3}, a_{1,0}\}$ does not coincide with $\{a_{1,0}, b\}$ as $v_{3} \in S_{2}$. The first-level vertices bring at least $2\nu_{1}+4$ edges, making 
\begin{equation}
\rho(H_{2}) \geq \frac{16 + 1 + (2\nu_{1}+4) + 2\nu_{2} + e}{10 + \nu_{1} + \nu_{2}} = \frac{21 + 2(\nu_{1}+\nu_{2}) + e}{10 + (\nu_{1}+\nu_{2})} > 2 > \frac{13}{7}.
\end{equation}

\subsubsection{When $b_{1,1} \in S_{2}$:} The first-level vertices bring at least $2\nu_{1}+2$ edges, hence
\begin{equation}
\rho(H_{2}) \geq \frac{16 +1 + 2\nu_{1}+2 + 2\nu_{2} + e}{10 + \nu_{1} + \nu_{2}} = \frac{19 + 2(\nu_{1}+\nu_{2}) + e}{10 + (\nu_{1}+\nu_{2})} \geq \frac{19}{10} > \frac{13}{7}.
\end{equation}

\subsection{Fifth case:}\label{subcase_5}
Here we assume that $a_{0,1} = b_{0,1}$ and add $a_{0,1}^{0,1}$ as a new vertex, along with edges $\{a, b_{0,1}\}$, $\left\{b_{0,1}, a_{0,1}^{0,1}\right\}$ and $\left\{a_{0,1}^{0,1}, a\right\}$ to get $H_{0}$ in Figure~\ref{H_{0}_subcase_5} with $17$ edges and $11$ vertices.
\begin{figure}[h!]
  \begin{center}
   \includegraphics[width=0.5\textwidth]{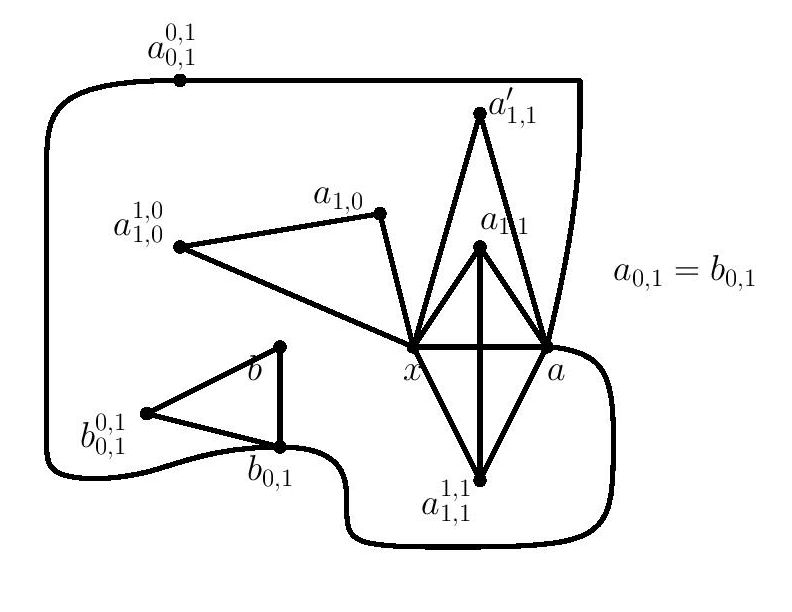}
\end{center}
  \caption{}
\label{H_{0}_subcase_5}
\end{figure}

\subsubsection{When $b_{1,1} \in V\left(H'_{0}\right) \setminus V(H_{0})$:}\label{new_subcase_5} When $v_{3}$ is old, it brings at least the edge $\{v_{3}, a_{1,0}\}$, which coincides with no other edge added due to old first-level vertices except possibly $\{v_{2}, b_{1,1}\}$ if $v_{2} = a_{1,0}$ and $v_{3} = b_{1,1}$, but then we additionally get the edges $\{a_{1,0}, b\}$ and $\{a, b_{1,1}\}$. When $v_{1}$ and $v_{2}$ are old, we get edges $\{a, b\}$ and $\{a, b_{1,1}\}$ when $v_{1} = v_{2} = a$, and at least $3$ edges otherwise; when $v_{1}$ and $v_{4}$ are old, we get edges $\{v, b\}$ and $\{v, b_{0,1}\}$ when $v_{1} = v_{4} = v$ for some $v \in S_{2}$, or $\{v_{4}, b_{0,1}\}$ and $\{a, b\}$ when $v_{4} \in S_{2}$ and $v_{1} = a$, or $\{a, b_{1,1}\}$ and $\{b_{1,1}, b_{0,1}\}$ when $v_{1} = v_{4} = b_{1,1}$, and at least $3$ edges otherwise. All other pairs of old first-level vertices bring at least $2$ edges. When $v_{1}$, $v_{2}$, $v_{3}$ are old, we get edges $\{a, b\}$, $\{a, b_{1,1}\}$ and $\{b_{1,1}, a_{1,0}\}$ when $v_{1} = v_{2} = a$ and $v_{3} = b_{1,1}$, or $\{a, b\}$, $\{a, b_{1,1}\}$ and $\{v_{3}, a_{1,0}\}$ when $v_{1} = v_{2} = a$ and $v_{3} \in S_{2}$, and at least $3$ edges otherwise. When $v_{1}$, $v_{2}$, $v_{4}$ are old, we get edges $\{a, b\}$, $\{a, b_{1,1}\}$ and $\{v_{4}, b_{0,1}\}$ when $v_{1} = v_{2} = a$ and $v_{4} \in S_{2}$, or $\{v, b\}$, $\{v, b_{0,1}\}$ and $\{v, b_{1,1}\}$ when $v_{1} = v_{2} = v_{4} = v$ for some $v \in S_{2}$, or $\{a, b\}$, $\{a, b_{1,1}\}$ and $\{b_{1,1}, b_{0,1}\}$ when $v_{1} = v_{2} = a$ and $v_{4} = b_{1,1}$, and at least $4$ edges otherwise. When $v_{1}$, $v_{3}$, $v_{4}$ are old, we get edges $\{b_{1,1}, b_{0,1}\}$, $\{a, b_{1,1}\}$ and $\{b_{1,1}, a_{1,0}\}$ when $v_{1} = v_{3} = v_{4} = b_{1,1}$, or $\{b_{1,1}, b_{0,1}\}$, $\{a, b_{1,1}\}$ and $\{v_{3}, a_{1,0}\}$ when $v_{1} = v_{4} = b_{1,1}$ and $v_{3} \in S_{2}$, $\{v, b\}$, $\{v, b_{0,1}\}$ and $\{v_{3}, a_{1,0}\}$ when $v_{1} = v_{4} = v$ for some $v \in S_{2}$ and $v_{3} \in S_{2}$, or $\{v_{4}, b_{0,1}\}$, $\{v_{3}, a_{1,0}\}$ and either $\{b_{1,1}, b_{0,1}\}$ or $\{a, b\}$ when $v_{4} \in S_{2}$, $v_{3} \in S_{2}$ and $v_{1}$ is either $b_{1,1}$ or $a$. When $v_{2}$, $v_{3}$, $v_{4}$ are old, they bring at least $4$ edges. When all $4$ are old, they bring at least $4$ edges. 

Any $2$ coinciding banal first-level vertices bring at least $4$ edges, any $3$ bring at least $5$ edges, and when all $4$ are banal and coincide, they bring $6$ edges. Thus the first-level vertices bring at least $2\nu_{1}+4$ edges, so that
\begin{equation}
\rho(H_{2}) \geq \frac{17 + 2 + (2\nu_{1}+4) + 2\nu_{2} + e}{11 + 1 + \nu_{1} + \nu_{2}} = \frac{23 + 2(\nu_{1}+\nu_{2}) + e}{12 + (\nu_{1}+\nu_{2})} \geq \frac{23}{13} > \frac{13}{7}.
\end{equation}

\subsubsection{When $b_{1,1} = a_{1,1}$:}\label{a_{1,1}_subcase_5} When $v_{1}$ and $v_{2}$ are old, we get edge $\{a, b\}$ if $v_{1} = v_{2} = a$, and at least $2$ edges otherwise; when $v_{1}$ and $v_{4}$ are old, we get $\{a_{1,1}, b_{0,1}\}$ if $v_{1} = v_{4} = a_{1,1}$, and at least $2$ edges otherwise. Any other pair of old first-level vertices bring at least $2$ edges. When $v_{1}$, $v_{2}$, $v_{3}$ are old, they bring edges $\{a, b\}$ and $\{v_{3}, a_{1,0}\}$ when $v_{1} = v_{2} = a$, and at least $3$ edges otherwise. When $v_{1}$, $v_{2}$, $v_{4}$ are old, they bring edges $\{a, b\}$ and $\{v_{4}, b_{0,1}\}$ when $v_{1} = v_{2} = a$, edges $\left\{a_{1,1}^{1,1}, b\right\}$ and $\left\{a_{1,1}^{1,1}, b_{0,1}\right\}$ when $v_{1} = v_{2} = v_{4} = a_{1,1}^{1,1}$, edges $\{a_{1,1}, b_{0,1}\}$ and $\{v_{2}, b\}$ when $v_{1} = v_{4} = a_{1,1}$ and $v_{2} \in \left\{a, a_{1,1}^{1,1}\right\}$, and at least $3$ edges otherwise. When $v_{1}$, $v_{3}$, $v_{4}$ are old, they bring edges $\{a_{1,1}, b_{0,1}\}$ and $\{v_{3}, a_{1,0}\}$ if $v_{1} = v_{4} = a_{1,1}$, and at least $3$ edges otherwise. When $v_{2}, v_{3}, v_{4}$ are old, we get at least $3$ edges. When all $4$ are old, we get at least $3$ edges. Contributions from banal first-level vertices is the same as in \S~\ref{new_subcase_5}. The first-level vertices thus contribute at least $2\nu_{1}+3$ edges when $\nu_{1} \leq 2$ and at least $2\nu_{1}+4$ edges when $\nu_{1} \geq 3$, making
\begin{equation}
\rho(H_{2}) \geq \frac{17 + 1 + (2\nu_{1}+3) + 2\nu_{2} + e}{11 + \nu_{1} + \nu_{2}} = \frac{21 + 2(\nu_{1} + \nu_{2}) + e}{11 + \nu_{1} + \nu_{2}} \geq \frac{21}{11} > \frac{13}{7}.
\end{equation}

\subsubsection{When $b_{1,1} \in \left\{ a_{1,1}^{1,1}, a'_{1,1}\right\}$:}\label{a_{1,1}^{1,1}_subcase_5} The analysis is \emph{mutatis mutandis} to \S~\ref{a_{1,1}_subcase_5}.


\subsubsection{When $b_{1,1} = a_{1,0}$:}\label{a_{1,0}_subcase_5} Any $2$ old first-level vertices bring at least $2$ edges, any $3$ bring at least $3$ edges; when all $4$ are old, they bring edges $\{v, a_{1,0}\}$, $\{v, b_{0,1}\}$ and $\{v, b\}$ only if $v_{1} = v_{2} = v_{3} = v_{4} = v$ for some $v \in S_{2}$, and at least $4$ edges otherwise. Any $2$ coinciding banal first-level vertices bring at least $4$ edges, any $3$ bring at least $5$ edges, and when all $4$ are banal and coincide, the common vertex is adjacent to $x$, $b$, $b_{0,1}$, $a_{1,0}$, $a$. The first-level vertices cotntribute at least $2\nu_{1}+3$ edges when $\nu_{1} \leq 1$, and at least $2\nu_{1}+4$ edges when $\nu_{1} \geq 2$, so that 
\begin{equation}
\rho(H_{2}) \geq \frac{17 + 1 + (2\nu_{1}+3) + 2\nu_{2} + e}{11 + \nu_{1} + \nu_{2}} = \frac{21 + 2(\nu_{1}+\nu_{2}) + e}{11 + \nu_{1} + \nu_{2}} \geq \frac{21}{11} > \frac{13}{7}.
\end{equation}

\subsubsection{When $b_{1,1} = a_{1,0}^{1,0}$:}\label{a_{1,0}^{1,0}_subcase_5} The first-level vertices bring at least $2\nu_{1}+4$ edges, so that
\begin{equation}
\rho(H_{2}) \geq \frac{17 + 1 + (2\nu_{1}+4) + 2\nu_{2} + e}{11 + \nu_{1} + \nu_{2}} = \frac{22 + 2(\nu_{1} + \nu_{2}) + e}{11 + (\nu_{1} + \nu_{2})} \geq 2 > \frac{13}{7}.
\end{equation} 

\subsubsection{When $b_{1,1} = a$:} The first-level vertices bring at least $2\nu_{1}+3$ edges, so that
\begin{equation}
\rho(H_{2}) \geq \frac{17 + 1 + (2\nu_{1}+3) + 2\nu_{2} + e}{11 + \nu_{1} + \nu_{2}} = \frac{21 + 2(\nu_{1} + \nu_{2}) + e}{11 + \nu_{1}+\nu_{2}} \geq \frac{21}{11} > \frac{13}{7}.
\end{equation}

\subsection{Sixth case:}\label{subcase_6}
We assume here that $a_{0,1} = b_{0,1}$ and $a_{0,1}^{0,1} = b$, and add $\{a, b\}$ and $\{a, b_{0,1}\}$, to get $H_{0}$ in Figure~\ref{H_{0}_subcase_6} with $16$ edges and $10$ vertices.
\begin{figure}[h!]
  \begin{center}
   \includegraphics[width=0.5\textwidth]{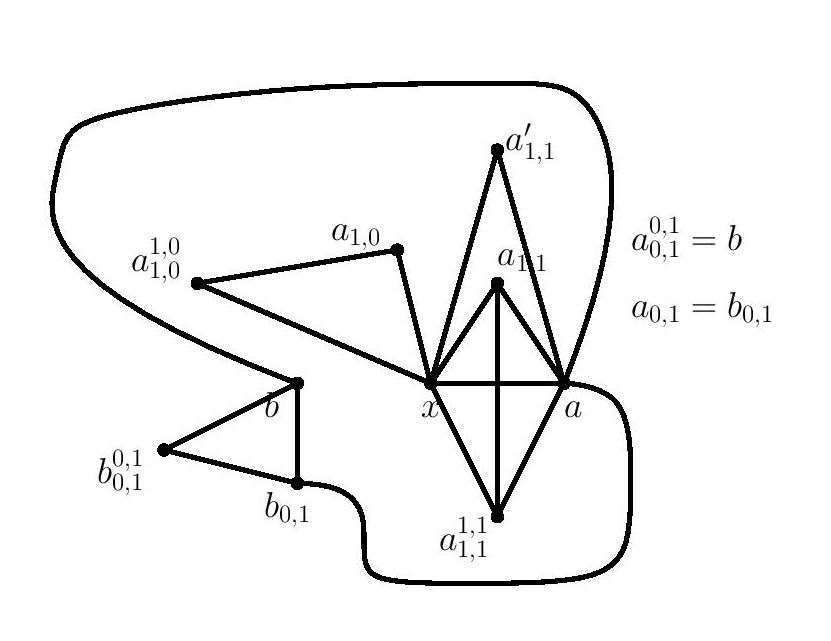}
\end{center}
  \caption{}
\label{H_{0}_subcase_6}
\end{figure} 

\subsubsection{When $b_{1,1} \in V\left(H'_{0}\right) \setminus V(H_{0})$:} We add edge $\{v_{3}, a_{1,0}\}$ when $v_{3}$ is old, edge $\{v_{4}, b_{0,1}\}$ when $v_{4}$ is old, edge $\{v_{2}, b_{1,1}\}$ when $v_{2}$ is old; edges $\{v_{3}, a_{1,0}\}$ and $\{v_{2}, b_{1,1}\}$ coincide only if $v_{2} = a_{1,0}$ and $v_{3} = b_{1,1}$, but then we additionally get edges $\{a_{1,0}, b\}$ and $\{a, b_{1,1}\}$. Any $2$ coinciding banal first-level vertices bring at least $4$ edges, when all $3$ are banal and coincide, they bring $5$ edges. The first-level vertices thus contribute at least $2\nu_{1}+3$ edges, so that
\begin{equation}
\rho(H_{2}) \geq \frac{16 + 2 + (2\nu_{1}+3) + 2\nu_{2} + e}{10 + 1 + \nu_{1} + \nu_{2}} = \frac{21 + 2(\nu_{1}+\nu_{2}) + e}{11 + (\nu_{1}+\nu_{2})} \geq \frac{21}{11} > \frac{13}{7}.
\end{equation}

\subsubsection{When $b_{1,1} = a$:}\label{subcase_6_b_{1,1}=a} We add edge $\{v_{3}, a_{1,0}\}$ when $v_{3}$ is old, edge $\{v_{4}, b_{0,1}\}$ when $v_{4}$ is old, edge $\{v_{2}, b\}$ when $v_{2}$ is old, and all these are distinct. Any $2$ coinciding banal first-level vertices bring at least $4$ edges, when all $3$ are banal and coincide, they bring $5$ edges. The first-level vertices contribute at least $2\nu_{1}+3$ edges, so that
\begin{equation}
\rho(H_{2}) \geq \frac{16 + 2\nu_{1}+3 + 2\nu_{2} + e}{10 + \nu_{1} + \nu_{2}} = \frac{19 + 2(\nu_{1} + \nu_{2}) + e}{10 + (\nu_{1} + \nu_{2})} \geq \frac{19}{10} > \frac{13}{7}.
\end{equation}

\subsubsection{When $b_{1,1} \in \left\{a_{1,0}, a_{1,0}^{1,0}\right\}$:} By the same argument as in \S~\ref{subcase_6_b_{1,1}=a}, the first-level vertices contribute at least $2\nu_{1}+3$ edges, so that
\begin{equation}
\rho(H_{2}) \geq \frac{16 + 1 + (2\nu_{1}+3) + 2\nu_{2} + e}{10 + \nu_{1} + \nu_{2}} = \frac{20 + 2(\nu_{1} + \nu_{2}) + e}{10 + (\nu_{1} + \nu_{2})} \geq 2 > \frac{13}{7}.
\end{equation}

\subsubsection{When $b_{1,1} \in S_{2}$:} The first-level vertices $v_{3}$ and $v_{4}$ bring at least $2\nu_{1}+2$ edges, hence
\begin{equation}
\rho(H_{2}) \geq \frac{16 + 1 + 2\nu_{1}+2 + 2\nu_{2} + e}{10 + \nu_{1} + \nu_{2}} = \frac{19 + 2(\nu_{1} + \nu_{2}) + e}{10 + (\nu_{1}+\nu_{2})} \geq \frac{19}{10} > \frac{13}{7}.
\end{equation}

\subsection{Seventh case:}\label{subcase_7}
We assume that $a_{0,1} = b_{0,1}$ and $a_{0,1}^{0,1} = b_{0,1}^{0,1}$, and we add $\{a, b_{0,1}\}$ and $\left\{a, b_{0,1}^{0,1}\right\}$ to get $H_{0}$ in Figure~\ref{H_{0}_subcase_7} with $16$ edges and $10$ vertices.
\begin{figure}[h!]
  \begin{center}
   \includegraphics[width=0.5\textwidth]{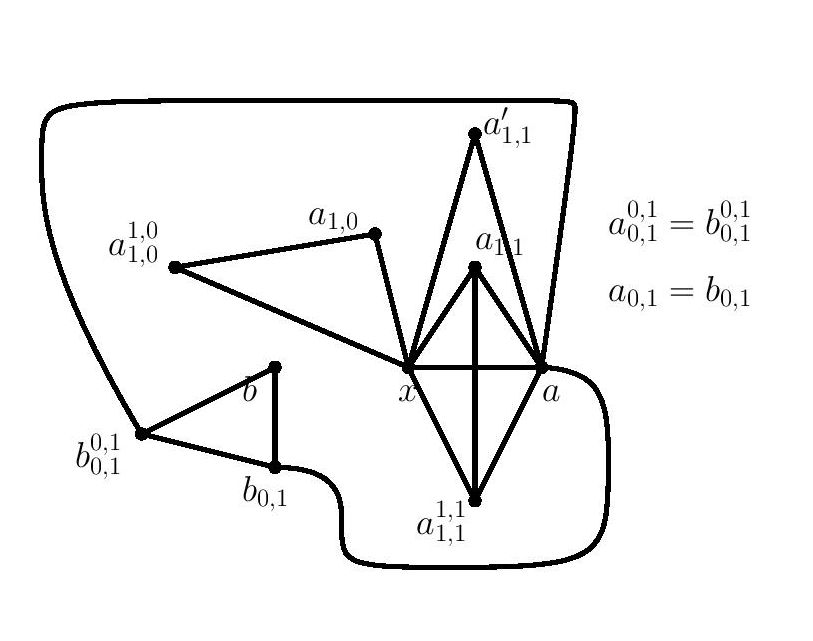}
\end{center}
  \caption{}
\label{H_{0}_subcase_7}
\end{figure} 

\subsubsection{When $b_{1,1} \in V\left(H'_{0}\right) \setminus V(H_{0})$:}\label{new_subcase_7} When $v_{3}$ is old, we add $\{v_{3}, a_{1,0}\}$; when $v_{4}$ is old, we add $\{v_{4}, b_{0,1}\}$; when $v_{2}$ is old, we add $\{v_{2}, b\}$ and $\{v_{2}, b_{1,1}\}$. If $v_{1} = a$, we add edge $\{a, b\}$, which can coincide with $\{v_{2}, b\}$ only if $v_{2} = a$. If $v_{1} \in S_{1} \setminus \{a, b_{1,1}\}$, we add edges $\{v_{1}, b\}$ and $\{v_{1}, b_{0,1}\}$ -- even if $v_{1} = v_{2} = v_{4} = v$ for some $v \in S_{2}$, the edges are $\{v, b_{0,1}\}$, $\{v, b\}$, $\{v, b_{1,1}\}$. Any $2$ coinciding banal first-level vertices bring at least $4$ edges; any $3$ bring at least $5$ edges; if all $4$ are banal coincide, they bring $6$ edges. The first-level vertices contribute at least $2\nu_{1}+4$ edges, so that
\begin{equation}
\rho(H_{2}) \geq \frac{16 + 2 + (2\nu_{1} + 4) + 2\nu_{2} + e}{10 + 1 + \nu_{1} + \nu_{2}} = \frac{22 + 2(\nu_{1} + \nu_{2}) + e }{11 + (\nu_{1} + \nu_{2})} \geq 2 > \frac{13}{7}.
\end{equation}

\subsubsection{When $b_{1,1} = a_{1,1}$:}\label{a_{1,1}_subcase_7} The only scenario where $2$ old first-level vertices together contribute a single edge is when $v_{2} = v_{1} = a$. Any $2$ coinciding banal first-level vertices bring at least $4$ edges; any $3$ bring at least $5$ edges; if all $4$ are banal coincide, they bring $6$ edges. The first-level vertices contribute at least $2\nu_{1}+3$ edges when $\nu_{1} \leq 2$, and at least $2\nu_{1}+4$ when $\nu_{1} \geq 3$, so that

\begin{equation}
\rho(H_{2}) \geq \frac{16 + 1 + 2\nu_{1} + 3 + 2\nu_{2} + e}{10 + \nu_{1} + \nu_{2}} = \frac{20 + 2(\nu_{1} + \nu_{2}) + e }{10 + (\nu_{1} + \nu_{2})} \geq 2 > \frac{13}{7}.
\end{equation}

\subsubsection{When $b_{1,1} \in \left\{a_{1,1}^{1,1}, a'_{1,1}\right\}$:}\label{a_{1,1}^{1,1}_subcase_7} The analysis of \S~\ref{a_{1,1}_subcase_7} applies \emph{mutatis mutandis} here.


\subsubsection{When $b_{1,1} = a_{1,0}$:}\label{a_{1,0}_subcase_7} Any $2$ old first-level vertices bring at least $2$ edges, any $3$ bring at least $3$ edges, and when all $4$ are old, they bring $3$ edges and no more only if $v_{1} = v_{2} = v_{3} = v_{4} = v$ for some $v \in S_{2}$. Any $2$ coinciding banal first-level vertices bring at least $4$ edges, any $3$ bring at least $5$, and if all $4$ are banal and coincide, the common vertex is adjacent to $x$, $a$, $b$, $b_{0,1}$ and $a_{1,0}$. Thus first-level vertices contribute at least $2\nu_{1}+3$ edges when $\nu_{1} \leq 1$ and at least $2\nu_{1}+4$ edges when $\nu_{1} \geq 2$. The remaining analysis is the same as in \S~\ref{a_{1,1}_subcase_7}.

\subsubsection{When $b_{1,1} = a_{1,0}^{1,0}$:}\label{a_{1,0}^{1,0}_subcase_7} The first-level vertices contribute at least $2\nu_{1}+4$ edges, so that
\begin{equation}
\rho(H_{2}) \geq \frac{16 + 1 + (2\nu_{1} + 4) + 2\nu_{2} + e}{10 + \nu_{1} + \nu_{2}} = \frac{21 + 2(\nu_{1}+\nu_{2}) + e}{10 + (\nu_{1}+\nu_{2})} > 2 > \frac{13}{7}.
\end{equation}

\subsubsection{When $b_{1,1} = a$:}\label{a_subcase_7} The first-level vertices bring at least $2\nu_{1}+3$ edges, and the analysis is the same as in \S~\ref{a_{1,1}_subcase_7}.

\subsection{Eighth case:}\label{subcase_8}
We assume that $a_{0,1}$ is a new vertex and $a_{0,1}^{0,1} = b$, and add edges $\{a, b\}$, $\{a, a_{0,1}\}$ and $\{a_{0,1}, b\}$ to get $H_{0}$ in Figure~\ref{H_{0}_subcase_8} with $17$ edges and $11$ vertices.
\begin{figure}[h!]
  \begin{center}
   \includegraphics[width=0.5\textwidth]{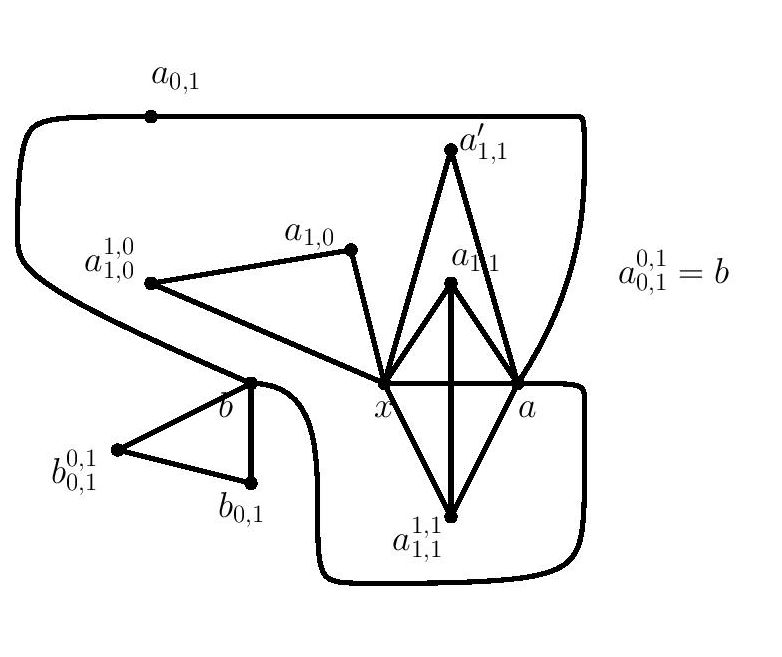}
\end{center}
  \caption{}
\label{H_{0}_subcase_8}
\end{figure} 

\subsubsection{When $b_{1,1} \in V\left(H'_{0}\right) \setminus V(H_{0})$:}\label{new_subcase_8} When $v_{1}$ is old, it brings $\{v_{1}, b_{0,1}\}$; when $v_{2}$ is old, it brings $\{v_{2}, b_{1,1}\}$; when $v_{3}$ is old, it brings $\{v_{3}, a_{1,0}\}$; when $v_{4}$ is old, it brings $\{v_{4}, a_{0,1}\}$. Of these, only $\{v_{3}, a_{1,0}\}$ and $\{v_{2}, b_{1,1}\}$ may coincide, which happens only if $v_{3} = b_{1,1}$ and $v_{2} = a_{1,0}$, but in that case we additionally get edges $\{a_{1,0}, b\}$ and $\{a, b_{1,1}\}$. Any $2$ coinciding banal first-level vertices bring at least $4$ edges, any $3$ at least $5$ edges, and all $4$ bring $7$ edges. The first-level vertices thus contribute at least $2\nu_{1}+4$ edges, so that
\begin{equation}
\rho(H_{2}) \geq \frac{17 + 2 + (2\nu_{1} + 4) + 2\nu_{2} + e}{11 + 1 + \nu_{1} + \nu_{2}} = \frac{23 + 2(\nu_{1}+\nu_{2}) + e}{12 + (\nu_{1}+\nu_{2})} \geq \frac{23}{12} > \frac{13}{7}.
\end{equation}

\subsubsection{When $b_{1,1} = a$:}\label{a_subcase_8} When $v_{1}$ is old, it brings $\{v_{1}, b_{0,1}\}$; when $v_{2}$ is old, it brings $\{v_{2}, b\}$; when $v_{3}$ is old, it brings $\{v_{3}, a_{1,0}\}$; when $v_{4}$ is old, it brings $\{v_{4}, a_{0,1}\}$ -- all these edges are distinct. Any $2$ coinciding banal first-level vertices bring at least $4$ edges, any $3$ at least $5$ edges, and all $4$ bring $6$ edges. The first-level vertices thus contribute at least $2\nu_{1}+4$ edges, so that
\begin{equation}
\rho(H_{2}) \geq \frac{17 + (2\nu_{1}+4) + 2\nu_{2} + e}{11 + \nu_{1} + \nu_{2}} = \frac{21 + 2(\nu_{1}+\nu_{2}) + e}{11 + (\nu_{1}+\nu_{2})} \geq \frac{21}{11} > \frac{13}{7}.
\end{equation}

\subsubsection{When $b_{1,1} \in \left\{a_{1,0}, a_{1,0}^{1,0}\right\}$:}\label{a_{1,0}_subcase_8} By the same argument as in \S~\ref{a_subcase_8}, the first-level vertices bring at least $2\nu_{1}+4$ edges, so that
\begin{equation}
\rho(H_{2}) \geq \frac{17 + 1+ (2\nu_{1}+4) + 2\nu_{2} + e}{11 + \nu_{1}+\nu_{2}} = \frac{22 + 2(\nu_{1} + \nu_{2}) + e}{11 + (\nu_{1} + \nu_{2})} \geq 2 > \frac{13}{7}.
\end{equation}


\subsubsection{When $b_{1,1} \in S_{2}$:} The first-level vertices contribute at least $2\nu_{1}+3$ edges, so that
\begin{equation}
\rho(H_{2}) \geq \frac{17 + 1 + (2\nu_{1} + 3) + 2\nu_{2} + e}{11 + \nu_{1} + \nu_{2}} = \frac{21 + 2(\nu_{1} + \nu_{2}) + e}{11 + (\nu_{1} + \nu_{2})} \geq \frac{21}{11} > \frac{13}{7}.
\end{equation}

\subsection{Ninth case:}\label{subcase_9} 
We assume that $a_{0,1}$ is added as a new vertex and $a_{0,1}^{0,1} = b_{0,1}$, and add edges $\{a_{0,1}, a\}$, $\{a_{0,1}, b_{0,1}\}$ and $\{a, b_{0,1}\}$, to get $H_{0}$ in Figure~\ref{H_{0}_subcase_9} with $17$ edges and $11$ vertices.
\begin{figure}[h!]
  \begin{center}
   \includegraphics[width=0.5\textwidth]{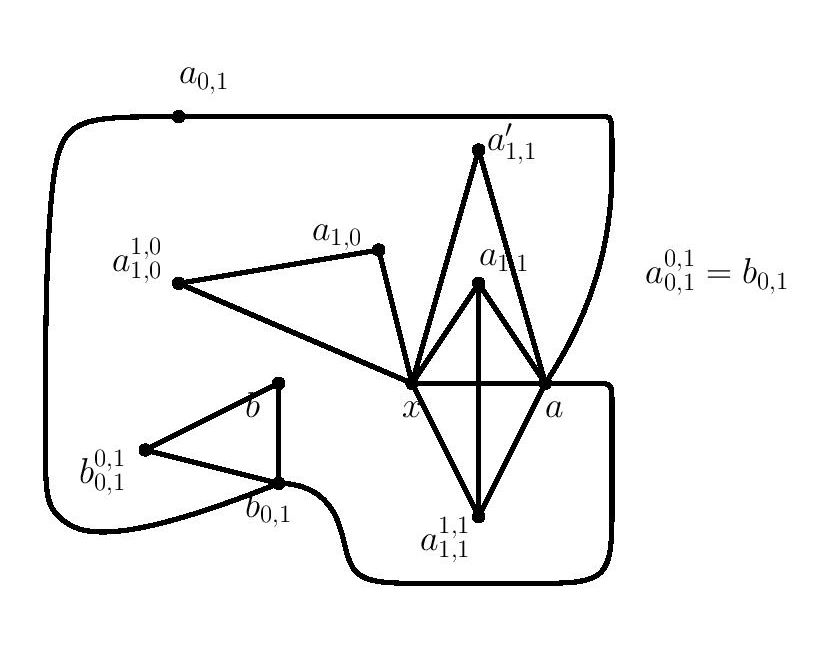}
\end{center}
  \caption{}
\label{H_{0}_subcase_9}
\end{figure}

\subsubsection{When $b_{1,1} \in V\left(H'_{0}\right) \setminus V(H_{0})$:}\label{new_subcase_9} When $v_{1} \neq a$ but old, it contributes edge $\{v_{1}, b_{0,1}\}$; when $v_{2}$ is old, we get $\{v_{2}, b\}$ and $\{v_{2}, b_{1,1}\}$; when $v_{3}$ is old, we get $\{v_{3}, a_{1,0}\}$; when $v_{4}$ is old, we get $\{v_{4}, a_{0,1}\}$ -- of these, only $\{v_{3}, a_{1,0}\}$ and $\{v_{2}, b_{1,1}\}$ may coincide, but this happens only when $v_{2} = a_{1,0}$ and $v_{3} = b_{1,1}$, which brings the additional edge $\{a, b_{1,1}\}$. If $v_{1} = a$, it brings edge $\{a, b\}$, which can coincide with $\{v_{2}, b\}$ only when $v_{2} = a$, but in this case we also get the edge $\{a, b_{1,1}\}$ -- it either coincides with no other edge added due to old first-level vertices (when $v_{3} \neq b_{1,1}$ and $v_{4} \neq b_{1,1}$), or else we additionally get at least one of the edges $\{b_{1,1}, a_{1,0}\}$ (when $v_{3} = b_{1,1}$) and $\{b_{1,1}, a_{0,1}\}$ (when $v_{4} = b_{1,1}$). Any $2$ coinciding banal first-level vertices bring at least $4$ edges, any $3$ at least $5$, and all $4$ bring $7$ edges. The first-level vertices thus contribute at least $2\nu_{1}+4$ edges, so that
\begin{equation}
\rho(H_{2}) \geq \frac{17 + 2 + (2\nu_{1}+4) + 2\nu_{2} + e}{11 + 1 + \nu_{1} + \nu_{2}} = \frac{23 + 2(\nu_{1} + \nu_{2}) + e}{12 + (\nu_{1} + \nu_{2})} \geq \frac{23}{12} > \frac{13}{7}.
\end{equation}

\subsubsection{When $b_{1,1} = a_{1,1}$:}\label{a_{1,1}_subcase_9} The only situation where $2$ old first-level vertices bring a single edge, and no more, is if $v_{1} = v_{2} = a$. When $v_{1}$ is old but $v_{1} \neq a$, it contributes edge $\{v_{1}, b_{0,1}\}$, when $v_{2}$ is old, it contributes $\{v_{2}, b\}$, when $v_{3}$ is old, it contributes $\{v_{3}, a_{1,0}\}$, and when $v_{4}$ is old, it contributes $\{v_{4}, a_{0,1}\}$, all of which are distinct. Any $2$ banal, coinciding first-level vertices bring at least $4$ edges, any $3$ bring at least $5$ edges, all $4$ banal and coinciding bring $7$ edges. The first-level vertices bring at least $2\nu_{1}+3$ edges when $\nu_{1} \leq 2$ and at least $2\nu_{1}+4$ edges when $\nu_{1} \geq 3$, so that

\begin{equation}
\rho(H_{2}) \geq \frac{17 + 1 + (2\nu_{1}+3) + 2\nu_{2} + e}{11 + \nu_{1} + \nu_{2}} = \frac{21 + 2(\nu_{1} + \nu_{2}) + e}{11 + (\nu_{1} + \nu_{2})} \geq \frac{21}{11} > \frac{13}{7}.
\end{equation}

\subsubsection{When $b_{1,1} \in \left\{a, a_{1,1}^{1,1}, a'_{1,1}\right\}$:}\label{a_{1,1}^{1,1}_subcase_9} The analysis of \S~\ref{a_{1,1}_subcase_9} applies \emph{mutatis mutandis} here.


\subsubsection{When $b_{1,1} \in \left\{a_{1,0}, a_{1,0}^{1,0}\right\}$:}\label{a_{1,0}_subcase_9} The first-level vertices contribute at least $2\nu_{1}+4$ edges, so that
\begin{equation}
\rho(H_{2}) \geq \frac{17 + 1 + (2\nu_{1} + 4) + 2\nu_{2} + e}{11 + \nu_{1} + \nu_{2}} = \frac{22 + 2(\nu_{1} + \nu_{2}) + e}{11 + (\nu_{1} + \nu_{2})} \geq 2 > \frac{13}{7}.
\end{equation}



\subsection{Tenth case:}\label{subcase_10}
We assume that $a_{0,1}$ is added as a new vertex and $a_{0,1}^{0,1} = b_{0,1}^{0,1}$, and add edges $\left\{b_{0,1}^{0,1}, a\right\}$, $\left\{b_{0,1}^{0,1}, a_{0,1}\right\}$ and $\{a_{0,1}, a\}$ to get $H_{0}$ in Figure~\ref{H_{0}_subcase_10} with $17$ edges and $11$ vertices.
\begin{figure}[h!]
  \begin{center}
   \includegraphics[width=0.5\textwidth]{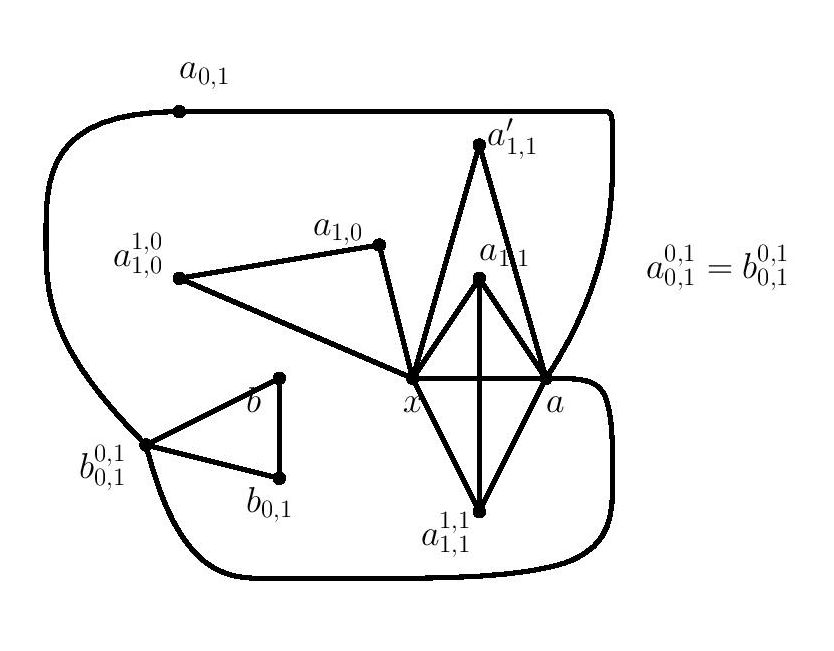}
\end{center}
  \caption{}
\label{H_{0}_subcase_10}
\end{figure}
When $v_{1}$ is old, it brings at least edge $\{v_{1}, b_{0,1}\}$, when $v_{2}$ is old, it at least brings edge $\{v_{2}, b\}$, when $v_{3}$ is old, it brings $\{v_{3}, a_{1,0}\}$, and when $v_{4}$ is old, it brings $\{v_{4}, a_{0,1}\}$, all of which are distinct. No matter what $b_{1,1}$ is, any $2$ coinciding banal first-level vertices bring at least $4$ edges, any $3$ bring at least $5$ edges, and all $4$ bring at least $6$ edges. The first-level vertices thus contribute at least $2\nu_{1}+4$ edges, so that when $b_{1,1} \in V\left(H'_{0}\right) \setminus V(H_{0})$, 
\begin{equation}
\rho(H_{2}) \geq \frac{17 + 2 + (2\nu_{1}+4) + 2\nu_{2} + e}{11 + 1 + \nu_{1} + \nu_{2}} = \frac{23 + 2(\nu_{1} + \nu_{2}) + e}{12 + (\nu_{1} + \nu_{2})} \geq \frac{23}{12} > \frac{13}{7},
\end{equation}
and when $b_{1,1} \in S_{1}$, 
\begin{equation}
\rho(H_{2}) \geq \frac{17 + 1 + (2\nu_{1} + 4) + 2\nu_{2} + e}{11 + \nu_{1} + \nu_{2}} = \frac{22 + 2(\nu_{1}+\nu_{2}) + e}{11 + (\nu_{1}+\nu_{2})} \geq 2 > \frac{13}{7}.
\end{equation}

\subsection{Eleventh case:}\label{subcase_11} 
We assume that $a_{0,1} = b_{0,1}^{0,1}$ and $a_{0,1}^{0,1}$ is added as a new vertex, and we add edges $\left\{a, b_{0,1}^{0,1}\right\}$, $\left\{a, a_{0,1}^{0,1}\right\}$ and $\left\{a_{0,1}^{0,1}, b_{0,1}^{0,1}\right\}$ to get $H_{0}$ in Figure~\ref{H_{0}_subcase_11} with $17$ edges and $11$ vertices.
\begin{figure}[h!]
  \begin{center}
   \includegraphics[width=0.5\textwidth]{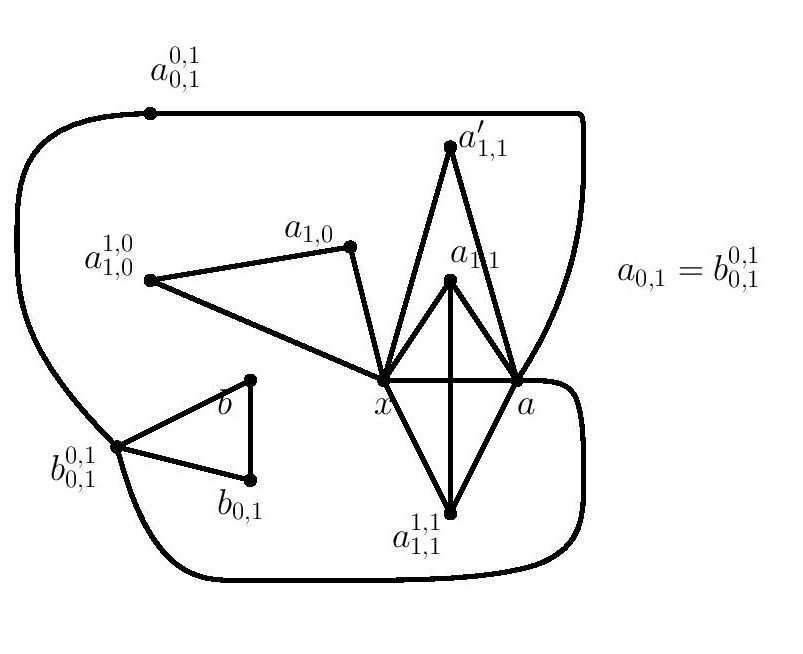}
\end{center}
  \caption{}
\label{H_{0}_subcase_11}
\end{figure}
The argument of \S~\ref{subcase_10} applies \emph{mutatis mutandis} here.


\subsection{Twelfth case:}\label{subcase_12}
We assume that $a_{0,1} = b_{0,1}^{0,1}$ and $a_{0,1}^{0,1} = b$, and add edges $\left\{a, b_{0,1}^{0,1}\right\}$ and $\left\{a, b\right\}$ to get $H_{0}$ in Figure~\ref{H_{0}_subcase_12} with $16$ edges and $10$ vertices.
\begin{figure}[h!]
  \begin{center}
   \includegraphics[width=0.5\textwidth]{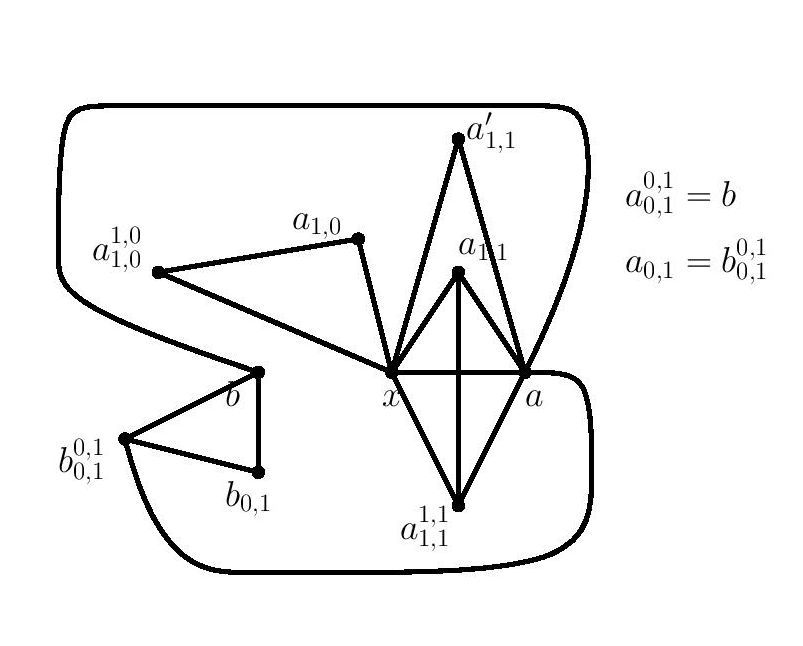}
\end{center}
  \caption{}
\label{H_{0}_subcase_12}
\end{figure}

\subsubsection{When $b_{1,1} \in V\left(H'_{0}\right) \setminus V(H_{0})$:}\label{new_subcase_12} When $v_{1}$ ins old, we get $\{v_{1}, b_{0,1}\}$; when $v_{4}$ is old, we get $\left\{v_{4}, b_{0,1}^{0,1}\right\}$; when $v_{2}$ is old we get $\{v_{2}, b_{1,1}\}$; when $v_{3}$ is old, we either get $\{v_{3}, a_{1,0}\}$ when $v_{3} \neq b_{1,1}$, or $\{b_{1,1}, a_{1,0}\}$ and $\{a, b_{1,1}\}$, only one of which can coincide with $\{v_{2}, b_{1,1}\}$. Any $2$ coinciding banal first-level vertices bring at least $4$ edges, any $3$ at least $5$, and all $4$ bring $7$ edges. The first-level vertices thus contribute at least $2\nu_{1}+4$ edges, so that
\begin{equation}
\rho(H_{2}) \geq \frac{16 + 2 + (2\nu_{1}+4) + 2\nu_{2} + e}{10 + 1 + \nu_{1} + \nu_{2}} = \frac{22 + 2(\nu_{1} + \nu_{2}) + e}{11 + (\nu_{1} + \nu_{2})} \geq 2 > \frac{13}{7}.
\end{equation}

\subsubsection{When $b_{1,1} = a$:}\label{a_subcase_12} When $v_{1}$ is old, we get $\{v_{1}, b_{0,1}\}$; when $v_{2}$ is old, we get $\{v_{2}, b\}$; when $v_{3}$ is old, we get $\{v_{3}, a_{1,0}\}$; when $v_{4}$ is old, we get $\left\{v_{4}, b_{0,1}^{0,1}\right\}$ -- all these edges are distinct. Any $2$ coinciding banal first-level vertices bring at least $4$ edges, any $3$ at least $5$, and all $4$ bring $6$ edges. The first-level vertices thus contribute at least $2\nu_{1}+4$ edges, so that
\begin{equation}
\rho(H_{2}) \geq \frac{16 + (2\nu_{1}+4) + 2\nu_{2} + e}{10 + \nu_{1} + \nu_{2}} = \frac{20 + 2(\nu_{1}+\nu_{2}) + e}{10 + (\nu_{1}+\nu_{2})} \geq 2 > \frac{13}{7}.
\end{equation}

\subsubsection{When $b_{1,1} \in \left\{a_{1,0}, a_{1,0}^{1,0}\right\}$:}\label{a_{1,0}_subcase_12} By the same argument as in \S~\ref{a_subcase_12}, the first-level vertices contribute at least $2\nu_{1}+4$ edges, so that
\begin{equation}
\rho(H_{2}) \geq \frac{16 + 1 + (2\nu_{1}+4) + 2\nu_{2} + e}{10 + \nu_{1} + \nu_{2}} = \frac{21 + 2(\nu_{1}+\nu_{2}) + e}{10 + (\nu_{1}+\nu_{2})} > 2 > \frac{13}{7}.
\end{equation}

\subsubsection{When $b_{1,1} \in S_{2}$:}\label{S_{2}_subcase_12} The first-level vertices contribute at least $2\nu_{1}+3$ edges, so that
\begin{equation}
\rho(H_{2}) \geq \frac{16 + 1 + (2\nu_{1}+3) + 2\nu_{2} + e}{10 + \nu_{1} + \nu_{2}} = \frac{20 + 2(\nu_{1}+\nu_{2}) + e}{10 + (\nu_{1}+\nu_{2})} \geq 2 > \frac{13}{7}.
\end{equation}

\subsection{Thirteenth case:}\label{subcase_13}
We assume that $a_{0,1} = b_{0,1}^{0,1}$ and $a_{0,1}^{0,1} = b_{0,1}$, and add the edges $\left\{a, b_{0,1}^{0,1}\right\}$ and $\left\{a, b_{0,1}\right\}$ to get $H_{0}$ in Figure~\ref{H_{0}_subcase_13} with $16$ edges and $10$ vertices.
\begin{figure}[h!]
  \begin{center}
   \includegraphics[width=0.5\textwidth]{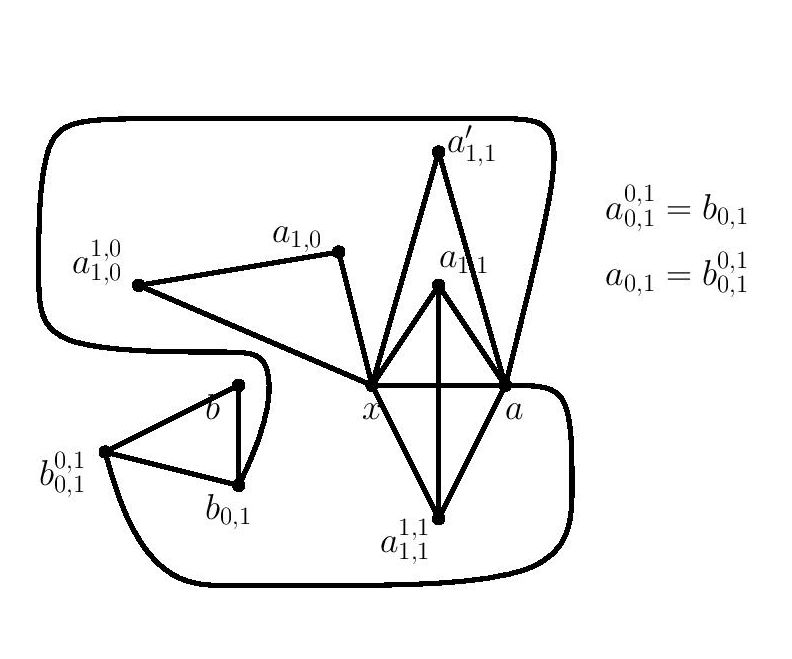}
\end{center}
  \caption{}
\label{H_{0}_subcase_13}
\end{figure}

\subsubsection{When $b_{1,1} \in V\left(H'_{0}\right) \setminus V(H_{0})$:}\label{new_subcase_13} When $v_{1}$ is old but $v_{1} \neq a$, we get the edge $\{v_{1}, b_{0,1}\}$; when $v_{4}$ is old, we get the edge $\left\{v_{4}, b_{0,1}^{0,1}\right\}$; when $v_{3}$ is old, it brings at least the edge $\{v_{3}, a_{1,0}\}$; when $v_{2}$ is old, it brings $\{v_{2}, b\}$ and $\{v_{2}, b_{1,1}\}$. The edges $\{v_{3}, a_{1,0}\}$ and $\{v_{2}, b_{1,1}\}$ can coincide only if $v_{3} = b_{1,1}$ and $v_{2} = a_{1,0}$, but then we also additionally get $\{a, b_{1,1}\}$. When $v_{1} = a$, we get the edge $\{a, b\}$, which can coincide with $\{v_{2}, b\}$ if $v_{2} = a$, but we also get the edge $\{a, b_{1,1}\}$ due to $v_{2}$ which either stays distinct from all other edges added due to old first-level vertices (when $v_{3} \neq b_{1,1}$ and $v_{4} \neq b_{1,1}$), or else is accompanied by at least one of $\{b_{1,1}, a_{1,0}\}$ and $\left\{b_{1,1}, b_{0,1}^{0,1}\right\}$. 

Any $2$ coinciding banal first-level vertices bring at least $4$ edges, any $3$ at least $5$, and all $4$ bring $7$ edges. The first-level vertices thus contribute at least $2\nu_{1}+4$ edges, so that
\begin{equation}
\rho(H_{2}) \geq \frac{16 + 2 + (2\nu_{1}+4) + 2\nu_{2} + e}{10 + 1 + \nu_{1} + \nu_{2}} = \frac{22 + 2(\nu_{1} + \nu_{2}) + e}{11 + (\nu_{1} + \nu_{2})} \geq 2 > \frac{13}{7}.
\end{equation}

\subsubsection{When $b_{1,1} \in S_{2}$:}\label{a_{1,1}_subcase_13} When $v_{3}$ is old, it adds edge $\{v_{3}, a_{1,0}\}$; when $v_{4}$ is old, it add edge $\left\{v_{4}, b_{0,1}^{0,1}\right\}$ -- both these edges are distinct from all other edges added due to old first-level vertices. The only situation where $2$ old first-level vertices together bring a single edge, and no more, is when $v_{1} = v_{2} = a$. Any $2$ coinciding banal first-level vertices bring at least $4$ edges, any $3$ at least $5$, and all $4$ bring $7$ edges. The first-level vertices thus contribute at least $2\nu_{1}+3$ edges when $\nu_{1} \leq 2$ and at least $2\nu_{1}+4$ when $\nu_{1} \geq 3$, so that

\begin{equation}
\rho(H_{2}) \geq \frac{16 + 1 + (2\nu_{1}+3) + 2\nu_{2} + e}{10 + \nu_{1}+\nu_{2}} = \frac{20 + 2(\nu_{1}+\nu_{2}) + e}{10 + (\nu_{1}+\nu_{2})} \geq 2 > \frac{13}{7}.
\end{equation}

\subsubsection{When $b_{1,1} \in \left\{a_{1,0}, a_{1,0}^{1,0}\right\}$:}\label{a_{1,0}_subcase_13} When $v_{1}$ is old, we get either edge $\{v_{1}, b_{0,1}\}$ or $\{a, b\}$; when $v_{2}$ is old, we get $\{v_{2}, b\}$ where $v_{2} \neq a$; when $v_{3}$ is old, we get $\{v_{3}, a_{1,0}\}$; when $v_{4}$ is old, we get $\left\{v_{4}, b_{0,1}^{0,1}\right\}$ -- all these edges are distinct. Any $2$ coinciding banal first-level vertices bring at least $4$ edges, any $3$ at least $5$, and all $4$ bring at least $6$ edges. The first-level vertices thus bring at least $2\nu_{1}+4$ edges, so that
\begin{equation} 
\rho(H_{2}) \geq \frac{16 + 1 + (2\nu_{1}+4) + 2\nu_{2} + e}{10 + \nu_{1} + \nu_{2}} = \frac{21 + 2(\nu_{1}+\nu_{2}) + e}{10 + (\nu_{1}+\nu_{2})} > 2 > \frac{13}{7}.
\end{equation}

\subsubsection{When $b_{1,1} = a$:}\label{a_subcase_13} The first-level vertices bring at least $2\nu_{1}+3$ edges, so that 
\begin{equation}
\frac{16 + 1 + (2\nu_{1}+3) + 2\nu_{2} + e}{10 + \nu_{1} + \nu_{2}} = \frac{20 + 2(\nu_{1} + \nu_{2}) + e}{10 + (\nu_{1} + \nu_{2})} \geq 2 > \frac{13}{7}.
\end{equation}

\end{document}